\documentclass[11pt,reqno]{amsart}
\usepackage{amssymb, bbm, color, stmaryrd}
\usepackage[makeroom]{cancel} 
\usepackage{amsfonts}
\usepackage{mathrsfs}
\usepackage{amsmath}
\usepackage{graphicx}
\usepackage{hyperref}
\usepackage{float}
\usepackage{epstopdf}
\usepackage{color}
\usepackage{bm}
\usepackage{comment}
\usepackage{soul}
\usepackage{esint}
\usepackage[shortlabels]{enumitem}

\allowdisplaybreaks
\newtheorem{theorem}{Theorem}[section]
\newtheorem{proposition}[theorem]{Proposition}
\newtheorem{lemma}[theorem]{Lemma}
\newtheorem{corollary}[theorem]{Corollary}
\newtheorem{remark}[theorem]{Remark}
 
\newtheorem{definition}[theorem]{Definition}

\def\diam{\mathrm{diam}}

 \def\Ex{\mathfrak E}

\def\mcB{\mathcal{B}}

\def\mcE{\mathcal{E}}
\def\mcF{\mathcal{F}}

\def\sB{\mathcal{B}}
\def\sE{\mathcal{E}}
\def\sF{\mathcal{F}}

\def\sH{\mathcal{H}}
\def\sX{\mathcal{X}}
\def\sN{\mathcal{N}}
\def\sL{\mathcal{L}}

\def\tD{D^*}

\def\R{{\mathbbm R}}
\def\H{{\mathbbm H}}
\def\N{{\mathbbm N}}
\def\Z{{\mathbbm Z}}

\def\bP{{\mathbbm P}}
\def\bE{{\mathbbm E}}

 \def\lb{{\llbracket}}
 \def\rb{{\rrbracket}}
 
 \def\1{{\mathbbm 1}}
 
 \makeatletter
\@namedef{subjclassname@2020}{ \textup{2020} Mathematics Subject Classification}
\makeatother

\def\eps{\varepsilon}

\def\rf{{\rm ref}}

\def\<{\langle}
\def\>{\rangle}

\numberwithin{equation}{section}

\begin{document}

 \title[Boundary trace theorems for symmetric reflected diffusions]{Boundary trace theorems for symmetric reflected diffusions}

\author{Shiping Cao}
\address{Department of Mathematics, The Chinese University of Hong Kong, Shatin, Hong Kong}
\email{spcao@math.cuhk.edu.hk}
\thanks{}

\author{Zhen-Qing Chen}
\address{Department of Mathematics, University of Washington, Seattle, WA 98195, USA}\email{zqchen@uw.edu}
\thanks{}


\date{}


 \begin{abstract}
 Starting  with a transient irreducible diffusion process $X^0$ on a locally compact separable metric space $(D, d)$,
one can construct a canonical symmetric reflected diffusion process $\bar X$ on a completion $D^*$ of $(D, d)$ through the theory of 
reflected Dirichlet spaces. The boundary trace process $\check X$ of $X$ on the boundary  $\partial D:=D^*\setminus D$ is the reflected diffusion process $\bar X$ time-changed by a smooth measure $\nu$ having  full  quasi-support on  $\partial D$.  The Dirichlet form of the trace process $\check X$ is called the trace Dirichlet form.  In the first part of  the paper, we give a Besov space type characterization of   the domain of  the trace Dirichlet form 
     for any good smooth measure   $\nu$ on the boundary $\partial D$.  In the second part of this paper, we study  properties of the harmonic measure 
     of $\bar X$  on the boundary $\partial D$.  In particular, we provide a condition equivalent to the doubling property of the harmonic measure. 
  Finally, we characterize and provide estimates of the jump kernel of the trace Dirichlet form under the doubling condition of the harmonic measure on
  $\partial D$. 
\end{abstract}

\maketitle

\section{Introduction}\label{sec1}

Consider a symmetric reflected Brownian motion $X$ on a smooth domain $D\subset \R^d$, which can be described by the 
following stochastic differential equation:
\begin{equation}\label{e:1.1} 
dX_t =dB_t +{\bf n}(X_t) dL_t,
\end{equation} 
where $X$ is a continuous process taking values in $\overline D$, 
$B$ is a standard Brownian motion in $\R^d$, ${\bf n}$ is the unit inward normal vector field on $\partial D$, and $L$
is a continuous  increasing process that increases only when $X$ is on the boundary $\partial D$.
The process $L$ is called the boundary local time of $X$. For $t\geq 0$, define
$\tau_t:=\inf\{s\geq 0: L_s>t\}$.  The time-changed process $\check X_t=X_{\tau_t}$ is called the
the boundary trace of $X$  on $\partial D$. Heuristically, it is a process obtained from $X$   by erasing all the excursions inside $D$.
The boundary trace process $\check X$ is a pure jump process on $\partial D$. 
It can have infinite lifetime, for instance, when $D$ is bounded or $D$ is an  half space in $\R^d$.
It can also have finite lifetime and  killings on  $\partial D$, 
for instance, when $D$ is an exterior open ball in $\R^d$ with $d\geq 3$.
When $D$ is the upper half space $\H^d$  in $\R^d$ with $d\geq 2$, it is easy to see that the trace process $Y$ is an isotropic Cauchy process
on $\partial \H =\R^{d-1}\times \{0\}$, whose infinitesimal generator is  the fractional Laplacian $-(-\Delta)^{1/2}$ which is also 
the Dirichlet-to-Neumann map on $\partial \H^d$. 
In fact, Molchanov and Ostrovskii \cite{MO} showed  that any isotropic $\alpha$-stable process on $\R^d$ 
with $\alpha \in (0, 2)$ is the boundary trace of some symmetric (possibly degenerate) reflected diffusion on $\H^{d+1}$. 
This fact has later been rediscovered analytically in \cite{CS} by Caffarelli and  Silvestre.  For additional examples  about  boundary traces processes of reflected diffusions in the upper half space $\H^d$, we refer the reader to \cite{CW,KMu} and the references therein. 

The purpose of this paper is to investigate  various properties of the 
boundary  trace processes of symmetric reflected diffusions in a  general context, not only on Euclidean
spaces but also on general metric measure spaces including fractals.  Suppose that  $(D, d)$ is  a locally compact 
separable metric space and $m$ is a Radon measure on $D$ with full support. 
Suppose $X^0$ is an $m$-symmetric  transient irreducible    continuous Hunt process on $D$ that has no killings inside $D$.
Without loss of  generality, we assume $X^0$ is irreducible in the sense of \cite{CF, FOT}.
It is known that the Dirichlet form    $(\sE^0, \sF^0)$ of $X^0$ is    strongly local and quasi-regular   on $L^2(D; m)$.
     Denote by $\sF^0_{\rm loc} $ the collection of functions $f$ 
on $D$ so that for any relatively compact subset $U$ of $D$, there is some $u\in \sF^0$ so that $u=f$ $m$-a.e. on $U$.
For each $f\in  \sF^0_{\rm loc} $, by the strong locality of $(\sE^0, \sF^0)$, the energy measure $\mu_{\<f\>}$ is well defined
so that  $\mu_{\<f\>}(U)=\mu_{\< u\>}(U)$ for the above $U$ and $u$. 
Define 
\begin{equation} \label{e:rf}
\sF^\rf :=\left\{ f\in \sF^0_{\rm loc}:  \mu_{\< f\>} (D)<\infty  \right\} 
\end{equation}
and 
\begin{equation*} 
\sE^\rf (f, f)=\frac12 \mu_{\<f\>}(D) \quad \hbox{for } f\in \sF^\rf.
\end{equation*}
It is shown in \cite{C1} that    $(\bar \sE, \bar \sF):= (\sE^\rf, \sF^\rf\cap L^2(D; m))$ is a Dirichlet form on $L^2(D; m)$, 
which is called an active reflected Dirichlet form on $D$. 
By \cite[Theorems 6.6.3 and 6.6.5]{CF}, there is a locally compact metric measure space $(\hat  D, \hat d)$
so that $ (\bar \sE, \bar \sF)$ is a regular Dirichlet form on $L^2(\hat D; m|_D)$ 
and $\hat D$  has $D$ as a quasi-open subset.  For simplicity, 
we assume that    $(\bar \sE, \bar \sF)$ is a regular Dirichlet form on $L^2(\tD; m_0)$, 
where $(\tD, d)$  is the completion of  $(D, d)$ and   $m_0(A):=m(A\cap D)$.
In particular,  $m_0$ is a Radon measure on $\tD$. 
 Since $(\sE^0, \sF^0)$ of $X^0$ is    strongly local,  
  the regular Dirichlet form $(\bar \sE, \bar \sF)$  is strongly local on  $L^2(\tD; m_0)$ in the sense of \cite[Definition 1.3.17]{CF};
  see \S \ref{S:2.2} for details. 
Thus by \cite[Theorems 1.5.1 and 4.3.4]{CF}, there is a   symmetric  continuous Hunt  process $\bar X =\{\bar X_t, t\geq 0; \bar \bP_x, x\in \tD \setminus \sN\}$   on $\tD$ associated with $(\bar \sE, \bar \sF)$    which admits no killings inside $\tD$,  
 where $\sN$ is a proper exceptional set for $\bar X$. 
We call  $\bar X$ the reflected diffusion on $\overline D$. 
By \cite[Theorem 6.6.5]{CF}, the subprocess of $\bar X$ killed upon leaving 
  $\partial D$ has the same distribution
as $X^0$.  When $D$ is a smooth domain in $\R^d$ and $X^0$ is the absorbing Brownian motion in $D$,   the diffusion process 
$\bar X$ constructed in this way is exactly the classical reflected Brownian motion on $\overline D$ in the sense of \eqref{e:1.1};
see, e.g.,   \cite{C2}. 

Denote by  $\bar \sF_e$   the extended Dirichlet space of $ (\bar \sE , \bar \sF)$; that is, $f\in  \bar \sF_e$ if and only if $f$ is finite $m$-a.e. on $D$
 and there exists an $ \bar \sE$-Cauchy sequence $\{f_n; n\geq 1\}\subset  \bar \sF$ such that $f_n\to f$ $m$-a.e. on $D$. 
 It is known (cf. \cite[Theorem 2.3.4]{CF}) that every $f\in  \bar \sF_e$ has an $  \bar \sE$-quasi-continuous $m$-version.
 We always represent functions in a regular Dirichlet form by their quasi-continuous versions.

 Suppose that $\partial D:=\tD \setminus D$ is of positive $\bar \sE$-capacity. 
There is always a smooth measure   with full $\bar \sE$-quasi-support on $\partial D$. 
 Indeed, take some strictly positive function $\varphi $ on $D$ with $\int_D \varphi (x ) m(dx)=1$. Note that $\partial D$ is closed and $m_0 (\partial D) =0$. It is shown in the proof of \cite[Lemma 5.2.9(i)]{CF}) that the weighted harmonic measure $\omega_0$ defined by
 \begin{equation}\label{e:1.2}
 \omega_0(dz)=\int_{D\setminus\sN}\omega_x(dz)\varphi(x)m(dx)  
 \end{equation} 
 is a finite smooth measure on $\partial D$ whose $\bar \sE$-quasi-support is $\partial D$,
 where $\omega_x$ is the harmonic measure of $D$ with pole at $x\in D \setminus \sN$ defined by
 $$
 \omega_x (A) = \bar \bP_x (\bar X_{\sigma_{\partial D}} \in A;  \sigma_{\partial D}  <\infty ) 
 \quad \hbox{for  every Borel set } A\subset \partial D.
 $$
 Here   $\sigma_{\partial D}:=\inf\{t>0: \bar X_t\in   \partial D\}$.  
If  elliptic Harnack  principle (EHP) holds for $X^0 $ in $D$ 
(see Definition \ref{D:4.5}(i) below),  
then all harmonic measures $\{\omega_x ; x\in D \}$ 
 $x\in D$, are   mutually equivalent. In this case, for each $x\in D$, $\omega_x$ is equivalent to $\omega_0$ 
 and hence is a  finite smooth measure on $\partial D$ having full   $\bar \sE$-quasi-support on $\partial D$.
  Let $\nu $ be a smooth measure with full $\bar \sE$-quasi-support on $\partial D$, and let  $A^\nu$ be the positive continuous additive functional 
  of $\bar X$ having $\nu$ as its Revuz measure.
 Define
 $$
\tau_t:= \inf\{ r>0: A^\nu_r >t\} \quad \hbox{for } t\geq 0. 
$$
Then the time-changed process $\check X_t:=\bar X_{\tau_t}$ is a $\nu$-symmetric strong Markov process taking values on $\partial D$, 
which is called a trace process of the reflected diffusion $\bar X$ on $\partial D$.  

Time changes of symmetric Markov processes have been studied in depth by  Silverstein \cite{S1, S2}.
As a special case of this general theory, 
the Dirichlet form $(\check{\sE},\check{\sF} )$ of the trace process $\check X$ is known to be regular on $L^2(\partial D; \nu)$ and 
is characterized by   (see  \cite[Theorem  5.2.2 and   Corollary 5.2.10]{CF}):  
\begin{eqnarray}
\check{\sF}&:=&\check \sF_e \cap L^2(\partial D;\nu), \label{e:trace1} \\
\check{\sE}(u,v)&:=& \bar \sE  (\mathcal{H}u,\mathcal{H}v)
\quad \hbox{ for each }u,v\in \check \sF_e,   \label{e:trace2} 
\end{eqnarray}
where  $\check \sF_e =\bar \sF_e|_{\partial D}$ and 
$$
\mathcal{H}u(x):=\bar \bE_x[u( \bar X_{\sigma_{\partial D}});\sigma_{\partial D}<\infty], \quad x \in \tD\setminus \sN. 
$$ 
Note that  $(\check \sE, \check  \sF_e )$ is independent of the smooth measure $\nu$ with full quasi-support on $\partial D$ that is used to do the time change. The above characterization  naturally gives   a restriction operator $ f\mapsto f|_{\partial D}$ from $ \bar \sF_e$  to $\check \sF_e$ and an extension operator $u \mapsto \sH u$ from $\check \sF_e$ to $ \bar \sF_e$. 
 
Recall that $(\bar \sE, \bar \sF)$ is a strongly local  regular Dirichlet form on $L^2(\tD; m_0)$. 
As a particular case of a general result established in
  \cite{CF2, CFY}, see also \cite[Theorems 5.5.9 and 5.6.2]{CF}, 
 the trace Dirichlet space $(\check \sE , \check \sF_e )$  admits the following  Beurling-Deny decomposition: 
   \begin{eqnarray*} 
	  		\check \sE(f,g) &=& \check{\sE}^{\rm (c)}(f,g)+ \frac12 \int_{\partial D\times \partial D \setminus \operatorname{diag}}\big(f(x)-f(y)\big)\big(g(x)-g(y)\big)\check J(dx, dy)
		\nonumber \\
		&&+\int_{\partial D}f(x)g(x)\check\kappa(dx) \qquad \hbox{for any } f, g\in \check \sF_e,
	 \end{eqnarray*}
 where $\check{\sE}^{\rm (c)}(f,g) =\frac12 \mu_{\<f, g\>}(\partial D)$, which vanishes by \eqref{e:em} below,  and $\check J$  and $\check\kappa$ are the Feller measure and supplementary Feller measure  for $\partial D$, respectively, defined in terms of the energy functional of the transient diffusion process $X^0$ on $D$ associated with $(\sE^0, \sF^0)$   as  in \cite[(5.5.7)]{CF}. That is, for any Borel subsets $A, B$ of $\partial D$,
 \begin{eqnarray}
  \check J(A, B) &=& L^0(\sH  \1_A ,  \sH \1_B ) :=\lim_{t\to 0} t^{-1} \int_D (\sH \1_A- P^0_t \sH \1_A) (x) \sH \1_B (x) m(dx),  \label{e:1.6}\\
 \check\kappa(A) &=& L^0(\sH \1_A, q) := \lim_{t\to 0} t^{-1} \int_D (\sH \1_A - P^0_t \sH \1_A) (x) q(x) m(dx),  \label{e:1.7}
 \end{eqnarray}
   where  $\{P^0_t; t\geq 0\}$ is the transition semigroup of the Hunt process $X^0$ and $q(x):=1-\sH 1 (x) =\bar \bP_x (\sigma_{\partial D} =\infty)$
   for $x\in D\setminus \sN$.

 The Feller measure $\check J(dx, dy)$ is a symmetric Radon measure on the product space $\partial D\times \partial D\setminus\operatorname{diag}$, where  
$$
\operatorname{diag} :=\{(x,x):\,x\in \partial D\},
$$
and the supplementary Feller measure $\check \kappa$ is a non-negative Radon measure on $\partial D$.   
 The Feller measure $\check J(dx, dy)$ and the supplementary Feller measure $\check \kappa(dx)$  can also be defined  in terms of  the excursions away from $\partial D$ of the reflected diffusion process $\bar X$ on $\tD$; see \cite[Theorem 5.7.6]{CF}. The measures $\check J$ and $\check\kappa$ are called the jumping measure and the killing measure, respectively,  
  of the trace Dirichlet space $(\check \sE , \check \sF_e )$.  
  In summary, the trace Dirichlet space $(\check \sE , \check \sF_e )$ on $\partial D$ is purely non-local which admits the 
  Beurling-Deny decomposition
   \begin{eqnarray}\label{e:2.7}
	  		\check \sE(f,g)  = \int_{\partial D\times \partial D\setminus \operatorname{diag} }\big(f(x)-f(y)\big)\big(g(x)-g(y)\big)\check J(dx, dy)
		 		 +\int_{\partial D}f(x)g(x)\check\kappa(dx)  
	 \end{eqnarray}
 for any $ f, g\in \check \sF_e$.

\begin{remark} \label{R:1.1} \rm
  (i) It follows from \eqref{e:1.6}-\eqref{e:1.7} and \cite[(5.5.14)]{CF}  that 
 \begin{equation} \label{e:Feller}
 \check J(dx, dy)\ll \omega_0 (dx)\omega_0 (dy) \ \hbox{ on } \ \partial D\times \partial D \setminus \hbox{diag} 
   \quad   \hbox{ and } \quad 
  \check \kappa (dx)\ll \omega_0 (dx) \ \hbox{ on } \partial D,
   \end{equation}
        where  $\omega_0$ is  a harmonic measure of $\bar X$ on $\partial D$ defined by \eqref{e:1.2}. 
     Indeed, for $A\in \sB (\partial D) $ having $\omega_0(A)=0$,  by definition we have   $\sH \1_A =0$ on $m$-a.e. on $D$.
   Hence  by \eqref{e:1.7},    $\check\kappa(A)= L^0(\sH \1_A, q)  =0$,
        This shows that $\check\kappa(dx)\ll \omega_0 (dx)$  on  $\partial D$.  For $\alpha >0$ and $x\in D\setminus \sN$,
   define the $\alpha$-order harmonic measure $\omega^{(\alpha)}_x$ on $\partial D$ by 
   $$
   \omega^{(\alpha)}_x (A):= \bar \bE_x \left[ e^{-\alpha \sigma_{\partial D}} \1_A (\bar X_{ \sigma_{\partial D}}); \sigma_{\partial D}<\infty ]\right]
   \quad \hbox{for } A\in \sB (\partial D).
   $$
   Clearly, $ \omega^{(\alpha)}_x \leq  \omega_x$ for each $x\in D\setminus \sN$. We know from \cite[(5.5.13)-(5.5.14)]{CF} that for any
   $F\in \sB (\partial D\times \partial D)$,
   \begin{equation}\label{e:1.9}
   \check J(F) = \,  \uparrow \lim_{\alpha \to \infty} \alpha \int_D (\omega^{(\alpha)}_x \otimes \omega_x ) (F) m(dx),
   \end{equation}
   where the notation $ \uparrow \lim_{\alpha \to \infty}$ means that it is an increasing limit as $\alpha $ increases to infinity.   
   For future reference, we record the following formula for the supplementary Feller measure $\check \kappa$ on $\partial D$,
   which follows from  \cite[(5.5.13)-(5.5.14)]{CF} in an analogous way.  For any $A\in \sB (\partial E)$,
   \begin{equation}\label{e:1.10}
   \check\kappa(A) = \, \uparrow \lim_{\alpha \to \infty} \alpha \int_D \omega^{(\alpha)}_x  (A) q(x) m(dx).
   \end{equation}
   For each $y\in \partial D$   and $F\in \sB (\partial D\times \partial D)$, denote by $F_y$ the $y$-section of $F$, that is,  $F_y :=\{z\in \partial D: (y, z)\in F\}$. 
   For   $F\in \sB (\partial D\times \partial D)$ having $(\omega_0 \otimes \omega_0) (F)=0$, 
   set $A:=\{y\in \partial D: \omega_0 (F_y)=0\}$. By  Fubini's theorem,
   $\omega_0 (\partial D\setminus A)=0$.  Thus 
   \begin{eqnarray*}
   \int_D (\omega^{(\alpha)}_x \otimes \omega_x ) (F) m(dx) 
   &=& \int_D \int_{\partial D} \omega_x (F_y) \omega_x^{(\alpha)} (dy) m(dx) \\
   &=& \int_D \int_{A} \omega_x (F_y) \omega_x^{(\alpha)} (dy) m(dx) =0  
      \end{eqnarray*}
   as for each $y\in A$, $\omega_x(F_y)=0$ for $m$-a.e. $x\in D$.  This proves that 
   $ \int_D (\omega^{(\alpha)}_x \otimes \omega_x )   m(dx) \ll \omega_0 \otimes \omega_0$ on $\partial D\times \partial D$.
   Consequently, by \eqref{e:1.9},  $\check J(dx, dy) \ll \omega_0 (dx) \omega_0( dy)$. This establishes the claim \eqref{e:Feller}.  
   
   We  point out that the above argument works for any symmetric Hunt process, not just for symmetric diffusions.
  That the trace Dirichlet space is of  jump type \eqref{e:2.7}, the formulas \eqref{e:1.9}-\eqref{e:1.10} for $\check J$ and $\check \kappa$,  and the absolute continuity property   \eqref{e:Feller} for the Feller measure and supplementary Feller measure    is a general fact, true  for any reflected process $\bar X$ of a symmetric Hunt process $X^0$ on $D$. As mentioned earlier, when  (EHP)  holds for $X^0$, which is weaker than (EHP)   for the reflected process $\bar X$,   all the harmonic measures $\{\omega_x; x\in D\}$ are mutually absolutely continuous and thus are equivalent to the weighted 
   harmonic measure $\omega_0$. In this case, we have from \eqref{e:Feller} that for any $z\in D$, 
    \begin{equation} \label{e:Feller2}
  \check J(dx, dy)\ll \omega_z (dx)\omega_z (dy) \ \hbox{ on } \ \partial D\times \partial D \setminus \hbox{diag} 
   \quad   \hbox{ and } \quad 
   \check\kappa(dx)\ll \omega_z (dx) \ \hbox{ on } \partial D,
   \end{equation}
    In this paper, we will not use the fundamental property \eqref{e:Feller} in our proofs. It only serves as a motivation for our investigation. 
   In Theorem \ref{thm71} below, under a heat kernel bound condition, we establish these absolute continuity property \eqref{e:Feller} 
   by another  way     together with a two-sided bounds on their Radon-Nikodym derivatives. 
    
    \medskip
    
         (ii) When $X^0$ is an absorbing Brownian motion in a bounded smooth domain $D$ in $\R^d$,   it is easy to see (cf. \cite[\S 7]{Doob})
    that  the harmonic measure $\omega_z$ with pole at $z\in D$ is given by 
    $$
    \omega_z(dy)= \frac{\partial G_D (z, \cdot) }{ \partial n_y}  (y) \sigma (dy) \quad \hbox{on } \partial D, 
    $$
    and 
    the jump kernel for the boundary trace of the reflected Brownian motion on $\overline D$ is given by
    $$ 
    \check J(dx, dy) =\frac{\partial^2 G_D  }{\partial n_x \partial n_y}   (x, y) \sigma (dx) \sigma (dy)
    \quad \hbox{on }   \partial D \times \partial D \setminus {\rm diag}, 
    $$
    where $G_D$ is the Green function for the Brownian motion in $D$, $\frac{\partial}{\partial n_y}$ is the directional derivative
    along the inward unit normal vector $n_y$ at $y\in \partial D$ and $\sigma (dy)$ the Lebeuge surface measure on $\partial D$.
    Fix some $z \in D$. 
    When $D$ is smooth, it is easy to show  that  for any $ x\not=y \in \partial D$, 
    \begin{equation} \label{e:1.11}
    \theta (x, y) =\lim_{x', y'\in D \atop x'\to x, y'\to y} \frac{G_D(x', y')}{G_D(z, x') G_D(z, y')}  \,  \hbox{ exists}
   \end{equation} 
    and that
    $$
    \theta (x, y) \omega _z (dx) \omega_z (dy)=\check J(dx, dy)= \frac{\partial^2 G_D  }{\partial n_x \partial n_y}   (x, y) \sigma (dx) \sigma (dy).
    $$
    When $D$ is a   connected open set in $\R^d$ so that $\R^d\setminus D$ is non-polar  (or, more generally, a Green space in the Brelot-Choquet terminology),
    Na\"im   \cite{Naim} showed that $ \frac{G_D(x, y)}{G_D(x_0, y)}$ extends continuously to $D\times \partial_M D$ as a Martin kernel, 
      and that $ \theta (x, y):=   \frac{G_D(x, y)}{G_D(z, x) G_D(z, y)}$ extends  to $\partial_M  D \times \partial_M D 
    \setminus {\rm diag}$ as a positive lower semicontinuous symmetric function.  
    Here   $\partial_M D$ is the Martin boundary of $D$, that is, $\partial_M D=\hat D \setminus D$ where  $\hat D$ is  the
    Martin compactification of $D$. 
    The  symmetric function $\theta (x, y)$ is nowadays called  Na\"im kernel.
      Doob   \cite[Theorem 9.2]{Doob}  showed that  the Dirichlet energy $\int_D | \nabla u  (x)|^2 dx$ of a harmonic function  in $D$ having finite energy
       can be represented
     as a Douglas integral over the Martin boundary $\partial_M D $ in terms of the Na\"im kernel $\theta$; that is,  in the  terminology of this paper, 
     \begin{equation}  \label{e:1.13}
       \theta (x, y) \omega_z (dx) \omega_z (dy)   = \check J(dx, dy)    \quad \hbox{on } \partial_M  D \times \partial_M D 
    \setminus {\rm diag}
     \end{equation}
     is the jump measure for the trace Dirichlet form $(\check \sE, \check \sF)$. 
       Fukushima     \cite{Fu} showed for one-dimensional Brownian motion, 
    the Na\"im kernel coincides with the Feller kernel, which is the Radon-Nikodym derivative
    of the Feller measure $\check J(dx, dy)$ with respect to $\omega_{z}(dx)\otimes \omega_z (dy)$. 
    In a very recent paper, Kajino and Murugan extended the results of Na\"im and Doob
    to the part process $X^0$ of a symmetric diffusion $\tilde X$ killed open leaving a uniform domain $D\subset \sX$,
    where $\tilde X$ is the diffusion process on $\sX$ associated with a  strong local regular Dirichlet space $(\sX, \tilde d,\tilde m, \tilde \sE, \tilde \sF)$ that enjoys a two-sided heat kernel estimates {\rm HK($\Psi$)} (see \S \ref{S:4} below for its definition).
    In this case, the Martin boundary coincides with the    topological boundary $\partial D$.    They  showed in \cite[Proposition 3.14 and Theorem 5.8]{KM}
    that  the Na\"im kernel $\theta (x, y)$ exists on $\bar D \setminus D \setminus {\rm diag}$  as a  continuous limit of 
    $ \frac{G_D(x, y)}{G_D(z, x) G_D(z, y)}$ and that the Doob-Na\"im formula \eqref{e:1.13} 
    holds. \qed
    \end{remark}

 \medskip
 
  In this paper, we  aim at  obtaining 
   explicit characterizations of the trace Dirichlet space  $(\check \sE , \check \sF_e )$ on $\partial D$ and on   explicit bounds for  the jumping measure $\check J(dx, dy)$ and the killing measure $\check \kappa(dx)$, in a form  as illustrated by the following simple  example. 
  As mentioned previously, when  $\bar X$ is the reflected Brownian motion on  the upper half space $\H^{d+1}$, its trace process
 on $\partial \H^{d+1} \cong \R^d$ is the isotropic Cauchy process. In this case, 
 $\bar \sE (f, f )=\frac12 \int_{\H^{d+1}} |\nabla f(x)|^2 dx$,
 $\bar \sF = \{f\in L^2  (\H^{d+1}; dx): \bar \sE (f, f) <\infty\}$, 
  $\bar \sF_e= \{f\in L^2_{loc}(\H^{d+1}; dx): \bar \sE (f, f) <\infty\}$ is the Beppo-Levi space on $\H^{d+1}$,
  $$
  \check \sE (u, u)= \frac12 \int_{\R^d\times \R^d} (u(x)-u(y))^2 \frac{c}{|x-y|^{d +1}} dx dy,
  $$
  and $\check \sF_e= \{u\in {\mathcal B}(\R^d): \check \sE (u, u)<\infty\}$;
  see \cite[Examples $1^o$ and $5^o$ in \S 6.5]{CF}.  Denote by $m$ and $\nu$ the Lebesgue measures on $\H^{d+1}$ and $\partial \H^{d+1}\cong  \R^d$, respectively.  Note that $\check \sF_e\cap L^2(\R^d; \nu)$ is the Sobolev space $W^{1/2, 2} (\R^d)$ of fractional order.  Observe that the density $\frac{c}{|x-y|^{d+\alpha}}$ of the jump kernel of $(\check \sE, \check \sF_e)$  can be expressed,
    as  
    $     \frac{c_1 m(B(x, r))} { r^2 \nu (B(x, r))^2} $ with $r=|x-y|$.  
Observe also that the scale function
    $ r^2$, up to a constant multiple, is the constant in the Poincar\'e inequality 
   on  balls with radius $r$ in $\H^{d+1}$ for Brownian motion. 
     
     We assume $(D, d)$ is a uniform domain;  see Definition \ref{D:2.1} below.
In the first part of this paper, we  derive explicit characterizations of the trace Dirichlet space  $(\check \sE , \check \sF_e )$ on $\partial D$,
and establish restriction and extension theorems. 
As we saw from \eqref{e:Feller}, the weighted harmonic measure $\omega_0$ defined by \eqref{e:1.2} plays a special role.    
In the second part of this paper, we study doubling property of a renormalized  harmonic measure $\omega$ that is equivalent to $\omega_0$,
and  derive explicit two-sided bounds on $\check J(dx, dy)$ and $\check\kappa(dx)$ with respect to   $\omega$ on $\partial D$. 
This renormalized harmonic measure $\omega$ is a usual 
harmonic measure $\omega_x$  when $\partial D$ is bounded and is the elliptic measure from infinity when $\partial D$ is unbounded.

Suppose that     $\Psi$ is  a continuous bijection from $(0,\infty)$ to $(0,\infty)$ such that there are  constants $C_\Psi\in [1,\infty)$ and $0< \beta_1\leq \beta_2<\infty$ so that 
\begin{equation}\label{eqnpsi} 
C_\Psi^{-1}\Big(\frac{R}{r}\Big)^{\beta_1}\leq\frac{\Psi(R)}{\Psi(r)}\leq C_\Psi \Big(\frac{R}{r} \Big)^{\beta_2}
\quad \hbox{ for every } 0<r<R<\infty . 
\end{equation}
  For a Radon measure $\nu$ with full support on $\partial D$, define  a state-dependent scale function on $\partial D\times (0, \infty)$ by 
\begin{equation}\label{e:1.12}
\Theta_{\Psi, \nu } (x,r):=\Psi(r)\frac{\nu \big(B(x,r)\big)}{m_0\big(B(x,r\big)\big)}\quad 
 \hbox{ for }x\in \partial D \hbox{ and } r>0.
\end{equation} 
For notational convenience, in the sequel,  we denote $\mu (B(x, r))$ by $V_\mu (x, r)$  for any Radon measure $\mu$ on $\tD$.
When $\mu=m_0$, we write $V(x, r)$ for $V_{m_0}(x, r)$.  In this notation,
$$
\Theta_{\Psi, \nu } (x,r)=\Psi(r)\frac{V_\nu  (x,r)}{V(x,r) }\quad 
 \hbox{ for }x\in \partial D \hbox{ and } r>0.
 $$

We  introduce  a corresponding  Beppo-Levi type space $\dot \Lambda_{\Psi,  \sigma}$ and Besov type space 
$ \Lambda_{\Psi,  \sigma}$ on $\partial D$ as follows.

\begin{definition}\label{def34}
Suppose that  $\sigma$ is  a  Radon measure  with full support on $\partial D$.
 For each $f\in \mcB (\partial D)$, we define a Besov-type  semi-norm associated with $\sigma$  by 
\begin{align*}	  
\lb  f  \rb_{\Lambda_{\Psi, \sigma}}&:=\Big(\int_{x\in \partial D}\int_{y\in \partial D}
\frac{\big(f(x)-f(y)\big)^2}{V_\sigma (x,d(x,y))  \, \Theta_{\Psi, \sigma} \big(x,d(x,y)\big)} \, \sigma(dy)\sigma(dx)\Big)^{1/2}.
\end{align*}
We define the Beppo-Levi type space $\dot \Lambda_{\Psi,  \sigma}$
and the   Besov type space  $\Lambda_{\Psi,  \sigma}$ on $\partial D$  by 
  \begin{eqnarray*}
  \dot \Lambda_{\Psi,  \sigma}  &:= &
  \left\{f \in \mcB (\partial D): \,\lb f \rb_{\Lambda_{\Psi,  \sigma}}<\infty\right\}= 
  \left\{f\in L_{\rm loc}^2 (\partial D; \sigma): \,\lb f \rb_{\Lambda_{\Psi,  \sigma}}<\infty\right\}, \medskip   \\
  	 \Lambda_{\Psi,  \sigma} &:=& \dot \Lambda_{\Psi,  \sigma}   \cap L^2(\partial D; \sigma)=
	 \left\{f\in L^2(\partial D;\sigma):\,\lb f \rb_{\Lambda_{\Psi,  \sigma}}<\infty\right\}.
\end{eqnarray*}
Note that  $\lb  f  \rb_{\Lambda_{\Psi,  \sigma} }$ is a semi-norm on  $\dot \Lambda_{\Psi,  \sigma}  $ with $\lb  f  \rb_{\Lambda_{\Psi,  \sigma} }=0$
if and only if $f$ is constant $\sigma$-a.e. on $\partial D$. 
Define  
$$
\|f\|_{\Lambda_{\Psi,  \sigma} }:=\sqrt{\lb f \rb^2_{\Lambda_{\Psi,  \sigma} }+\|f\|_{L^2(\partial D;\sigma)}^2}.
$$
Then $\Lambda_{\Psi,  \sigma} $ is a Hilbert space with   norm $\| \cdot \|_{\Lambda_{\Psi,  \sigma} }$. 
 \end{definition}

 \begin{remark} \rm 
The second equality in the definition of $\dot\Lambda_{\Psi,\sigma}$  is due to the fact 
 that $f\in L^2_{loc}(\partial D;\sigma)$ for any $f\in\mcB(\partial D)$ having  $\lb f\rb_{\Lambda_{\Psi,\sigma}}<\infty$. 
 Indeed, suppose $f\in\mcB(\partial D)$ having  $\lb f\rb_{\Lambda_{\Psi,\sigma}}<\infty$. Then  for  every $x\in \partial D$ and $0< r<\diam(\partial D)/3$,   
 \begin{align*}
\int_{y\in B(x,r)\cap\partial D}\int_{z\in \partial D\setminus B(x, 3 r)}
 \frac{\big(f(y)-f(z)\big)^2}{V_\sigma (y,d(y, z))\, \Theta_{\Psi, \sigma} \big(y,d(y,z)\big)} \, \sigma(dz)\sigma(dy)<\infty. 
 \end{align*} 
 So by Fubini's Theorem, there exists $z\in \partial D\setminus B(x, 3 r)$ such that
 \[
 \int_{B(x,r)\cap\partial D}\frac{(f(y)-f(z))^2}{V_\sigma (y,d(y, z))\,\Theta_{\Psi, \sigma}\big(y,d(y,z)\big)}\sigma(dy)<\infty.
 \]
 This implies that $f\in L^2(B(x,r)\cap \partial D;\sigma)$, as 
 \begin{align*}
 \frac{V_\sigma (x,r)^2\Psi( 2 r)}{V (x,2r+d(x,z))}\leq  {V_\sigma (y,d(y, z))}  \Theta_{\Psi, \sigma} \big(y,d(y,z)\big)
\leq \frac{V_\sigma  (x,2r+d(x,z))^2\Psi(r+d(x,z))}{V(x,r)}.
 \end{align*}
 \end{remark}

Suppose now  $\sigma$ is  a     Radon measure  with full support on $\partial D$  satisfying {\rm (VD)} property and  that 
    the scale function    $\Theta_{\Psi, \sigma}$ satisfies the following lower scaling property (LS):
    there are positive constants $C$ and $\beta$ so that 
    \begin{equation}\label{assum3}
  \frac{\Theta_{\Psi, \sigma} (x, R)}{\Theta_{\Psi, \sigma} (x,r)}\geq C \Big(\frac{R}{r} \Big)^\beta
 \quad \hbox{for each }x\in \partial D \hbox{ and } 0<r<R\leq\diam(\partial D). 
\end{equation} 
    The (LS)  property for $\Theta_{\Psi, \sigma}$  plays an important role in our approach in this paper.   We also remark that when $\sigma$ and $m_0$ are doubling measures, $\Theta_{\Psi, \sigma}(x,r)$ has the doubling property, and hence satisfies 
    an 
    upper scaling property, that is, for some positive constants $C'$ and $\beta'$
    	\[
    	\frac{\Theta_{\Psi, \sigma} (x, R)}{\Theta_{\Psi, \sigma} (x,r)}\leq C'\Big(\frac{R}{r} \Big)^{\beta'}
    	\quad \hbox{for each }x\in \partial D \hbox{ and } 0<r<R<\infty. 
    	\]
    See the discussion below Definition \ref{D:3.1}. 
    
    The main results of this paper are   as follows.

\begin{enumerate}[(i)]
\item  (Restriction Theorem.) Theorems  \ref{thmrestriction1} and \ref{thmrestriction2}
 on the embedding of $\check \sF_e$ into $\dot \Lambda_{\Psi, \sigma}$
under the assumption  that the Poincar\'e inequality \rm PI$(\Psi;D)$ holds for  $(D,d,m, \sE^0, \sF^0)$.
Moreover, under these assumptions,  $\sigma$ is a smooth measure that does not charge zero $\bar \sE$-polar sets. 

\item (Extension Theorem.)  Under the condition that {\rm PI}$(\Psi;D)$ as well as a capacity upper bound condition 
  {\rm Cap}$_\leq(\Psi; D)$  hold  for  $(D, d, m, \sE^0, \sF^0)$, it is shown in Proposition \ref{P:3.15} and Theorem \ref{T:3.16}
  that $\sigma$ is a smooth measure having full $\bar \sE$-quasi-support on $\partial D$ and that 
  $ \dot \Lambda_{\Psi, \sigma} \cap C_c(\partial D)$   is a core in $\check \sF$. 
In particular, it implies that  for  $u\in \check \sF_e $, 
  if either $\partial D$ is  unbounded or  $(\bar \sE , \bar \sF )$ is recurrent, then   
  \begin{equation}\label{eqn11}
\check \sE(u,u)\asymp \int_{\partial D\times\partial D}\frac{\big(u(x)-u(y)\big)^2}{\Theta_{\Psi, \sigma} (x,d(x,y))V_\sigma (x,d(x, y))}\sigma(dx)\sigma(dy);
\end{equation}
 if $\partial D$ is bounded and $(\bar \sE , \bar \sF )$ is transient, then
\begin{align}\label{eqn1.1}
\begin{split} 
	\check \sE(u,u)
	&\asymp  \int_{\partial D\times\partial D}\frac{\big(u(x)-u(y)\big)^2}{\Theta_{\Psi, \sigma} (x,d(x,y)) V_\sigma (x,d(x, y))}\sigma(dx)\sigma(dy)\\
	&\quad\ +\int_{\partial D}u (x)^2d\sigma (dx).
\end{split} 
\end{align}
 Here $\asymp$ means the ratio of both sides is bounded between two positive constants.

\item (Doubling property of harmonic measure.) 
 Under a heat kernel estimate condition {\rm HK($\Psi$)} for  the reflected Dirichlet space $(\tD, d, m_0, \bar \sE, \bar \sF)$, 
Theorem \ref{thm41} and Theorem \ref{thm5+7}  give the characterization of the doubling property of harmonic measures
and renormalized harmonic measure in terms of the relative boundary capacity doubling property. 

\item (Equivalent conditions for (LS) property of the scale function $\Theta_{\Psi, \omega}$.) 
Equivalent conditions are given in Theorem \ref{T:6.1}  
for the renormalized harmonic measure $\omega$ being a doubling measure with full quasi-support on
$\partial D$ and $\Theta_{\Psi, \omega}$ having (LS) property 
 under a heat kernel estimate condition {\rm HK($\Psi$)} for 
 the reflected Dirichlet space  $(\tD, d, m_0, \bar \sE, \bar \sF)$, including a capacity density condition. It is also equivalent to existence of any   doubling Radon measure $\nu$ with full quasi-support on $\partial D$ so that $\Theta_{\Psi, \nu}$ having (LS) property. The latter gives an effective way to verify the  boundary capacity density condition
  in concrete cases; see the snowflake example in \S \ref{S:9.3}.

\item (Two-sided estimates on $\check J(dx, dy)$ and $\check \kappa(dx)$.) 
Under a heat kernel estimate condition {\rm HK($\Psi$)} for the reflected Dirichlet space  $(\tD, d, m_0, \bar \sE, \bar \sF)$ and any of the equivalent conditions in (iv), 
we show in Theorem \ref{thm71} that 
\begin{equation}\label{eqn12}
\check J(dx,dy)\asymp \frac{\omega(dx)\omega(dy)}{V_\omega (x,d(x, y)) \Theta_{\Psi, \omega} (x,d(x,y)) },
\end{equation}
and 
\begin{equation}\label{e:1.18}
\check\kappa(dx) \asymp \omega (dx)
\end{equation}
when $\partial D$ is bounded and $(\bar \sE, \bar \sF)$ is transient, and $\check\kappa=0$ otherwise. 
 From which, we conclude that the trace processes are of the mixed stable-like and one can derive the two-sided heat kernel estimates
from these estimates and the results from  \cite{CKW}; see Theorem \ref{T:8.2}.
\end{enumerate}

\medskip

These results are new even on Euclidean spaces for reflected Brownian motions and symmetric reflected diffusions  
in inner uniform domains in $\R^d$; see \S \ref{S:9.5} below. 
We point out that although the weighted harmonic measure $\omega_0$ is a natural smooth measure with full
$\bar \sE$-support on $\partial D$ to use   for the trace process,    
 it typically does have not a concrete expression. So it is important in (i) and (ii) above that we have the freedom
 to choose other smooth measures on the boundary $\partial D$ to characterize the domain of the trace Dirichlet spaces. For instance,
 when $D$ is the Koch snowflake   domain in $\R^2$, its harmonic measure does not have a good concrete expression.
On the other hand,  the Hausdorff measure   $\mu$ on 
$\partial D$, which is a smooth measure,  is  Ahlfors $d$-regular with $d=\frac{\log4}{\log3}$. 
Using it  one can easily characterize the domain of the trace Dirichlet form on $\partial D$ via 
\eqref{eqn11}-\eqref {eqn1.1};  see \S \ref{S:9.3} for details.  

Under   a heat kernel estimate condition {\rm HK($\Psi$)} for 
 an ambient Dirichlet space   $(\sX, \tilde d, \tilde m, \tilde \sE, \tilde \sF)$,
the two-sided jump kernel estimates in \eqref{eqn12} has also been independently obtained 
in \cite[Proposition 5.8]{KM} under a slightly stronger condition  (CDC) than our capacity density condition \eqref{e:6.2}, 
one of  the equivalent conditions mentioned in (iv).
In a recent updated version,   the authors  outlined in  \cite[\S 5.4]{KM} how their arguments can be modified to establish 
the estimates \eqref{eqn12}  as well as \eqref{e:1.18} under the same   condition as ours. 

\medskip
 
The first part of this paper can be regarded as boundary trace theorems   for reflected diffusions on metric measure spaces.
Boundary trace theorems for Sobolev and Besov spaces on Euclidean spaces have been extensively investigated;
see, e.g., \cite{A, AdH, JW, W}.  
For instance, for a uniform domain $D$ in $\R^n$ whose boundary $\partial D$ is Ahlfors $d$-regular 
with $d\in [n-1, n)$, it is shown in \cite[Chapter VII]{JW} that the trace of the Sobolev space $W^{1,2}(D)$ on $\partial D$ is the Besov space $B^{2,2}_\beta (\partial D)$ with $\beta =1 -(n-d)/2$ and there are bounded linear restriction and extension operators between these two spaces. The Besov space  $B^{2,2}_\beta (\partial D)$ can be represented by in terms of the Hausdorff measure $\sigma$ on $\partial D$: 
$$
B^{2,2}_\beta (\partial D)
=\left\{ u\in L^2(\partial D; \sigma):  \int_{\partial D\times \partial D} \frac{(u(x)-u(y))^2}{|x-y|^{d+2\beta}} \sigma (dx) \sigma (dy)
<\infty \right\}.
$$
Observe that in this context, $\Psi (r)=r^2$, $m_0$ is the Lebesgue measure on $D$, $V_\sigma(x, r) \asymp   r^d$. So
$\Theta_{\Psi , \sigma}(x, r)\asymp r^{2+d-n}$  and $\Theta_{\Psi , \sigma}(x, r)V_\sigma (x, r) \asymp r^{d+2\beta}$. 
Thus the Besov space $B^{2,2}_\beta (\partial D)$ 
is exact the space $\Lambda_{\Psi, \sigma}$ defined in Definition \ref{def34}. 
 In the context of fractals,  Jonsson \cite{Jo} studied the trace of the standard self-similar Dirichlet form on the Sierpinski gasket onto the bottom line (with respect to the $1$ dimensional Lebesgue measure), and  showed that the trace Dirichlet space is the Besov space $B^{1,1}_\alpha([0,1])$, with $\alpha=\frac{\log 5}{\log 4}-(\frac{\log 3}{\log 2}-1)/2$. It has been further investigated recently in \cite{KTa}. The  result of \cite{Jo} is extended  in \cite{HK} 
 to  a class of self-similar sets, with an application to the penetrating processes on fractal fields \cite{HaK, HK, Ku}. Extension and restriction theorems have also been studied recently for Newton-Sobolev functions in metric measure spaces; see  \cite{BS,GS}  and the   references therein.
   Boundary trace theorem also plays a central role in our recent work \cite{CC}   in solving an open question of Barlow-Bass about the convergence of resistances on Sierpi\'nski carpets.

The second part is on two-sided estimates of the jump kernel $\check J(dx, dy)$ and the killing measure $\check\kappa(dx)$ with respect to 
the renormalized harmonic measure $\omega$ under a heat kernel estimates condition {\rm HK$(\Psi)$} for the reflected Dirichlet space
$(\bar \sE, \bar \sF)$, or equivalently, for the reflected diffusion $\bar X$ on $\tD$. 
We emphasize that the Beuring-Deny decomposition \eqref{e:2.7} for the trace Dirichlet form $(\check \sE, \check \sF)$, and the formulas \eqref{e:1.6}-\eqref{e:1.7} and 
\eqref{e:1.9}-\eqref{e:1.10}  for the Feller measure $\check J$ and supplementary Feller measure $\check\kappa$ hold for any strongly local regular Dirichlet form
$(\sE^0, \sF^0)$ whose actively reflected Dirichlet form $(\bar \sE, \bar \sF)$ is regular on $L^2(\tD; m_0)$.  
No uniform domain assumption on the metric space $(D, d)$ nor heat kernel estimate condition {\rm HK$(\Psi)$}  is needed. 
This is in contrast with the  Na\"im kernel for the boundary trace of reflected diffusion $\bar X$,   whose existence is established in   \cite{KM},
under a heat kernel estimates condition {\rm HK$(\Psi)$} and the condition that $D$ is a uniform domain.  
  In view of \cite{BCM},  the  {\rm HK$(\Psi)$}  condition for $(\bar \sE, \bar \sF)$  is essentially 
  equivalent to that a scale-invariant elliptic Harnack inequality holds
for $(\bar \sE, \bar \sF)$. Thus all the harmonic measures $\{\omega_x ; x\in D\}$ are equivalent
to the weighted harmonic measure $\omega_0$.  However, the weighted harmonic measure $\omega_0$
may not have the doubling property in general. The renormalized harmonic measure is a measure that is equivalent to
$\omega_0$ and has the doubling property.  
Consequently, we know from \eqref{e:Feller} that $\check J(dx, dy) \ll \omega (dx) \omega (dy)$ and $\check \kappa(dx) \ll \omega (dx)$. 
We focus on the two-sided estimates of $\frac{\check J(dx, dy)}{ \omega (dx) \omega (dy)}$ and $\frac{\check \kappa(dx)}{\omega (dx)}$
rather than on their exact expressions. It is established in \cite{CKW, CKW2}    that 
 many important objects such as heat kernel estimates and parabolic Harnack inequalities  
for symmetric jump diffusions are invariant under bounded perturbations of the jump kernels.

In a very recent paper \cite{KM} by Kajino and Murugan, under a slightly stronger condition (under which the killing measure $\check\kappa$ for the boundary trace process has to vanish), they also obtained estimate \eqref{eqn12}, independently, by showing the existence of Na\"im kernel and deriving the Doob-Na\"im formula for the trace Dirichlet form $(\check \sE, \check \sF)$.  
They assumed that there is  an  ambient complete strongly local MMD space $(\sX, d, m, \tilde \sE, \tilde \sF)$  
that satisfies   (VD) and {\rm HK$(\Psi)$} so that $(\sE^0, \sF^0)$ is its  part Dirichlet form on  an uniform domain $D$ in $(\sX, d)$.
Their approach is different from ours.  See Remark \ref{R:8.4} for more information.

We do not use the Doob-Na\"im approach as described in Remark \ref{R:1.1}(ii), nor do we directly use the Feller measure formulas \eqref{e:1.6}-\eqref{e:1.7} and \eqref{e:1.9}-\eqref{e:1.10}. 
The idea of  our study  of the jump kernel $\check J(dx, dy)$   can be illustrated by  the following observation of a toy model, for which we can extract the jump kernel information directly from the active  reflected Dirichlet form
  $(\bar \sE, \bar \sF)$. 
  Consider a star shaped electrical network on $V=\{o\}\cup \partial V$, where $o$ is a central node, and $\partial V=\{x_1,x_2,\cdots,x_n\}$ are viewed as the boundary. Let $m$ be the counting measure on $V$. The Dirichlet form for the  continuous time reflected random walk $\bar X$ is 
$(\bar \sE, \sB (V))$, where $\sB (V)$ is the space of all measurable functions on the vertex set $V$ and 
\[
  \bar \sE (f,g):=    \sum_{i=1}^nc_i\big(f(o)-f(x_i)\big)\big(g(o)-g(x_i)\big)
  \quad \hbox{ for every }f,g\in \mcB (V).
\]
It is easy to see that the discrete harmonic measure on $\partial V$ is given by 
$$
\omega(x_j):= \bP_o ( \bar X_{\sigma_{\partial V}}=x_j ) = \frac{c_j}{ \sum_{i=1}^n c_i }  
\quad \hbox{for } j=1, \cdots, n.
$$
 We use the notation ${\mathbbm 1}_A$  to denote the indicator function of $A$ on $V$, that is,   
  ${\mathbbm 1}_A(x)=1$ if $x\in A$ and ${\mathbbm 1}_A(x)=0$ if $x\in V\setminus A$. For $f\in \mcB (\partial V)$, 
   let $\mathcal{H}f$ be the harmonic extension of $f$, that is,  
$$
\mathcal{H}f(x) =
\begin{cases}
f(x_i)   \quad  \quad &\hbox{when } x=x_i,  \, i=1,2,\cdots,n , \\
 \sum_{i=1}^n  \omega(x_i) f(x_i)  &\hbox{when  } x=o. 
 \end{cases}
 $$
   For $f,g\in \mcB (\partial V)$,  
  \begin{eqnarray*}
   \check  \sE (f,g)  &=&\bar  \sE (\sH f, \sH g) = \sum_{i,j=1}^n f(x_i)g(x_j) \bar \sE (\sH \1_{\{x_i\}}, \sH \1_{\{x_j\}})  
       \\
      &=& \frac12 \sum_{i, j=1}^n c_{i,j} (f(x_i)-f(x_j))( g(x_i)-g(x_j)),
     \end{eqnarray*}
where $c_{i,j}:= - \bar \sE (\sH \1_{\{x_i\}}, \sH \1_{\{x_j\}})$ and the last identity is due to the fact that for every $1\leq j\leq n$, 
$$
\sum_{i=1}^n c_{i, j}= \bar \sE (\sH \1_{\partial V}, \sH \1_{\{x_j\}})  = \bar \sE (1, \sH \1_{\{x_j\}}) = 0.
$$ 
The constant $c_{i,j}$ gives the discrete jump intensity for the boundary trace process $\check X$  to jump from $x_i$ to $x_j$.
By definition,  for $i\not= j$, 
\begin{eqnarray}\label{eqn13}
 c_{i,j} &=& - \sum_{k=1}^n c_k  \left( \omega (x_i)  - \1_{\{x_i\}}(x_k)  \right) \left(\omega (x_j) - \1_{\{x_j\}}(x_k) \right) \nonumber  \\
&=&-     \sum_{k=1}^n   c_k  \omega (x_i) \left(\omega (x_j) - \1_{\{x_j\}}(x_k) \right) 
 +  \sum_{k=1}^n c_k  \1_{\{x_i\}}(x_k) \left(\omega (x_j) - \1_{\{x_j\}}(x_k) \right)  \nonumber \\
&= & - \omega (x_i)  \omega (x_j)   \sum_{k=1}^n c_k    + c_j  \omega (x_i) + c_i  \omega (x_j) \nonumber \\
&= &  \omega (x_i)  \omega (x_j)   \sum_{k=1}^n c_k     \nonumber \\
  &=& \omega(x_i)\omega(x_j) \bar \sE({\mathbbm 1}_{\{o\}},{\mathbbm 1}_{\{o\}})   . 
 \end{eqnarray}
We will use a similar strategy in the general setting of strongly local Dirichlet forms. For non-negative $f,g\in C_c(\partial D)\cap \check \sF$ with small, non-intersecting supports (meaning  $\operatorname{supp}[f]$ and $\operatorname{supp}[g]$  have smaller diameters  than $d(\operatorname{supp}[f],\operatorname{supp}[g]$), we choose a suitable compact set $K$ such that 
$$
\diam(K)\asymp 
d(K,\partial D)\asymp d(K,\operatorname{supp}[f])\asymp d(\operatorname{supp}[f],\operatorname{supp}[g]).
$$
The compact  $K$  plays the same role as the central point $o$ in the discrete setting. Let $e_K$ be the condenser potential of $K$ in $D$, that is, $e_K(x)= \bP_x (\sigma_K <\tau_D)$ where $\tau_D:=\inf\{ t\geq 0: X_t\notin D\}$. Then 
  by a strategy similar to \eqref{eqn13} and by careful estimates of the error terms, we can show   
\[
\check  \sE  (f,g)\asymp \mathcal{H}f(y)\mathcal{H}g(y)\,\bar \sE (e_K,e_K) \quad \hbox{ for }y\in K. 
\]
Then, \eqref{eqn12}  
follows from the observations that $\omega_x\asymp \frac{\omega}{\omega(E)}$ on $E$, where $E$ is a neighborhood of the support of $f,g$ with radius about $r$, and $ \bar\sE(e_K,e_K) \asymp \frac{m(B(x,r))}{\Psi(r)}$ with $x$ being a point in the support of $f$. \medskip
Of course, carrying this strategy out rigorously in the general setting of   strongly local Dirichlet forms on metric measure spaces 
 requires much more  efforts and careful analysis in depth, and needs a two-sided heat kernel estimates, or equivalently,
 the elliptic Harnack inequality assumption for the reflected Dirichlet space $(\bar \sE,  \bar\sF )$ on $L^2(\tD; m_0)$. 
 A similar approach has also been used in \cite{KM}, see, e.g., the proof of  Theorem 5.8 there.
 
\medskip

The rest of the paper is organized as follows. In Section \ref{S:2},
we carefully lay out the settings of this paper and present some basic properties  of uniform domains. 
Trace  theorems are studied in Section \ref{S:3}. 
Volume doubling property and local comparability of harmonic measures are investigated in Sections \ref{S:4} and \ref{S:5},
respectively.  In Section \ref{S:6}, equivalence between the lower scaling property (LS) of the scale function
$\Theta_{\Psi, \sigma}$ and the capacity density condition is given. In Section \ref{S:7}, we derive estimates on the jump kernel and killing measure
for the boundary trace process,  while the two-sided heat estimates for the trace process  are given in Section \ref{S:8}. 
Several examples are given Section \ref{S:9} to illustrate the scope of the main results of this paper. 
 
We mention that the approach developed in this paper is quite robust. It works for the trace process of reflected jump diffusions as well. This will be carried out in a forthcoming paper.

\smallskip

In this paper,  we use $:=$ as a way of definition.  
For $a, b\in \R$, $a\wedge b:=\min\{a, b\}$, $a\vee b:=\max\{a, b\}$, and $a^+ := a\vee 0$. 
We denote by $[a]$ the largest integer not exceeding $a\in \R$. 
The notation $f\lesssim g$ means that  there exists $C\in (0,\infty)$ such that $f\leq Cg$, and $f\asymp g$ means that $f\lesssim g\lesssim f$ on the common domain of definitions of $f$ and $g$.  For a subset $A$, $\1_A$ denotes the indicator function of $A$.

\section{Basic settings}\label{S:2}

In this section, we introduce the basic settings for this paper, including the geometric assumptions about the state space,  Dirichlet forms and the associated diffusion processes.

\subsection{State space}\label{S:2.1}

Let $(D,d)$ be a locally compact separable metric measure space and $m$ a Radon measure on $D$ with full support. 
 We denote by $(\tD,d)$ the completion of $(D,d)$, and extend the measure $m$ to a measure $m_0$ on $\tD$ by 
 setting $m_0(E)=m(E\cap D)$ for  $E\subset \tD$.  Note that  $(\tD,d)$ is also a  separable metric   space and $m_0$ has full support on $\tD$.  Moreover, $D$ is an open subset of $(\tD,d)$ as every point $x$ in $D$ has a compact neighborhood in $D$ by the local compactness of $(D,d)$.  
 
 We write 
\[
B(x,r):=\{y\in \tD:\,d(x,y)<r\}
\]
for the open ball in $\tD$. For each $E\subset \tD$, we denote by $\bar{E}$ the closure of $E$ in $(\tD,d)$, and 
$\partial E :=\bar{E}\cap \overline{(\tD\setminus E)}$ is the boundary of $E$.  In particular, $\partial D:=\tD \setminus D$, and  $\overline{B(x,r)}$ is the closure of an open ball centered at $x$, which may not equal to the closed ball $\overline{B}(x,r):=\{y\in \tD:\,d(x,y)\leq r\}$ of radius $r$ centered at $x$. 

For two subset $A, B\subset \tD$,   
 $$
 d(A,B):=\inf\limits_{x\in A,  \, y\in B}   d(x,y)
 $$ 
  is the distance between $A$ and $B$, 
 and    $d(x,A)=d(\{x\}, A)$ is the distance between $x\in \tD$ and $A$.
  For $A\subset \tD$, $\diam(A) :=\sup_{x,y\in A}d(x,y)$ is the diameter of $A\subset \tD$. 

\medskip

For a Borel measurable $E\subset \tD$, denote by $\mcB(E)$   the space of Borel measurable functions on $E$; $C(E)$ the space of continuous functions on $(E,d)$; and $C_c(E)$ the space of continuous functions on $E$ with compact support, i.e. $\overline{\{x\in E:f(x)\neq 0\}}\cap E$ is compact  for every $f\in C_c(E)$. Let  $C_0(E)$  be the  closure of $C_c(E)$ with respect to the supremum norm $\| f \|_\infty:= \sup_{x\in E} |f(x)|$. 
We denote by $C_b (E)$ the space  bounded continuous functions $E$. 

  For $f\in C(E)$ on a closed subset $E\subset\tD$, we denote the support of $f$ by $\operatorname{supp}[f]$, 
  i.e.,  $\operatorname{supp}[f]:=\overline{\{x\in E:f(x)\neq 0\}}$. 

\smallskip 

 Throughout this paper, we always assume that $m_0$ is Radon on $\tD$ and is  volume doubling (VD),
that is,  there is $C\in (1,\infty)$ such that
\[
  V(x,2r) \leq C\,  V(x,r)     \qquad \hbox{ for every }x\in \tD   \hbox{ and }   r\in (0,\infty). 
\]
This is equivalent to the existence of positive constants $c_1$ and $d_1$ so that 
\begin{equation}\label{e:VD}
 \frac{V(x, R)}{V(x, r)}  \leq c_1  \Big(\frac{R}{r} \Big)^{d_1}  \quad 
\hbox{ for every }x\in \tD   \hbox{ and }   0<r \leq R<\infty.
\end{equation}
 We say that reverse volume doubling property (RVD) holds if there are positive constants 
positive constants $c_2$ and $d_2$ so that 
$$
 \frac{V(x, R)}{V(x, r)}  \geq  c_2 \Big(\frac{R}{r} \Big)^{d_2}  \quad 
\hbox{ for every }x\in \tD   \hbox{ and }   0<r \leq R  \leq {\rm diam}(\tD).
$$
It is known that (VD) implies (RVD) if $\tD$ is connected; See \cite[Proposition 2.1 and  the  paragraph before Remark 2.1]{YZ}. Moreover, since $(\tD, d)$ is complete, (VD) implies that each open ball is relatively compact in $(\tD, d)$  in view of Lemma \ref{lemma2path}(a), which shows every closed ball is totally bounded.

 \medskip

\begin{lemma}\label{lemma2path}
\begin{enumerate}  [\rm (a)]
\item Let $E\subset\tD$ be a bounded set. Then, there is an integer  $N^* \geq 1$ depending only on $\diam(E)/r$ and the parameter in {\rm (VD)} for $m_0$ such  that one can find $\{z_i\}_{i=1}^N\subset E$ 
 so that $N\leq N^*$ and $E\subset\bigcup_{i=1}^NB(z_i,r/2)$.
	
\item Let  $\gamma$ be a path in $\tD$ and $ r>0$. Then, there is   an upper bound  $L>0$  depending only on $\diam(\gamma)/r$ and the parameter in {\rm (VD)} for $m_0$ such that we can find a sequence $\gamma(0)=z_0,z_1,\cdots,z_l=\gamma(1)$ in $\gamma$ such that $l\leq L$ and $d(z_i,z_{i+1})<r$ for $i=0,1,\cdots,l-1$. 
\end{enumerate}
\end{lemma}

\begin{proof}
(a).  This is a standard statement, as the (VD) doubling property of $m_0$ implies that $(\tD,d)$ is a doubling space, see \cite[Section 10.13,  Exercise 10.17]{Hei}. For the convenience of readers, we provide a detailed proof here. 

We find the finite set of points $\{z_i\}_{i=1}^{N}$ by the following procedure. First, we pick $z_1\in E$. Next, if $E\subset B(z_1,r/2)$, we do nothing and end the process with $\{z_1\}$; otherwise we pick $z_2 \in \gamma\setminus B( z_1, r/2)$ to form a larger set $\{z_i\}_{i=1}^2$. Next, we repeat the procedure for the set $\{z_i\}_{i=1}^2$. If $\gamma\subset \bigcup_{i=1}^2B(z_i,r/2)$, we end the process; otherwise, we pick $z_3\in\gamma\setminus\bigcup_{i=1}^2B(z_i,r/2)$ to form  $\{z_i\}_{i=1}^3$. We keep doing this until $E\subset \bigcup_{i=1}^{N}B(z_i,r/2)$. This process has to stop after finitely many steps and $N$ has an upper bound that depends only on $\diam(E)/r$ and the parameter in (VD) for $m_0$. This is because  $\{B(z_i,r/4);i\geq 1\}$ are pairwise disjoint, and for $i\geq 1$, 
	\[
	B(z_i,r/4)\subset B(z_1,\diam(E)+r/4)\subset 
	B(z_i,2\diam(E)+r/4),
	\] 
	so by \eqref{e:VD} (VD) of $m_0$, 
	\begin{align*}
		&\quad\ N\cdot  V (z_1,\diam(E)+r/4)
		\leq \sum_{i=1}^{N} V(z_i,2\diam(E)+r/4)   \\
		&\leq c (1+8\diam(E)/r)^{d_1}\sum_{i=1}^{N} V (z_i,r/4)
		\leq c(1+8\diam(E)/r)^{d_1} V (z_1,\diam(E) +r/4 ). 
	\end{align*} 
	It follows that $N\leq N^*:= [ c(1+8\diam(E)/r)^{d_1}]+1$.

(b). Let $z'_0=x$ and $z'_1=y$. By (a), we can find $\{z'_i\}_{i=2}^{L'}$ such that $L'\leq c(1+8\diam(\gamma)/r)^{d_1}+1$ 
and $\gamma\subset\bigcup_{i=2}^{L'} B(z'_i,r/2)$.
	Next, we define the set of edges 
	\begin{eqnarray*}
		E & =& \left\{ \{z'_i,z'_j\}:\,B(z'_i,r/2)\cap B(z'_j,r/2)\neq\emptyset,0\leq i,j\leq L'  \right\} \\
		&=& \left \{\{z'_i,z'_j\}:\,d(z'_i,z'_j)<r,0\leq i,j\leq L' \right\}.
	\end{eqnarray*}
	Then, $(\{z'_i\}_{i=0}^{L'},E)$ is a connected graph  as  $B(z'_i,r/2),0\leq i\leq L'$ is an open cover of the connected set $\gamma$. We can therefore find a path $\gamma(0)=z_0,z_1,z_2,\cdots,z_l=\gamma(1)$ in $\{z'_i\}_{i=0}^{L'}\subset\gamma$ such that $d(z_i,z_{i+1})<r/2$ for $0\leq i<l$ with $l \leq L' $.
\end{proof}

 Define the distance to the boundary function $d_D(x)$ on $D$ by 
\begin{equation} \label{e:dstf} 
	d_D(x) := \inf\{d(x, z): z\in D^*\setminus D \}   \quad \hbox{for } x\in D.
\end{equation}
For each $0\leq r<s<\infty$, we define
\[
D_{r, s}:=\{x\in \tD:\, r\leq d_D(x)<s\}
\]
and
\[D_{r}:=\{x\in \tD:\,  d_D(x)\geq r\}.\] 

\begin{definition} \label{D:2.1}  We say 
 $(D,d)$ is    $A$-uniform   for some positive constant $A>1$ if 
 for every $x,y\in D$, there exists a continuous curve $\gamma\subset D$ 
 so that $\gamma(0)=x$, $\gamma(1)=y$,  $\diam(\gamma) \leq A\,d(x,y)$ and 
\begin{align*}
  d_D(z)\geq A^{-1}\min\{d(x,z),d(y,z)\}\quad\hbox{ for every }z\in \gamma.
\end{align*}
We say $(D, d)$ is uniform if it is $A$-uniform for some $A>1$. 
\end{definition}

\begin{remark}\rm  
\begin{enumerate}
	\item Note that if $(D, d)$ is $A$-uniform,  then  $D$ is  path connected   in $(\tD,d)$. Consequently, $(\tD, d)$ is connected. We also note that 
	\[
	D\cap \partial B(x,r)\neq\emptyset\quad \hbox{ for }x\in\tD\hbox{ and } 0< r<\diam(D)
	\] 
	as we can find a path in $D$ that connects $D\cap B(x,r)$ and  $D\setminus\overline{B(x,r)}$.

	\item This definition of uniform domain is due to  \cite[Definition 2.9]{Vai}. This is also the same definition used in \cite[Definition 2.3]{Mathav} and in \cite[Definition 2.5]{KM}.
	See \cite[Theorem 2.10]{Vai} for various equivalent definitions of uniform domains in the Euclidean spaces,
	however some of which may not be equivalent in general metric measure spaces as discussed 
	in  \cite[Section 2.2]{Mathav}. 
\end{enumerate}

\end{remark}

\begin{lemma}\label{lemma21} Suppose that $(D,d)$ is an $A$-uniform domain for some $A>1$. 
\begin{enumerate} [\rm (a)]
\item For each $x\in \partial D$ and $r\in (0,\diam(D)/2)$, there is some $y\in D$ such that $B(y,r/(12A))\subset B(x,r)\cap D_{r/(4A)}$.
So  there is $C_1\in (0,1)$ such that 
\[
m_0(B(x,r))\geq m_0\big(B(x,r)\cap D_{r/(4A)}\big)\geq C_1m_0\big(B(x,r)\big).
\]

\item  Let $x,y\in D$, then there is a path $\gamma$ in $D$ connecting $x,y$ such that $\diam(\gamma)\leq A\,d(x,y)$ and 
\[
d_D(z)\geq \frac{d_D(x)\wedge d_D(y)}{1+A}\ \hbox{ for every }z\in \gamma.
\]
	
\item  Let $x\in \partial D$ and $r>s>0$. Suppose  that $B(x,r)\cap D_{s,r}\neq\emptyset$.
Then  there exists a path connected set $E$ such that
\[
B(x,r)\cap D_s\subset E\subset B(x,2Ar+r)\cap D_{s/(1+A)}.
\] 
\end{enumerate} 
\end{lemma}

\begin{proof}
(a). The first statement is known as the corkscrew condition, and it follows from a similar proof as \cite[Lemma 4.2]{BS}. 
We fix $x'\in D\cap B(x,r/3)$, $z\in D\setminus B(x,r)$ and pick a path $\gamma$ that connects $x',z$ in $D$ that has the properties in the definition of an $A$-uniform domain. There is some  $y\in \gamma$ so that $d(x',y)=r/3$. Note that  $d(x,y)\leq d(x,x')+d(x',y)\leq 2r/3$ and $d_D(y)\geq A^{-1}\min\{d(x',y),d(y,z)\}=r/(3A)$. This implies that $B(y,r/(3A))\subset D\cap B(x,r)$, and thus $B(y, r/{(12A)})\subset B(x,r)\cap D_{r/(4A)}$. Choose $k\geq 2$ so that $2^k\geq 24A$. Then
\[
B(y,2^kr/(12A))\supset B(y,2r)\supset B(x,r). 
\]
Hence, by (VD) property of $m_0$, we have 
\[
C_D^km_0\big(B(x,r)\cap D_{r/(4A)}\big)\geq C_D^k\,m_0\big(B(y,r/{(12A)})\big)\geq m_0\big(B(y,2^kr/(12A))\big)\geq m_0\big(B(x,r)\big),
\]
where $C_D$ is the constant of (VD). Hence (a) holds with $C_1:=C_D^{-k}$.

\smallskip

 (b). Let $\gamma$ be the path connecting $x,y$ as described in the definition of the uniform domain $D$. Then, $\diam(\gamma)\leq A\,d(x,y)$ and
\begin{align*}
d_D(z)&\geq \max\big\{A^{-1}\min\{d(z,x),d(y,z)\},\ d_D(x)-d(z,x),\ d_D(y)-d(z,y)\big\}\\
&\geq\max\big\{A^{-1}\min\{d(z,x),d(z,y)\},\ d_D(x)\wedge d_D(y)-d(z,x),\ d_D(x)\wedge d_D(y)-d(z,y)\big\}\\
&=\max\big\{A^{-1}\min\{d(z,y),d(w,z)\},\ d_D(x)\wedge d_D(y)-\min\{d(z,x),d(z,y)\}\big\}\\
&\geq\sup_{t>0}\max\{A^{-1}\,t,d_D(x)\wedge d_D(y)-t\}\\
&=\frac{d_D(x)\wedge d_D(y)}{1+A}.
\end{align*}

\smallskip 

(c). For each $y,z\in D$, we choose a path, denoted by $\gamma_{y,z}$, that connects $y$ to $z$ in $D$ as described in the definition of the 
uniform domain $D$. We define 
\[
E=\bigcup_{y,z\in B(x,r)\cap D_{s,r}}\gamma_{y,z}.
\] 
It follows immediately from the definition that $B(x,r)\cap D_{s,r}\subset E$.  Note that for each $y,z\in B(x,r)\cap D_s$,  $\diam(\gamma_{y,z})\leq 2Ar$ and so 
 $d(w,x)\leq d(w,y)+d(y,x)<2Ar+r$ for each  $w\in \gamma_{y,z}$. 
  This  implies that $E\subset B(x,2Ar+r)$. Moreover, by (b), $d_D(w)\geq (d_D(x)\wedge d_D(y))/(1+A)\geq s/(1+A)$ for every $y,z\in B(x,r)\cap D_s$ and hence $E\subset D_{s/(2A+1)}$. Combining the above two parts, $E\subset B(x,2Ar+r)\setminus D_{s/(2A+1)}$.

Finally, we show that $E$ is path connected. Indeed, for $w_1,w_2\in E$, we can find $y_1,z_1,y_2,z_2\in B(x,r)\cap D_{s,r}$ so that $w_1\in \gamma_{y_1,z_1},w_2\in\gamma_{y_2,z_2}$, and we can find a path contained in $\gamma_{y_1,z_1}\cup \gamma_{z_1,y_2}\cup\gamma_{y_2,z_2}$ that connects $w_1,w_2$.
\end{proof}

\subsection{Reflected Dirichlet space}\label{S:2.2} 

Suppose that $(\sE^0,\sF^0)$ is a transient  strongly local regular Dirichlet form on $L^2(D;m)$. We call $(D, d, m, \sE^0, \sF^0)$ a metric measure Dirichlet  (MMD) space. 
 We assume it is irreducible in the sense of \cite[p.43]{CF}. Denote by $\sF^0_e$ the extended Dirichlet space of $(\sE^0,\sF^0)$,
that is, $f\in \sF^0_e$ if and only if $f$ is finite $m$-a.e. on $D$
 and there exists an $ \sE^0$-Cauchy sequence $\{f_n; n\geq 1\}\subset   \sF^0$ such that $f_n\to f$ $m$-a.e. on $D$. 
 It is well known (see, e.g., \cite[Theorem 2.3.4]{CF})
 that $\sF^0\subset \sF^0_e$ and every $u\in \sF^0_e$ has a $\sE^0$-quasi-continuous $m$-version. 
 Throughout this paper, we always represent 
functions in $\sF^0_e$ by its $\sE^0$-quasi-continuous version. Denote by $\sF^0_b$ the space of bounded functions in $\sF^0$. 
For any $u\in. \sF_b^0$, there is a unique Radon measure $\mu_{\<u\>} $ on  $D$  so that 
\begin{equation}\label{e:2.1}
\int_D f (x) \mu_{\<u\>} (dx) = 2 \sE^0 ( u, uf)-\sE^0 (u^2, f) \quad \hbox{for every bounded } f\in \sF^0_e.
\end{equation}
For a general $u\in \sF^0_e$, 
take $u_n=((-n)\vee u)\wedge n$, which is in $\sF^0_e$. 
It is known   that the measure $ \mu^0_{\<u_n\>} $ is increasing in $n$.
Define. $\mu_{\<u\>} := \lim_{n\to \infty}  \mu_{\<u_n\>}$.
The measure $\mu_{\<u\>}$ is called the energy measure of $u$ and $\mu_{\<u\>}(D)=2\sE^0(u, u)$ for $u\in \sF^0_e$.
See, e.g., \cite[(4.3.15), Theorems 4.3.10, 4.3.11]{CF} for the above facts. 
Since $(\sE^0, \sF^0)$ is strongly local, 
  by \cite[Proposition 4.3.1(ii)]{CF}, the energy measure  $\mu_{\<u\>}$ has  the strong local property that 
for any  $\sE^0$-quasi-open subset $U\subset D$,
\begin{equation}\label{e:2.4}
\mu_{\<u\>}(U)=0 \quad \hbox{for any } u\in  \sF^0_{\rm loc} \hbox{ that is constant } \sE^0 \hbox{-q.e. on } U.
\end{equation}
  In particular,  
  $\mu_{\<u\>}$ has a local property in the sense for any $u, v\in \sF^0_e$, if $u=v$ $m$-a.e. on an open set $O\subset D$, then $\mu_{\<u\>}(O)=\mu_{\<v\>} (O)$.  
   For each open subset $U\subset D$, define 
\begin{eqnarray*}
\sF_{\rm loc}(U)&:=& \{f\in \mcB(U):\ \hbox{ for each relatively compact open  set $O\subset U$, there is some $u\in \sF$} \\
&& \hskip 0.9truein \hbox{ so that  $f=u$ $m$-a.e. on } O\}.
\end{eqnarray*}
 Then for any $f\in \sF_{\rm loc}(U)$,  by the strong locality of $(\sE^0, \sF^0)$, the energy measure  $\mu_{\<f\>}$ is well defined on $U$ so that for each relatively compact open $O\subset U$, 
$\mu_{\<f\>}(O)=\mu_{\<v\>}(O)$ for any $v\in \sF$ such  that $v=f$ $m$-a.e. on $O$.
Every $f\in \sF_{\rm loc}(U)$ admits an $\sE^0$-quasi-continuous version on $O$.
We always represent $f\in \sF_{\rm loc}(U)$ by its  $\sE^0$-quasi-continuous version. 
When $O=D$, we simply denote $\sF_{\rm loc}(D)$ by $\sF_{\rm loc}$.

Define  
\begin{eqnarray}
\bar \sF &:=& \left\{ f\in \sF_{\rm loc} \cap L^2(D; m):\   \mu_{\<f\>} (D)  <\infty \right \},  \label{e:RD1}\\
\bar \sE (f,f)&:=& \frac12  \mu_{\<f\> } (D) \quad\hbox{ for every }f\in \bar \sF .   \label{e:RD2} 
\end{eqnarray}
As mentioned earlier,  $(\bar \sE, \bar \sF)$  is always a Dirichlet form on $L^2(D; m)$; see \cite[Theorems 3.9 and 3.10]{C1} and \cite[Theorems 6.2.14 and 6.4.2]{CF}. It is called an active reflected Dirichlet form of $(\sF, \sE)$. In the rest of this paper, we always assume $(\bar \sE, \bar \sF)$ is regular on $L^2(\tD; m_0)$. It contains $(\sE^0, \sF^0)$ as its part Dirichlet form on $D$ by \cite[Theorem 6.6.5]{CF}. 
  Note that it follows from \eqref{e:2.4} that for any $u\in \bar \sF $ that is constant in an open set $U\subset \tD$, 
  $  \mu_{\<u\>}(U\cap D) =0$. 
Hence the regular Dirichlet form $(\bar \sE, \bar \sF)$  is strongly local on  $L^2(\tD; m_0)$ in the sense of \cite[Definition 1.3.17]{CF}. 
 Since $(\sE^0, \sF^0)$ is   irreducible and   $m_0(\partial D)=0$, 
$(\bar \sE, \bar \sF)$ is also   irreducible. 
For $f\in \bar \sF$, 
its energy measure defined by \eqref{e:2.1} with $(\bar \sE, \bar \sF)$ in place of $(\sE^0, \sF^0)$ is exactly the measure $\mu_{\< f\>}$ defined above 
for  $f$ as an element in $\sF_{\rm loc}$.  Hence we use  the same notation $\mu_{\<f\>}$ to denote the energy measure of $f\in \bar \sF$ with respect to the active  reflected Dirichlet form $(\bar \sE, \bar \sF)$.    For $f, g\in \bar \sF$, define $\mu_{\<f, g\>}:=\frac14 \left( \mu_{\<f+g\>} - \mu_{\<f-g\>}  \right)$.   We  conclude immediately from the definition \eqref{e:RD2} of $\bar \sE $ that  
\begin{equation}\label{e:em}
\mu_{\<f\>} (\partial D)=0  \quad \hbox{ for  every } f\in \bar \sF.
\end{equation}

 \begin{remark}  \label{R:2.4} \rm 
 Suppose that $(\sE^0,\sF^0)$ is a transient  strongly local regular Dirichlet form on $L^2(D;m)$.
In view of \eqref{e:2.4} and \cite[Theorems 6.6.3 and 6.6.5]{CF}, by the same reason as above, every regular representation of its active reflected Dirichlet form $(\bar \sE, \bar \sF)$ is strongly local.   \qed
\end{remark}

\subsection{Trace Dirichlet form}\label{S:2.3} 
Recall that $(D,d,m)$ is a locally compact separable metric  measure  space
and $(\sE^0,\sF^0)$ is a transient irreducible strongly local regular Dirichlet form on $L^2(D;m)$. 
Throughout this paper, we assume $(D,d)$ is an $A$-uniform for some $A>1$, 
the measure $m_0$ is a Radon measure on $(\tD, d)$ and  volume doubling, and 
the  active reflected Dirichlet form $(\bar \sE ,\bar \sF )$ of  $(\sE^0,\sF^0)$  defined by \eqref{e:RD1}-\eqref{e:RD2} is a regular Dirichlet form on $L^2(\tD;m_0)$.

For each open $O\subset \tD$, we define 
\begin{eqnarray*}
\bar{\sF}_{\rm loc}(O)&:=& \{f\in \mcB(O):\ \hbox{ for each relatively compact open  set $U\subset O$, there is some $u\in \bar\sF$} \\
&& \hskip 0.9truein \hbox{ so that  $f=u$ $m$-a.e. on } U \}.
\end{eqnarray*} 
Denote by $\bar\sF_e$ the extended Dirichlet form of $(\bar \sE, \bar \sF)$
and set 
$$
\check \sF_e:= \bar{\sF}_e |_{\partial D}.
$$ 
We always represent functions in $\bar \sF$ by their $\bar \sE$-quasi-continuous version. 

 It is well known   (see, e.g.,   \cite[Theorems 1.5.1 and 4.3.4]{CF})
that there is a continuous transient irreducible $m$-symmetric Hunt process $X^0=\{X^0_t, t\geq 0;  \bP_x, x\in D\setminus \sN\}$ on $D$ that admits no killings inside $D$ associated with the transient irreducible 
strongly local Dirichlet form  $ (\sE^0,\sF^0)$ on $L^2( D;m)$, where $\sN$ is a proper exceptional subset of $X^0$.
Similarly, there is a continuous   irreducible $m_0$-symmetric  Hunt process $\bar X=\{\bar X_t, t\geq 0; \bar \bP_x, x\in \tD \setminus \sN_1\}$  
on $\tD$  that admits no killings inside $\tD$  associated with the strongly
local irreducible
 regular Dirichlet form on $(\bar \sE ,\bar \sF )$ on $L^2(\tD;m_0)$, where $\sN_1\subset \tD$ is a a proper exceptional set  of $\bar X$.
 We call $\bar X$ the reflected diffusion process on $\tD$.  Since $(\bar \sE, \bar \sF)$ is irreducible,  $\bar X$ is either transient or recurrent by \cite[Proposition 2.1.3]{CF}.  When there is no danger of confusions, we simply denote  $\bar{\bP}_x$ by $\bP_x$. The part process of $\bar X$ killed upon hitting $\partial D:=D^*\setminus D$ has the same distribution as $X$;
see \cite[Theorem 6.6.5]{CF}.    

For a Hunt process $Y$ on a state space $E$ and $A\subset E$ a Borel subset, we define the hitting time and exit time of $A$ by $Y$ as
follows:
\[
\sigma_A:=\inf\{t\geq 0: Y_t\in A\}
\quad \hbox{and} \quad 
\tau_A:=\inf\{t\geq 0:\,Y_t\notin A\}.
\]

Suppose that $\nu$ is a smooth measure on $\partial D$ with respect to the regular Dirichlet form $(\bar \sE, \bar\sF)$ with quasi-support $\partial D$. Let $A^\nu$ be the positive continuous additive functional of $\bar X$ with Revuz measure $\nu$. 
Define $\tau_t:=\inf\{r\geq 0: A^\nu_t > t\}$. The time-changed process $\{\check X_t:=\bar X_{\tau_t}; t\geq 0\}$, which takes values in $\partial D$,
 is called a trace process of $\bar X$  on $\partial D$.  As mentioned in the Introduction, 
 the process $\check X$ is $\nu$-symmetric and its associated Dirichlet form 
 $(\check \sE, \check \sF)$ is regular on $L^2(\partial D; \nu)$, where   $(\check \sE, \check \sF)$ is given by \eqref{e:trace1}-\eqref{e:trace2}. 
It is known that  $|\mathcal{H} u (x) |=|\mathbb{E}_x [u (\bar X_{\tau_D});\tau_D<\infty]| 
 <\infty$  for $\bar \sE$-q.e. $x\in \tD$ 
 and $\mathcal{H}u \in \bar \sF_e$ for any $u\in 
\check \sF_e  $; see, e.g.,  \cite[Theorem 3.4.8]{CF}.
  We call $(\check \sE, \check \sF)$ the {\it trace Dirichlet form of $(\bar \sE, \bar \sF)$ on $L^2(\partial D; \nu)$.} 
 Denote by $(\check \sF)_e$ the extended Dirichlet space of  $(\check \sE, \check \sF)$. 
By \cite[Lemma 5.2.4]{CF}, we have 
$$
\check \sF\subset \check \sF_e \subset (\check \sF)_e.
$$  
For convenience,   we call $(\check \sE, \check \sF_e)$ the {\it trace Dirichlet space of $(\bar \sE, \bar \sF)$ on $\partial D$.}
While the trace Dirichlet form  $(\check \sE, \check \sF)$ on $L^2(\partial D; \nu )$ is dependent 
on the smooth measure $\nu$ on $\partial D$ used in the time change,
 the trace Dirichlet space $(\check \sE, \check \sF_e)$ depends only on $(\bar \sE, \bar \sF)$. 

 \medskip

\section{Trace theorems}\label{S:3} 

Recall that the metric space $(D, d)$ is an $A$-uniform for some $A>1$,  $m_0$ is a Radon measure with full support on $\tD$ with $m_0(\partial D)=0$ and is volume doubling, and the active reflected Dirichlet form   $(\bar \sE, \bar \sF)$ of \eqref{e:RD1}-\eqref{e:RD2} is   regular  on $L^2(\tD; m_0)$, which is strongly local. 

Let $\Psi$  be  a continuous bijection from $(0,\infty)$ to $(0,\infty)$  that satisfies \eqref{eqnpsi}.

\begin{definition}\label{D:3.1} 
 Let $\sigma$ be a Radon measure on $\partial D$ and $\Theta_{\Psi, \sigma} (x,r) $ be the scale function   on
 $\partial D\times (0, \infty)$    defined in \eqref{e:1.12}.
 We say $\Theta_{\Psi, \sigma} (x,r) $  satisfies the lower scaling condition  {\rm (LS)} 
   if   there are  constants $C,\beta\in (0,\infty)$ so that 
\begin{equation}\label{e:3.1}
\frac{\Theta_{\Psi, \sigma} (x,R)}{\Theta_{\Psi, \sigma} (x,r)}\geq C \Big(\frac{R}{r} \Big)^\beta\quad  \hbox{ for  every } x\in \partial D \hbox{ and }  0<r<R<\diam(\partial D).
\end{equation} 
\end{definition}
\medskip

Note that condition \eqref{e:3.1} is equivalent to that $r\to \Theta_{\Psi, \sigma} (x,r)$ satisfies uniform reverse doubling condition in the sense that there are some constants $\lambda_0 >1$ and $c_0>1$ so that 
\begin{equation}\label{e:3.2}
\Theta_{\Psi, \sigma} (x,  \lambda_0  r) \geq c_0 \Theta_{\Psi, \sigma} (x,r)\quad \hbox{ for  every } x\in \partial D \hbox{ and }0<r< \diam(\partial D)/\lambda_0.
\end{equation}
Note that under the (VD) assumption on $\sigma$, then for any $x\in \partial D$ and $ r>0$,  
 $$
\frac{\Theta_{\Psi, \sigma} (x,2r)}{\Theta_{\Psi, \sigma} (x,r)}=\frac{\Psi (2r)}{\Psi (r)} \frac{ V_\sigma  (x, 2r)}{ V_\sigma (x, r)} \frac{V(x, r)}{V(x, 2r)} 
  \leq C_\Psi 2^{\beta}  C;
  $$
  that is, $\Theta_{\Psi, \sigma}  (x, r)$ has doubling property in $r>0$ uniformly in $x\in \partial D$.

\medskip

Recall that for  a  Radon measure  $\sigma$ with full support on $\partial D$, the  Beppo-Levi type space $\dot \Lambda_{\Psi,  \sigma} $  and 
the Besov type space $\Lambda_{\Psi,  \sigma} $ are defined in Definition \ref{def34}. 
In this section, we  focus on characterizing the trace Dirichlet space $(\check \sE, \check \sF_e)$ in terms of 
the Beppo-Levi type space $\dot \Lambda_{\Psi,  \sigma} $  and 
the Besov type space $\Lambda_{\Psi,  \sigma} $ for some suitable Radon measure $\sigma$ on $\partial D$.

\subsection{Restriction theorems}\label{sec32}

In this subsection, we establish some restriction theorems. 

 \begin{definition}\rm 
\begin{enumerate} [\rm (i)]
 \item  We say  the MMD space  $(D, d, m,  \sE^0, \sF^0)$ satisfies the Poincar\'e inequality PI$(\Psi;D)$
  if  there are constants $C_p>0$ and  $A_p\geq 1$ so that for all  $x\in D$, $r\in (0, d_D(x)/A_p)$ and all $f\in \sF_{\rm loc}\big(B(x,A_pr)\big)$, 
 \begin{equation}\label{e:PI}
\int_{B(x,r)}\big(f-[f]_{B(x,r)}\big)^2dm_0\leq C_p\Psi(r)\, \mu_{\<f\>}  ( B(x,A_pr)),
\end{equation}
where for each Borel $E\subset \tD$, we write $[f]_E=\frac{1}{m_0(E)}\int_E fdm_0$.

 \item We say that the MMD space  $(\tD,d,m_0,\bar{\sE},\bar{\sF})$ satisfies the Poincar\'e inequality PI($\Psi$), if there exist constants $C_P$ and  $A_p\geq 1$ so that \eqref{e:PI} holds for all  $x\in \tD$, $r>0$ and all $f\in \bar\sF_{\rm loc}\big(B(x,A_pr)\big)$.
  \end{enumerate}
   \end{definition}

 \begin{remark}\label{remark34}  \rm 
We can show by the same proof as that for   \cite[Theorem 5.3]{Mathav} that 
if $(D, d, m, \sE^0, \sF^0)$ satisfies the Poincar\'e inequality PI$(\Psi;D)$, then $(\tD,d,m_0,\bar{\sE},\bar{\sF})$ satisfies the Poincar\'e inequality PI($\Psi$).   \qed
\end{remark} 

\begin{theorem}\label{thmrestriction1}
Suppose that \rm PI$(\Psi;D)$ holds for  $(D,d,m,  \sE^0, \sF^0)$,  and 
 $\sigma$ is  a     Radon measure  with full support on $\partial D$  satisfying {\rm (VD)} property so that  {\rm (LS)} holds for     $\Theta_{\Psi, \sigma}$.
There exists a constant $C\in (0,\infty)$ such that 
\begin{equation}\label{e:3.4}
\lb f|_{\partial D} \rb_{\Lambda_{\Psi,  \sigma} }^2\leq C\bar \sE (f,f)\quad\hbox{ for each }f\in C(\tD)\cap \bar \sF.
\end{equation}
Consequently,  
\begin{equation}\label{e:3.5}
\check \sF_e \subset \dot \Lambda_{\Psi,  \sigma}  \quad \hbox{ and } \quad 
\lb f|_{\partial D} \rb_{\Lambda_{\Psi,  \sigma} }^2\leq C\check  \sE  (f,f)\quad\hbox{ for each } f \in \check \sF_e.
\end{equation}
\end{theorem}

For its proof,  we need some lemmas.

\begin{lemma}\label{lemma36}
Suppose that {\rm PI$(\Psi;D)$} holds for $(D, d, m,  \sE^0, \sF^0)$. For each $\eta\in [1,\infty)$ and $\delta\in(0,\frac1{4A}]$, there is a constant $C\in (0,\infty)$ such that  
\[
\int_{x\in D_{\delta r,r}}\frac{1}{V(x,r)}\int_{y\in D_{\delta r,r}\atop d(x,y)<\eta r}\big(f(x)-f(y)\big)^2m_0(dy)m_0(dx)\leq C\,\Psi(r)\mu_{\<f\>}(D_{\eta_1r,\eta_2r}),
\]
for each $f\in \sF_{\rm loc}(D)$ and $r\in \big(0,\diam(D)/2\big)$, where $
\eta_1=\frac{\delta}{2(1+A)}$ and $\eta_2=2(1+A)(\eta+2)$.
\end{lemma}

\begin{proof}
First, we claim that for each $f\in \sF_{\rm loc}(D),r\in \big(0,\diam(D)/2\big)$ and $p\in\partial D$, 
\begin{eqnarray}
		&&\int_{x\in B(p,2r)\cap D_{\delta r}}\frac{1}{V(x, r)}\int_{y\in D_{\delta r,r}\atop d(x,y)<\eta r}\big(f(x)-f(y)\big)^2m_0(dy)m_0(dx)  \nonumber \\
		&\leq &C_1\Psi(r)\mu_{\<f\>}\big(B(p,\eta_2r)\cap D_{\eta_1r} \big).   \label{e:3.6}
\end{eqnarray}
	
	\medskip
 	 
Indeed,  by Lemma \ref{lemma21} (b), we can find a connected subset $E\subset D$ such that 
		\begin{equation}\label{eqn31}
		\begin{aligned}
				B\big(p,(\eta+2)r\big)\cap D_{\delta r}\subset E&\subset B\big(p,(2A+1)(\eta+2)r\big)\cap D_{2\eta_1r}\\
				&\subset B\big(p,(\eta_2-\eta_1)r\big)\cap D_{2\eta_1r}.
		\end{aligned}
		\end{equation}
Next,  by Lemma \ref{lemma2path}(a), there is  a finite subset $\{x_i\}_{i=1}^N\subset E$ such that  $E\subset\bigcup_{i=1}^NB(x_i,\eta_1r/(3A_p))$, and $N$ has an upper bound depending only on $\eta_1,\eta_2,A_p$ and the parameter of (VD).
 We define 
\[
E_* :=\bigcup_{i=1}^N B(x_i,\eta_1 r / A_p ) \quad\hbox{ and } \quad E^*:=\bigcup_{i=1}^N B(x_i,\eta_1r),
\]
so that by \eqref{eqn31}, 
\begin{equation}\label{eqn32}
	B\big(p,(\eta+2)r\big)\cap D_{\delta r}\subset E_* \subset E^*\subset B(p,\eta_2r)\cap D_{\eta_1r}.
\end{equation}
For $x_i$ and $x_j$ with  $d(x_i,x_j)<\frac{2\eta_1}{3A_p}r$,  $B(x_i,\eta_1r/(3A_p))\subset B(x_j,\eta_1 r / A_p )$, and so 
\begin{equation}\label{eqn33}
	m_0\big(B(x_i,\eta_1 r / A_p )\cap B(x_j,\eta_1 r / A_p )\big)\geq m_0\big(B(x_i,\eta_1r/(3A_p))\big)\geq C_2'\,m_0\big(B(x_i,\eta_1 r / A_p )\big),
		\end{equation}
		where $C_2'$ depends only on  the bound in  (VD) for the measure $m_0$. 
		Using PI($\Psi$;D) and \eqref{eqn33}, for $x_i,x_j$ such that $d(x_i,x_j)<\frac{2\eta_1}{3A_p}r$, we have 
		\begin{align*}
&\quad\ \big|[f]_{B(x_j,\eta_1 r / A_p )}-[f]_{B(x_i,\eta_1 r / A_p )}\big|\\
&\leq m_0\Big(B(x_i,\eta_1 r / A_p )\cap B(x_j,\eta_1 r / A_p  \Big)^{-1/2}
       \left(     \int_{B(x_i,\eta_1 r / A_p )\cap B(x_j,\eta_1 r / A_p )}\big([f]_{B(x_j,\eta_1 r / A_p )}-f(x)\big)^2m_0(dx)\right)^{1/2}\\
	&\qquad\ + m_0\Big(B(x_i,\eta_1 r / A_p )\cap B(x_j,\eta_1 r / A_p  \Big)^{-1/2}
	\left( \int_{B(x_i,\eta_1 r / A_p )\cap B(x_j,\eta_1 r / A_p )}\big([f]_{B(x_i,\eta_1 r / A_p )}-f(x)\big)^2m_0(dx) \right)^{1/2}\\
&\leq  \Big(C_2' V (x_j,\eta_1 r / A_p)\Big)^{-1/2}
              \left(  \int_{ B(x_j,\eta_1 r / A_p )}\big([f]_{B(x_j,\eta_1 r / A_p )}-f(x)\big)^2m_0(dx) \right)^{1/2}\\
			&\quad\ + \Big( C_2'V (x_i,\eta_1 r / A_p )\Big)^{-1/2}
			\left( \int_{ B(x_i,\eta_1 r / A_p )}\big([f]_{B(x_i,\eta_1 r / A_p )}-f(x)\big)^2m_0(dx) \right)^{1/2}\\
&\lesssim \sqrt{\frac{\Psi(r)\mu_{\<f\>}\big(B(x_i,\eta_1r)\big)}{ V (x_i,\eta_1 r / A_p ) }}+\sqrt{\frac{\Psi(r)\mu_{\<f\>}\big(B(x_j,\eta_1r)\big)}
 {V (x_j,\eta_1 r / A_p )}},
\end{align*}
		where in the last inequality we use \eqref{eqnpsi}. Noticing that $E$ is connected and $\{B(x_i,\eta_1r/(3A_p))\}_{i=1}^N$ is an open cover of $E$, for each $i=1,2,\cdots,N$, we can find a path $i_0=1,i_1,\cdots,i_L=i$ such that $L\leq N$ and $d(x_{i_k},x_{i_{k+1}})<\frac{2\eta_1r}{3A_p}$ for each $k=1,2,\cdots,L-1$. Thus  
		\begin{equation}\label{eqn34}
			\big|[f]_{B(x_i,\eta_1 r / A_p )}-[f]_{B(x_1,\eta_1 r / A_p )}\big|\lesssim \sqrt{\frac{\Psi(r)\mu_{\<f\>}(E^*)}{V (x_i,\eta_1 r / A_p ) }}\quad\hbox{ for each }i=1,2,\cdots,N,
		\end{equation} 
where we used the fact that $m_0\big(B(x_i,\eta_1 r / A_p )\big)\asymp m_0\big(B(x_j,\eta_1 r / A_p )\big)$ due to (VD) of $m_0$.

Finally, by using (VD), PI($\Psi;D$), \eqref{eqn32} and \eqref{eqn34}, we get
\begin{eqnarray*}
&&\int_{x\in B(p,2r)\cap D_{\delta r}}\frac{1}{V(x,r)}\int_{y\in D_{\delta r,r}\atop d(x,y)<\eta r}\big(f(x)-f(y)\big)^2m_0(dy)m_0(dx)\\
&\lesssim &\frac{1}{m_0(E_*)}\int\int_{x,y\in E_*}\big(f(x)-f(y)\big)^2m_0(dy)m_0(dx)\\
&\leq&2\int_{E_*}\big([f]_{B(x_1,\eta_1 r / A_p )}-f(x)\big)^2m_0(dx)\\
&\leq&4\sum_{i=1}^N \Big(\int_{B(x_i,\eta_1 r / A_p )}\big([f]_{B(x_i,\eta_1 r / A_p )}-f(x)\big)^2m_0(dx) \\
&& +m_0\big(B(x_i,\eta_1 r / A_p )\big)\big([f]_{B(x_i,\eta_1 r / A_p )}-[f]_{B(x_1,\eta_1 r / A_p )}\big)^2\Big)\\
& \lesssim&\Psi(r)\mu_{\<f\>}(E^*).
\end{eqnarray*}
This proves the  Claim   \eqref{e:3.6} as $E^*\subset B(p,\eta_2r)\cap D_{\eta_1r}$ by \eqref{eqn32}. 

\medskip
	
To finish the proof of the lemma, we pick a minimal $r$-net of $\partial D$, i.e., we pick $\{p_i\}_{i=1}^N\subset \partial D$ ($N$ can be infinity) so that $d(p_i,p_j)\geq r$ for $i\neq j$ and $d(x,\{p_i\}_{i=1}^N)<r$ for each $x\in \partial D$. One can see that $D_{\delta r,r}\subset \bigcup_{i=1}^N B(p_i,2r)$, and in addition, since $(\tD,d)$ is metric doubling, each $x\in D_{\delta r,r}$ is covered by finitely many balls of the form $B(p_i,\eta_2r)$, hence
 	\begin{eqnarray*}
		&&\int_{x\in D_{\delta r,r}}\frac{1}{V(x,r)}\int_{y\in D_{\delta r,r}\atop d(x,y)<\eta r}\big(f(x)-f(y)\big)^2m_0(dy)m_0(dx)\\
		&\leq &\sum_{i=1}^N\int_{x\in D_{\delta r,r}\cap B(p_i,2r)}\frac{1}{V  (x, r)}\int_{y\in D_{\delta r,r}\atop d(x,y)<\eta r}\big(f(x)-f(y)\big)^2m_0(dy)m_0(dx)\\
		&\lesssim &\sum_{i=1}^N C_1\Psi(r)\mu_{\<f\>}\big(B(p_i,\eta_2r)\cap D_{\eta_1r}\big)\\
		&\lesssim &\Psi(r)\mu_{\<f\>}(D_{\eta_1r,\eta_2r}),
	\end{eqnarray*}
 	where all the constants of $\lesssim$ are independent of $f$ and $r$. 
\end{proof}

\begin{lemma}\label{lemma37}
Suppose that \rm PI$(\Psi;D)$ holds for  $(D,d,m, \sE^0, \sF^0)$,  and 
$\sigma$ is a Radon measure  with full support on $\partial D$ having  {\rm (VD)}  property so that {\rm (LS)} holds for $\Theta_{\Psi, \sigma}$. For $0<r<\diam(\partial D)/2$ and $f\in L^1(\tD;m_0)$, define a function $f_r\in \mcB (\partial D)$ by
\[
f_r(x):=\frac{1}{m_0\big(B(x,r)\cap D_{\theta  r}\big)}\int_{B(x,r)\cap D_{\theta  r}}f(y)m_0(dy)\quad\hbox{ for every }x\in \partial D,
\]
 where $\theta=\frac1{4A}$.
Then, there is $C\in (0,\infty)$, which is independent of $f$ and $r>0$ such that for every $f\in C(\tD)\cap\bar\sF$  
\begin{equation} \label{eqn35}
\lb f_r\rb_{\Lambda_{\Psi,  \sigma} }^2  \leq C\mu_{\<f\>}(D_{\eta_1r,\eta_2r}),
\end{equation}
and for every $k\geq 1$, 
\begin{equation} \label{eqn36}
\int_{x\in\partial D}\frac{\big(f_{\theta^kr}(x)-f_{\theta^{k+1}r}(x)\big)^2}{\Theta_{\Psi, \sigma} (x,r)} 
\sigma(dx)\leq C\,\theta^{\beta k}\,\mu_{\<f\>}\big(D_{\theta^{k+1}\eta_1r,\theta^k\eta_2r}\big), 
\end{equation}
where $\eta_1=\frac{1}{8A(1+A)}$, $\eta_2=10(1+A)$ and $\beta >0 $ is the exponent  in condition {\rm (LS)}  for 
$\Theta_{\Psi, \sigma}$. 

\end{lemma}
\begin{proof}
\eqref{eqn35} follows from 
\begin{align*}
	&\int_{x\in\partial D}\int_{y\in \partial D\atop d(x,y)<r}\frac{\big(f_r(x)-f_r(y)\big)^2}{\Theta_{\Psi, \sigma} (x,r)V_\sigma (x,r)}\sigma(dy)\sigma(dx)\\
	\lesssim &\int_{x\in\partial D}\int_{y\in \partial D\atop d(x,y)<r}\int_{z\in B(x,r)\cap D_{\theta r}}\int_{w\in B(y,r)\cap D_{\theta r}}\frac{\big(f(w)-f(z)\big)^2}{\Theta_{\Psi, \sigma} (x,r)V_\sigma (x,r)V (x,r)^2}m_0(dw)m_0(dz)\sigma(dy)\sigma(dx)\\
	=&\int_{x\in\partial D}\int_{y\in \partial D\atop d(x,y)<r}\int_{z\in B(x,r)\cap D_{\theta r}}\int_{w\in B(y,r)\cap D_{\theta r}}\frac{\big(f(w)-f(z)\big)^2}{\Psi(r){ V_\sigma (x,r) }^2{ V (x,r) }}m_0(dw)m_0(dz)\sigma(dy)\sigma(dx)\\
	\leq &\int_{z\in D_{\theta r,r}}\int_{w\in D_{\theta r,r}\atop d(z,w)<3r}\big(f(w)-f(z)\big)^2\Big(\int_{x\in \partial D\atop d(x,z)<r}\int_{y\in \partial D\atop d(x,y)<r}\frac1{\Psi(r)V_\sigma (x,r)^2V(x,r)}\sigma(dy)\sigma(dx)\Big)m_0(dz)m_0(dw)\\
	\lesssim &\int_{z\in D_{\theta r,r}}\int_{w\in D_{\theta r,r}\atop d(z,w)<3r}\frac{\big(f(z)-f(w)\big)^2}{\Psi(r)V(z,r)}m_0(dz)m_0(dw)\\		
	\lesssim&\,\mu_{\<f\>}(D_{\eta_1r,\eta_2r}),
\end{align*}
	where we use Lemma \ref{lemma36} in the last inequality, and we use (VD) for $\sigma$, $m_0$ and Lemma \ref{lemma21}(a) a few times.

Next, \eqref{eqn36} follows from the following estimates
\begin{align*}
	&\int_{x\in\partial D}\frac{\big(f_{\theta^kr}(x)-f_{\theta^{k+1}r}(x)\big)^2}{\Theta_{\Psi, \sigma} (x,r)}\sigma(dx)\\
	\lesssim&\int_{x\in\partial D}\int_{z\in B(x,\theta^kr)\cap D_{\theta^{k+1}r}}\int_{w\in B(x,\theta^{k+1}r)\cap D_{\theta^{k+2}r}}\frac{\big(f(w)-f(z)\big)^2}{\Theta_{\Psi, \sigma} (x,r)V(x,\theta^kr)^2}m_0(dw)m_0(dz)\sigma(dx)\\
	\lesssim&\int_{z\in D_{\theta^kr,\theta^{k+2}r}}\int_{w\in D_{\theta^kr,\theta^{k+2}r}\atop d(w,z)\leq 2\theta^kr}\int_{x\in \partial D\atop d(x,z)<\theta^kr}\frac{\big(f(w)-f(z)\big)^2}{\Theta_{\Psi, \sigma}(x,r)V(x,\theta^kr)^2}\sigma(dx)m_0(dw)m_0(dz)\\
	\lesssim&\sup_{x\in \partial D}\frac{\Theta_{\Psi, \sigma} (x,\theta^kr)}{\Theta_{\Psi, \sigma}(x,r)}\,\int_{z\in D_{\theta^kr,\theta^{k+2}r}}\int_{w\in D_{\theta^kr,\theta^{k+2}r}\atop d(w,z)\leq 2\theta^kr}\frac{\big(f(w)-f(z)\big)^2}{V (x,\theta^kr) \Psi(\theta^kr)}m_0(dw)m_0(dz)\\
	\lesssim&\theta^{\beta k}\,\mu_{\<f\>}\big(D_{\theta^{k+1}\eta_1r,\theta^k\eta_2r}\big),
\end{align*}
where $\beta$ is the constant of condition {\rm (LS)}, and we use Lemma \ref{lemma36} in the last inequality. 
\end{proof}

\begin{theorem}\label{thmrestriction2}
Suppose that  $\sigma$ is  a  Radon measure on $\partial D$ that satisfies condition {\rm (LS)}, and {\rm PI$(\Psi;D)$} holds for  $(D, d, m,  \sE^0, \sF^0)$. 
Then  measure $\sigma$ does not charge sets of zero $\bar \sE$-capacity.  
    Moreover,  if  $\sup_{x\in \partial D}\Theta_{\Psi, \sigma} (x,r_0)<\infty$ for some $r_0\in (0,\diam(\partial D)/2)$, then there is $C>0$ such that
\[
\|f|_{\partial D}\|_{L^2(\partial D;\sigma)}\leq C \bar \sE_1 (f, f)^{1/2} \quad\hbox{ for each }f\in \bar \sF \cap C(\tD),
\]
where $\bar \sE_1 (f, f):=  \bar\sE (f,f)+\|f\|^2_{L^2(\tD;m_0)}$.
Consequently, 
$$
\|f|_{\partial D}\|_{L^2(\partial D;\sigma)}\leq C \bar \sE_1 (f, f)^{1/2} \quad\hbox{ for each }f\in \bar \sF .
$$
\end{theorem}

\begin{proof}
Note that by the volume doubling property of $m_0$ and $\sigma$, for every $r_0\in (0,\diam(\partial D)/2)$, $x\mapsto \Theta_{\Psi, \sigma} (x,r_0)$ is locally bounded on $\partial D$.

We choose $\theta,\eta_1,\eta_2$ and $f_r$ the same as in Lemma \ref{lemma37}. For $f\in \mcB(E)$ and $\mu$ being a Radon measure on $E$, we let $f \mu$ be the measure defined by $(f \mu) (F) :=\int_Ff(x)\mu(dx)$. Let $\psi$ be an arbitrary Borel measurable function on $ \partial D$ with compact  support such that $0\leq \psi\leq 1$. We have  for $f\in \bar \sF \cap C(\tD)$, 
\begin{align}\label{eqn37}
\begin{split}
&\quad\ \|f_{r_0}\|_{L^2(\partial D;\psi\sigma)}^2\\
&=\int_{x\in\partial D}\Big(\frac{1}{m_0\big(B(x,r_0)\cap D_{\theta r_0}\big)}\int_{z\in B(x,r_0)\cap D_{\theta r_0}}f(z)m_0(dz)\Big)^2\psi(x)\sigma(dx)\\
			&\leq \int_{x\in\partial D}\frac{1}{m_0\big(B(x,r_0)\cap D_{\theta r_0}\big)}\int_{z\in B(x,r_0)\cap D_{\theta r_0}}|f(z)|^2m_0(dz)\psi(x)\sigma(dx)\\
			&\lesssim \int_{z\in D_{\theta r_0,r_0}}f^2(z)\Big(\int_{x\in \partial D\atop d(x,z)<r_0}\frac{\psi(x)}{V(x,r_0)}\sigma(dx)\Big)m_0(dz)\\
			&\lesssim  \int_{z\in D_{\theta r_0,r_0}}f^2(z) \, \sup_{x\in \partial D\atop d(x,\operatorname{supp}(\psi))<r_0}\frac{V_\sigma (x,r_0)}{V (x,r_0)}\, m_0(dz)\\
			&\leq\|f\|^2_{L^2(\tD;m_0)}\cdot \frac{1}{\Psi(r_0)}\cdot \sup_{x\in \partial D\atop d(x,\operatorname{supp}(\psi))<r_0}\Theta_{\Psi, \sigma} (x,r_0),
		\end{split}
	\end{align}
    where all constants of `$\lesssim$' are independent of $f$ and $\psi$. Moreover, by \eqref{eqn36},
\begin{equation}\label{eqn38}
\begin{aligned}
\|f_{\theta^kr_0}-f_{\theta^{k+1}r_0}\|^2_{L^2(\partial D;\psi\sigma)}&\lesssim \Big(\sup_{x\in \partial D\atop d(x,\operatorname{supp}(\psi)<r_0)}\Theta_{\Psi, \sigma} (x,r_0)\Big)\,\theta^{\beta k}\,\mu_{\<f\>}(D_{\theta ^{k+1}\eta_1r,\theta ^k\eta_2r})\\
&\leq\Big(\sup_{x\in \partial D\atop d(x,\operatorname{supp}(\psi)<r_0)}\Theta_{\Psi, \sigma} (x,r_0)\Big)\,\theta ^{\beta k}\bar \sE (f,f),
\end{aligned}
\end{equation}
where all constants of `$\lesssim$' are independent of $f$ and $\psi$. Combining \eqref{eqn37} and \ref{eqn38}, we get 
\begin{equation}\label{eqn39}
\begin{aligned}
\|f|_{\partial D}\|_{L^2(\partial D;\psi\sigma)}&\leq\|f_{r_0}\|_{L^2(\partial D;\psi\sigma)}+\sum_{k=0}^\infty\|f_{\theta^kr_0}-f_{\theta^{k+1}r_0}\|_{L^2(\partial D;\psi\sigma)}\\
&\lesssim \Big(\sup_{x\in \partial D\atop d(x,\operatorname{supp}(\psi)<r_0)}\Theta_{\Psi, \sigma} (x,r_0)\Big)^{1/2}
\bar \sE_1 (f, f)^{1/2}.
\end{aligned}
\end{equation}
In particular, if $\sup_{x\in \partial D}\Theta(x,r_0)<\infty$, we can take $\psi\equiv 1$ so that $\|f|_{\partial D}\|_{L^2(\partial D;\sigma)}
	\lesssim \bar \sE_1 (f, f)^{1/2}$. 
	
	Next, for each $f$ with compact support, we choose $\psi$ so that $\psi\equiv 1$ on the support of $f$. Then, by \eqref{eqn39}
	\begin{align*}
	\int_{\partial D} g(x)f(x)\sigma(dx)&=\int_{\partial D} g(x)f(x)\psi(x)\sigma(dx)\\
	&\leq \|f\|_{L^2(\partial D;\psi\sigma)}\|g\|_{L^2(\partial D;\psi\sigma)}
	\lesssim \bar \sE_1 (f, f)^{1/2} \, \bar \sE_1 (g, g)^{1/2}
	\end{align*}
	for each $g\in C_c(\tD)\cap\bar \sF $. Hence, by \cite[Lemma 2.2.3]{FOT}, $f(x)\sigma(dx)$ does not charge sets of zero capacity. Since $C_c(\tD)\cap \bar \sF $ is dense in $C_c(\tD)$, we see that $\sigma$ does not charge sets of zero capacity. 
\end{proof}

\begin{proof}[Proof of Theorem \ref{thmrestriction1}]
Combining \eqref{eqn35} and \eqref{eqn36} from Lemma \ref{lemma37}, and noticing that $f_r(x)\to f(x)$ pointwise on $\partial D$, we have 
for $r<\diam(\partial D)/2$, 
\[
\sqrt{\int_{x\in\partial D}\int_{y\in \partial D\atop d(x,y)<r}\frac{\big(f(x)-f(y)\big)^2}{\Theta_{\Psi, \sigma} (x,r)V_\sigma (x,r)}\sigma(dy)\sigma(dx)}\leq C\sum_{k=0}^\infty \theta^{\beta k/2}\sqrt{\mu_{\<f\>}\big( D_{\theta^{k+1}\eta_1r,  \theta^k\eta_2r}\big)},
\]
where $C,\theta,\eta_1,\eta_2$ are the same as in Lemma \ref{lemma37}. Moreover, by (VD) property of $\sigma$,
\[
\frac{1}{\Theta_{\Psi,\sigma}(x,r)V_\sigma (x,r)}\lesssim \sum_{i\in\mathbb{Z}:\theta^i>r} \frac{1}{\Theta_{\Psi,\sigma}(x,\theta^i)V_\sigma (x, \theta^i)}
\] 
Then, the desired estimate follows from the following inequality, where we assume that $\partial D$ is unbounded.
\begin{eqnarray*}
	\lb f|_{\partial D} \rb_{\Lambda_{\Psi,  \sigma} }^2
	&=&\int_{x\in \partial D}\int_{y\in \partial D}\frac{\big(f(x)-f(y)\big)^2}{\Theta_{\Psi, \sigma} \big(x,d(x,y)\big) V_\sigma (x,d(x,y))}\sigma(dy)\sigma(dx)\\
	&\lesssim& \sum_{i=-\infty}^\infty  \int_{x\in\partial D}\int_{y\in \partial D\atop d(x,y)<\theta^i}\frac{\big(f(x)-f(y)\big)^2}{\Theta_{\Psi, \sigma} (x,  \theta^i  ){V_\sigma (x,\theta^i ) }}\sigma(dy)\sigma(dx)\\
	&=& \left\|  \, \sqrt{\int_{x\in\partial D}\int_{y\in \partial D\atop d(x,y)<\theta^i}\frac{\big(f(x)-f(y)\big)^2}{\Theta_{\Psi, \sigma} (x, \theta^i)V_\sigma (x,\theta^i ) }\sigma(dy)\sigma(dx)} \,  \right\|^2_{l^2(   \Z) \,  \hbox{\tiny in } i} \\
	& \lesssim& \Big\|\sum_{k=0}^\infty \theta^{\beta k/2}\sqrt{\mu_{\<f\>}\big(D_{\theta^{k+i+1}\eta_1, \, \theta^{k+i}\eta_2}\big)}\Big\|^2_{l^2(   \Z) \,  \hbox{\tiny in } i} \\
	& \lesssim& \Big(\sum_{k=0}^\infty \theta^{\beta k/2}\Big\| \sqrt{\mu_{\<f\>}\big(D_{\theta^{k+i+1}\eta_1, \, \theta^{k+i}\eta_2}\big)}\Big\|^2_{l^2(   \Z) \,  \hbox{\tiny in } i}\\
	& =& \big(\frac{1} { 1-\theta^{\beta/2}} \big)^2
	 \Big\| \sqrt{\mu_{\<f\>}\big(D_{\theta^{i+1}\eta_1 ,\theta^{i}\eta_2 }\big)}\Big\|^2_{l^2(\Z) \,  \hbox{\tiny in } i}\\
& \lesssim & \mu_{\<f\>}(D)\leq \bar\sE(f,f),
\end{eqnarray*}
where we use the notation $\|c_i\|_{l^2(-\infty<i<\infty)}$ to denote the norm $\big(\sum_{i=-\infty}^\infty c_i^2\big)^{1/2}$. For the case that $\partial D$ is bounded, we use summation $\sum_{i=i_0}^\infty$ instead for some suitable $i_0\in \Z$ so that $\theta^{i_0}<\diam(\partial D)/2\leq \theta^{i_0-1}$ and the same estimate holds. 

Under the condition of this theorem,  $\sigma$ is a smooth measure on $\partial D$
with respect to $(\bar \sE, \bar \sF)$ by Theorem \ref{thmrestriction2}. 
 Since $C(\tD)\cap \bar \sF $ is $\bar \sE_1^{1/2}$-dense in $\bar \sF$,  it follows that  \eqref{e:3.4} holds for all $f\in \bar \sF$.
 By \cite[Theorem 2.3.4]{CF}, for every $f\in \bar \sF_e$, for any $\bar\sE$-Cauchy sequence $\{f_n; n\geq 1\}\subset \bar \sF$
 so that $f_n \to f$ $m_0$-a.e. on $\tD$, there is subsequence $\{n_l; l\geq 1\}$ so that $ f_{n_l}$ converges to $f$ $\bar \sE$-quasi-everywhere on $\tD$
 to $f$. It follows that $ f_{n_l}$ converges to $f$ $\sigma$-everywhere on $\tD$ and $ f_{n_l}|_{\partial D}$ is $\lb \cdot \rb_{\Lambda_{\Psi,  \sigma} }$-Cauchy.
Hence $f|_{\partial D}\in \dot \Lambda_{\Psi,  \sigma} $ and \eqref{e:3.4} holds for $f|_{\partial D}$. For $u\in \check \sF_e$,
$$
 \check \sE(u, u) =\bar \sE ({\mathcal H}  u,  {\mathcal H}  u) = \inf\{ \bar \sE (f, f): f\in \sF_e, f=u \hbox{ $\bar \sE$-q.e. on } \partial D\}.
$$
 We conclude that \eqref{e:3.5} holds. We  emphasize that property \eqref{e:3.5} is not used in the proof of Theorem \ref{thmrestriction2}.
\end{proof}

It follows from Theorems \ref{thmrestriction1} and \ref{thmrestriction2}  that under the assumption of Theorem \ref{thmrestriction2}, 
\begin{equation} \label{e:3.16}
\bar \sF|_{\partial D}
\subset \dot \Lambda_{\Psi,  \sigma}  \cap L^2(\partial D; \sigma)= \Lambda_{\Psi,  \sigma}  .
\end{equation}

\subsection{Extension theorem}\label{S:3.2}
We consider an extension theorem.  The approach is based on the Whitney cover. 

\medskip

For $f\in L^1(\tD;m_0)$, the support ${\rm supp} [f] $ of $f$ is defined to be the smallest closed subset $K$ of $\tD$ such that $\int_{\tD\setminus K}|f(x)|m_0(dx)=0$.

\begin{definition}\rm 
\begin{enumerate} [\rm (i)]
\item We say that the  MMD space  $(D, d, m, \sE^0, \sF^0)$ satisfies Cap$_\leq$($\Psi;D$) 
if there exist constants $C,A_1>1$ such that for all  $x\in D$ and $ 0<r<d_D(x)/A_1$, 
$$
 {\operatorname{Cap}}\big(B(x,r), B(x,A_1r)\big)\leq C\,\frac{V(x,r)}{\Psi(r)},
$$
where for open sets $O_1, O_2\subset  D$ with $\overline O_1\subset  O_2$, 
\[
 {\operatorname{Cap}}(O_1, O_2):=\inf\{ \sE(f,f):\, f\in \sF \hbox{ with }  f=1\text{ on   } O_1 \hbox{ and }  \operatorname{supp}[f]\subset O_2\}.
\]

\item We say that the  MMD space  $(\tD,d,m_0,\bar\sE,\bar\sF)$ satisfies the (relative) capacity upper bound estimate  Cap$_\leq$($\Psi$)
if there exist constants $C,A_1,A_2>1$ such that for all  $x\in \tD$ and $0<r<\diam( D)/A_2$, 
\[
\overline{\operatorname{Cap}}\big(B(x,r), B(x,A_1r)\big)\leq C\,\frac{V (x,r)}{\Psi(r)},
\] 
where for open sets $O_1, O_2\subset \tD$ with $\overline O_1\subset  O_2$, 
\[
\overline{\operatorname{Cap}}(O_1, O_2):=\inf\{\bar\sE(f,f):\, f\in \sF \hbox{ with }  f=1\text{ on   } O_1 \hbox{ and }  \operatorname{supp}[f]\subset O_2\}.
\] 

\item We say that the  MMD space  $(\tD,d,m_0,\bar\sE,\bar\sF)$ satisfies the (relative) capacity  estimate Cap($\Psi$) if there exist constants $C,A_1,A_2>1$ such that for all  $x\in \tD$ and $0<r<\diam(D)/A_2$, 
\[
 C^{-1}\frac{V(x,r)}{\Psi(r)} \leq 
 \overline{\operatorname{Cap}}\big(B(x,r), B(x,A_1r)\big)\leq C\frac{V(x,r)}{\Psi(r)}.
\] 
\end{enumerate}
\end{definition}

 \medskip

\begin{definition}\label{def39}
Let $\varepsilon\in (0,1/2)$. A countable collection $\mathscr{R}:=\{B(x_i,r_i):x_i\in D,r_i>0,i\in I\}$ of open balls is said to be 
an $\varepsilon$-Whitney cover of $D$ if it satisfies the following properties:
\begin{enumerate}[\rm (i)]
\item   $r_i=\frac{\varepsilon}{1+\varepsilon}d_D(x_i)$ for each  $i\in I$. 

\smallskip

\item   $\{B(x_i,r_i),\,i\in I\}$ are pairwise disjoint.

\smallskip
			
\item $\bigcup_{i\in I}B (x_i,2(1+\varepsilon)r_i )=D$.  
\end{enumerate}
\end{definition}

\smallskip

The existence of   an $\eps$-Whitney cover follows from Zorn's lemma; see \cite[Proposition 3.2]{Mathav} for a proof.  
  The following lemma shows some basic properties of a Whitney cover. 

\begin{lemma}\label{lemma311+}
Let $\varepsilon\in (0,1/2)$ and let $\{B(x_i,r_i):x_i\in D,r_i>0,i\in I\}$ be an $\varepsilon$-Whitney cover of $D$. 
\begin{enumerate}[\rm (a)]
\item If $B(x_i,\lambda r_i)\cap B(x_i,\lambda r_j)\neq \emptyset$ for some $i,j\in I$ and $0<\lambda<\frac{1+\varepsilon}{\varepsilon}$, then 
\[
\frac{1+\varepsilon-\varepsilon\lambda}{1+\varepsilon+\varepsilon\lambda}r_i\leq r_j\leq \frac{1+\varepsilon+\varepsilon\lambda}{1+\varepsilon-\varepsilon\lambda}r_i.
\] 
		
\item If $x\in B(x_i,\lambda r_i)$ for some $i\in I$ and $0< \lambda<\frac{1+\varepsilon}{\varepsilon}$, then 
\begin{equation}\label{e:3.18a}
\frac\varepsilon{1+\varepsilon+\varepsilon\lambda}d_D(x)\leq r_i\leq \frac\varepsilon{1+\varepsilon-\varepsilon\lambda}d_D(x).
\end{equation}

\item For $0<\lambda<\frac{1+\varepsilon}{\varepsilon}$, there is a positive constant $C(\eps,\lambda)>0$ that depends also on the  parameters in (VD) of $m_0$ so that
\[
\#\{j\in I:B(x_j,\lambda r_j)\cap  B(x_i,\lambda r_i)\neq \emptyset\}\leq  C(\eps, \lambda) 
\quad  \hbox{ for every }i\in I.
\]

\item For $0<\lambda<\frac{1+\varepsilon}{\varepsilon}$, there is a positive constant $C(\eps,\lambda)>0$ that depends also on the  parameters in (VD) of $m_0$ so that 
\[
\#\{i\in I:x\in B(x_i,\lambda r_i)\}\leq C(\eps,\lambda)\quad  \hbox{ for every }  x\in D.
\]
\end{enumerate}
\end{lemma}

\begin{proof}
(a) is proved in \cite[Proposition 3.2(c)]{Mathav}.

\smallskip

(b).  For   $x\in B(x_i,\lambda r_i)$, by Definition \ref{def39}(i), 
$$
\Big| d_D(x)- \frac{1+\varepsilon}{\varepsilon} r_i \Big | = | d_D(x)- d_D(x_i)| \leq d(x, x_i) <\lambda r_i.
$$
This yields \eqref{e:3.18a}.

(c).   Note that  if $B(x_i,\lambda r_i)\cap B(x_i,\lambda r_j)\neq \emptyset$,  then   by (a),
\begin{equation} \label{e:3.19a}
B(x_j,r_j)\subset B(x_i,\lambda r_i+\lambda r_j+r_j)\subset 
B  (x_i,   Mr_i  ),
\end{equation}
where $M :=  \lambda+(\lambda+1) \frac{1+\varepsilon+\varepsilon\lambda}{1+\varepsilon-\varepsilon\lambda}$.
By    \eqref{e:VD} of (VD) of $m_0$,  
$$
V (x_j,r_j)  \leq  V (x_i, Mr_i)  \leq c M^{d_1}   V (x_j,r_j).
$$
Consequently,   by interchanging the role of $x_i$ and $x_j$,  we have 
\begin{equation}\label{e:3.20a}
V (x_j,r_j)\leq c M^{d_1}V (x_j,r_j). 
\end{equation}
Let 
$$
I_i:= \{ j\in I:B(x_j,\lambda r_j)\cap  B(x_i,\lambda r_i)\neq \emptyset\}.
$$
Then by \eqref{e:3.19a}-\eqref{e:3.20a}  and   pairwise disjointness of   the sets $\{B(x_i,r_i),\,i\in I\}$,
$$
V (x_j,r_j) \, \# I_i \leq c M^{d_1} \sum_{j\in I_i}  V (x_j,r_j) \leq  cM^{d_1}V (x_i, Mr_i)
 \leq c^2 M^{2 d_1} V (x_j,r_j).
$$ 
It follows that   $ \# I_i  \leq c^2 M^{2d_1}$.

(d)  Fix some $i\in I$ so that $x\in  B(x_i,\lambda r_i)$. Then $\{j\in I:x\in B(x_j,\lambda r_j)\} \subset I_i$. 
Hence by (c), $\# \{j\in I:x\in B(x_j,\lambda r_j)\} \leq \# I_i \leq C(\eps, \lambda).$  
\end{proof}

\begin{lemma}\label{lemma312+}
Let $\varepsilon\in (0,1/2)$ and let $\{B(x_i,r_i):x_i\in D,r_i>0,i\in I\}$ be an $\varepsilon$-Whitney cover of $D$. Let $x,y\in D$ and $\lambda>2(1+\varepsilon)$.  Suppose  that $d(x,y)\leq C_1\,d_D(x)$ for some $C_1 \geq 1$. Then, there is a path $\{i_0,i_2,\cdots,i_k\}\subset  I$ so that  $x\in B(x_{i_0},\lambda r_{i_0})$, $y\in B(x_{i_k},\lambda r_{i_k})$, and for $0\leq j'\leq j\leq k-1$, 
\begin{align*}
   B(x_{i_j},\lambda r_{i_{j}})\cap B(x_{i_{j+1}},\lambda r_{i_{j+1}})\neq\emptyset
   \quad \hbox{and} \quad  &r_{i_j}/r_{i_{j'}}\geq C_2e^{\alpha(j-j')} ,
\end{align*}
where $\alpha,C_2>0$ depend only on $\lambda,\varepsilon,C_1$, the constant $A>1$ in the definition of $A$-uniform domain 
for $D$ and the parameter in (VD) of $m_0$.
\end{lemma}

\begin{proof}  
 Let $x,y\in D$. By Lemma \ref{lemma21}(b), there is a path $\gamma$ in $D$ connecting  $x,y$ so that $\diam(\gamma)\leq A\,d(x,y)$ and 
\begin{equation}\label{e:lemma312.1}
\begin{split} 
d_D(z) \geq  \frac{d_D(x)\wedge d_D(y)}{A+1}\quad\hbox{ for every }z\in\gamma.
\end{split} 
\end{equation}

(i)  Suppose  that $d(x,y)\leq C_3\big(d_D(x)\wedge d_D(y)\big)$ for some $C_3\geq 1$.   
  Note  that  in this case, $\diam(\gamma)\leq Ad(x,y)\leq AC_3(d_D(x)\wedge d_D(y))$.  In view of \eqref{e:lemma312.1},  by  Lemma \ref{lemma2path}, 
there is 
a finite sequence of points $x=z_0,z_1,\cdots,z_k=y$ on $\gamma$ such that 
\begin{equation}\label{e:lemma312.2}
d(z_j,z_{j+1})<r :=2(\lambda-2(1+\varepsilon))\frac{\eps}{1+\eps+\eps\lambda}\frac{d_D(x)\wedge d_D(y)}{A+1}
\quad  \hbox{ for }0\leq i\leq k-1,
\end{equation}
where $l$ has an upper bound depending only on $\lambda,\varepsilon,A,C_3$ and the parameter in (VD) of $m_0$. For each $0\leq j\leq k$, we choose $i_j\in I$ so that $z_j\in B(x_{i_j},2(1+\eps)r_{i_j})$. Then by \eqref{e:lemma312.1},\eqref{e:lemma312.2} and Lemma \ref{lemma311+}(b),   $B(z_j,r/2)\subset B(x_{i_j},\lambda r_{i_j})$ for each $0\leq j\leq k$. Hence, the path $\{i_0,\cdots,i_k\}\subset  I$  has  the   properties that 
$x\in B(x_{i_0},\lambda r_{i_0})$, $y\in B(x_{i_k},\lambda r_{i_k})$, and 
$$	 B(x_{i_j},\lambda r_{i_{j}})\cap B(x_{i_{j+1}},\lambda r_{i_{j+1}})\neq\emptyset \quad  \hbox{for } 0\leq j\leq k-1,
$$
and
$$
C_4 (d_D(x)\wedge d_D(y))\leq r_{i_j}\leq C_5(d_D(x)\wedge d_D(y))\ \hbox{ for }0\leq j\leq k,
$$
 where $k$ has an upper bound depending only on $A$, $C_3$, $\lambda$, $\varepsilon$ and the parameter in (VD) of $m_0$, 
 and  $C_4,C_5$ are positive constants depending  only on $A$, $\lambda$ and $\varepsilon$.

\smallskip 

(ii) When $d(x,y)\leq C_1\,d_D(x)$ and $d_D(y)\geq d(x,y)/2$, 
 then $d(x,y)\leq C_1\big(d_D(x)\wedge d_D(y)\big)$.
In this case, the conclusion of the lemma follows directly from (i). 
So it remains to consider the case that $d(x,y)\leq C_1\,d_D(x)$ and $d_D(y)<d(x,y)/2$.

Let $\gamma$ be a path connecting $x,y$ in $D$ that satisfies  the defining property of $A$-uniform domain, and 
set $\bar \eta:=\big[ \frac{\log d(x,y)-\log d_D(y)}{\log 2} \big]\geq 1$. 
Set  $y_0:=y$ and  $y_{\bar\eta+1}:=x$. 
For each $\eta \in \N\cap[1, \bar \eta]$,   take some $y_\eta\in\gamma$ so that $d(y,y_\eta)=2^{\eta-1}d_D(y)\leq d(x,y)/2$.
Then 
$$
d_D(y_\eta)\geq\min \{d(y, y_\eta), d(y_\eta, x)\}/A\geq\min\{d(y,y_\eta),d(x,y)-d(y,y_\eta)\}/A\geq 2^{\eta-1}d_D(y)/A.
$$
In particular, $d_D(y_{\bar\eta})\geq d(x,y)/(4A)$. We claim that 
\begin{equation} \label{e:3.21a}
d(y_\eta,y_{\eta+1})\leq C_6\big(d_D(y_\eta)\wedge d_D(y_{\eta+1})\big)\ \hbox{ for } \eta \in \N \cap [0,  \bar\eta ],
\end{equation}
where $C_6=\max\{6A,3C_1/2\}$.  We can verify the claim case by case. Indeed, when $\eta=0$,
\[
d(y_0,y_1)=d_D(y)=A(d_D(y)\wedge \frac{d_D(y)}A)\leq A(d_D(y_0)\wedge d_D(y_1));
\]
when $1\leq\eta\leq\bar\eta-1$, 
\begin{align*}
d(y_\eta,y_{\eta+1})&\leq d(y,y_\eta)+d(y,y_{\eta+1})\leq 3\cdot 2^{\eta-1}d_D(y)\\
&=3A(\frac{2^{\eta-1}d_D(y)}{A}\wedge\frac{2^\eta d_D(y)}{A})\leq 
3A(d_D(y_\eta)\wedge d_D(y_{\eta+1}));
\end{align*}
and when $\eta=\bar\eta$, 
\begin{align*}
d(y_{\bar\eta},y_{\bar\eta+1})&=d(y_{\bar\eta},x)\leq d(x,y)+d_D(y,y_{\bar\eta})\leq \frac32d(x,y)\\
&\leq \max\{6A,\frac32C_1\}(d_D(y_{\bar\eta})\wedge\frac{d(x,y)}{C_1})\leq 
\max\{6A,\frac32C_1\}\cdot d_D(y_{\bar\eta})\wedge d_D(y_{\bar\eta+1}).
\end{align*}
This proves the claim \eqref{e:3.21a}. 
So, by the previous discussion in (i), we can find a path in $D$ connecting $y_\eta$ with $y_{\eta+1}$ for each $0\leq\eta\leq \bar\eta$ with $r_i\asymp 2^\eta d_D(y)$ for each $i$ in the path. It is then suffices to glue these paths together to get the desired path.
  \end{proof}

\medskip

 In the remaining of this section,  we assume  {\rm Cap}$(\Psi;D)_\leq$ holds for  $(D, d, m, \sE^0,\sF^0)$. We fix an $\varepsilon$-Whitney cover $\mathscr{R}$ of $D$ with $\varepsilon=1/(4A_1)$, where $A_1>1 $ is the constant in Cap$_\leq(\Psi;D)$. We use it to define an extension operator $L^2(\partial D;\sigma)\to \mcB (\tD)$ as follows.

\medskip
  For simplicity, in the rest of this section, let 
\begin{equation} \label{e:3.21}
B_i:=B\big(x_i,2A_1(1+\varepsilon)r_i\big)=B(x_i,d_D(x_i)/2).
\end{equation} 
 By {\rm Cap}$(\Psi;D)_\leq$ for $(D, d, m, \sE, \sF)$, for each $i\in I$,  there is some  $\widehat {\psi}_i\in C_c (\tD)\cap \sF$ so  that $0\leq \widehat {\psi}_i \leq 1$, $\widehat{\psi}_i =1$ on $ B(x_i,2(1+\varepsilon)r_i)$, $\widehat {\psi}_i = 0$ on $\tD \setminus   B_i$ and 
	\[
	  \sE(\widehat {\psi}_i,\widehat {\psi}_i)\leq C\frac{V (x_i,r_i)}{\Psi(r_i)},
	\]
	where $C\in (0,\infty)$ is a constant independent of $i\in I$. Set 
	\[
	\psi_i(x)=\frac{\widehat {\psi}_i(x)}{\sum_{j \in I}\widehat {\psi}_j (x)}\quad\hbox{ for }i\in I,\ x\in D.
	\]
 By property (iii) of the Whitney cover, $\sum_{j \in I}\widehat {\psi}_j (x)\geq 1$ on $D$.

\medskip 

\begin{lemma}\label{lemma311}
Suppose that condition {\rm Cap}$_\leq(\Psi;D)$ 
 holds for  $(D, d, m, \sE^0,\sF^0)$.  Then 
$\psi_i\in\sF^0$ for every $i\in I$, and there is a constant $C\in (0,\infty)$ so that 
$$
\sE^0  (\psi_i,\psi_i)\leq C\frac{V (x_i,r_i)}{\Psi(r_i)}  \quad \hbox{for every } i\in I . 
$$
\end{lemma}

\begin{proof} 
For $i\in I$, let $I_i:= \{ j \in I:   B_i \cap B_j  \neq\emptyset\}.$   By Lemma \ref{lemma311+} with $\lambda = \frac{1+\eps}{2\eps}$ and $\eps=1/(4A_1)$, there is an integer $N\geq 1$ and 
 a positive constant $C>1$ so that for each $i\in I$, $\# I_i\leq  N$  and 
 \begin{equation} \label{e:3.18} 
 C^{-1}r_i\leq r_j\leq C r_i  \quad \hbox{for every }  j\in I_i.
 \end{equation}   
 For each $i\in I$, define
\[
\phi_i= \Big(\sum_{j\in I_i}  \widehat  \psi_j \Big)\wedge \frac{1}{\sum_{j\in I_i} \widehat\psi_j}.
\]
Clearly $0\leq\phi_i\leq 1$ and $\sum_{j\in I_i}  \widehat  \psi_j \in \sF^0$.  Note that  $\phi_i$ is a normal contraction of
$\sum_{j\in I_i}  \widehat  \psi_j $, that is, $ | \phi_i(x) | \leq  \big| \sum_{j\in I_i}  \widehat  \psi_j (x)  \big| $ and 
$$
| \phi_i(x)  -  \phi_i (y) | \leq \Big|  \sum_{j\in I_i}  \widehat  \psi_j (x) - \sum_{j\in I_i}  \widehat  \psi_j  (y) \Big|
\quad \hbox{for every } x, y\in D.
$$
It follows that  $\phi_i\in \sF^0$  and 
$$
\sE^0. (\phi_i,\phi_i) \leq.  \sE^0. (\sum_{j\in I_i}  \widehat  \psi_j ,\sum_{j\in I_i}  \widehat  \psi_j) 
 \leq N \sum_{j\in I_i}   \sE^0(\widehat\psi_j,\widehat\psi_j) 
 \leq  N \sum_{j\in I_i}  \frac{V (x_j ,r_j)}{\Psi(r_j)}
 \leq C_1 \frac{V (x_i,r_i)}{\Psi(r_i)},
$$
where the last inequality is due to \eqref{e:3.18}, \eqref{eqnpsi} and VD of $m_0$. By the definition of Whitney cover, $\sum_{j\in I_i}\widehat\psi_j=\sum_{j\in I }\widehat\psi_j\geq 1$ on $B(x_i,2A_1(1+\varepsilon)r_i)$. Hence by
 \cite[Theorem 1.4.2]{FOT}, 
$\psi_i=\phi_i\widehat\psi_i\in\sF_b$ and $\sE^0(\psi_i,\psi_i)\leq 2\|\phi_i\|_\infty^2\sE^0(\widehat\psi_i,\widehat\psi_i)+2\|\widehat\psi_i\|_\infty^2\sE^0(\phi_i,\phi_i)
\leq C V (x_i,r_i)/ \Psi(r_i)$. 
\end{proof}

  Recall that $\eps=1/(4A_1)$. Define 
\[
F_i:=B\big(x_i,  8A_1 (1+\varepsilon)r_i\big)\cap \partial D= B(x_i, 2d_D(x_i))\cap\partial D \quad  \hbox{ and }\quad [u]_i :=\frac1{\sigma(F_i)}\int_{F_i}ud\sigma. 
\]
Let $\xi_i\in \partial D$ so  that $d(x_i,\xi_i)=d_D(x_i)$. Then  
\begin{equation}\label{e:3.19}
B(\xi_i,d_D(x_i))\subset F_i\subset B(\xi_i,3d_D(x_i)). 
\end{equation}
For each $u\in L^1_{\rm loc}(\partial D;\sigma)$,  we define $\Ex (u)\in \sF_{\rm loc}(D)$ by 
\[
\Ex (u) (x) =\sum_{i\in I} \psi_i (x)[u]_i.
\] 
 We extend the definition of $\Ex (u)$ to $\partial D$ by setting $\Ex (u)=u$ on $\partial D$. It is not hard to see that $\Ex :C(\partial D)\to C(\tD)$.

\begin{proposition}\label{prop312}
Suppose that 
 {\rm Cap}$_\leq(\Psi;D)$
holds for  $(D,d,m, \sE^0,\sF^0)$,  and 
 $\sigma$ is  a     Radon measure  with full support on $\partial D$ satisfying {\rm (VD)}  property so that 
 {\rm (LS)} holds for     $\Theta_{\Psi, \sigma}$. 
 There is $C\in (0,\infty)$ such that $\mu_{\<\Ex  (u)\>}(D)\leq C\,\lb u \rb^2_{\Lambda_{\Psi,  \sigma} }$ for each  $u\in \dot \Lambda_{\Psi,  \sigma} $.
\end{proposition}

\begin{proof}  For each $i\in I$, let $I_i:=\{j\in I: B_j\cap B_i\neq \emptyset\}$. For  $u\in \dot \Lambda_{\Psi,  \sigma} $,  by Lemmas  \ref{lemma311+}(c)  and \ref{lemma311}
\begin{eqnarray}\label{e:3.20+}
   \mu_{\<\Ex (u)\>}(B_i) 
&=&\mu_{\<\Ex (u)-[u]_i\>}(B_i)   \nonumber \\
&\lesssim& \sum_{j\in I_i} ([u]_i-[u]_j)^2\mu_{\<\psi_j\>}(D)  \nonumber \\
&\lesssim& \sum_{j\in I_i} ([u]_i-[u]_j)^2\frac{V(x_j,r_j)}{\Psi(r_j)} \nonumber\\
&\lesssim& \sum_{j\in I_i}\frac1{\sigma(F_i) \sigma(F_j)}\int_{x\in F_i}\int_{y\in F_j}\frac{V (x_j,r_j)}{\Psi(r_j)}(u(x)-u(y))^2\sigma(dy)\sigma(dx) \nonumber \\
&\lesssim & \int_{x\in F_i}\int_{y\in \partial D\atop d(x,y)<C_1r_i} 
\frac{\big(u(x)-u(y)\big)^2}{\Theta_{\Psi, \sigma} (x,C_1r_i) V_\sigma (x,C_1r_i)}\sigma(dx)\sigma(dy),
\end{eqnarray}
where the last inequality is due to Lemma \ref{lemma311+}(a),  \eqref{e:3.19},  and the volume doubling property of $m_0$ and $\sigma$ with  $C_1>0$ independent of 
 $i\in I$ and $j\in I_i$.

Next, we fix $\eta\in (0,1)$ and $k\in \Z$. Then by \eqref{e:3.20+} and Lemma \ref{lemma311+}(d), 
\begin{equation}\label{eqn310}
\begin{aligned}
&\sum_{i\in I\atop \eta^{k+1}\leq r_i<\eta^{k}}\mu_{\<\Ex  (u)\>}(B_i)\\\lesssim& \sum_{i\in I\atop \eta^{k+1}\leq r_i<\eta^k}\int_{x\in F_i}\int_{y\in \partial D\atop d(x,y)<C_1r_i} \frac{\big(u(x)-u(y)\big)^2}{\Theta_{\Psi, \sigma}  (x,C_1r_i)   V_\sigma  (x,C_1r_i)}\sigma(dy)\sigma(dx)\\
\lesssim &\int_{x\in\partial D}\int_{y\in \partial D\atop d(x,y)<C_1\eta^k} \frac{\big(u(x)-u(y)\big)^2}{\Theta_{\Psi, \sigma}  (x,C_1\eta^k)
 V_\sigma  (x,C_1\eta^k) }\sigma(dx)\sigma(dy).
\end{aligned}
\end{equation}
As $\bigcup_{i\in I}B_i=D$, 
we have by  \eqref{eqn310}  and  {\rm(LS)}  for $\Theta_{\Psi, \sigma}$, 
\begin{eqnarray*}
&& \mu_{\<\Ex  (u)\>}(D)
 \, \leq \,   \sum_{k\in \Z } \sum_{i\in I\atop \eta^{k+1}\leq r_i<\eta^{k}}\mu_{\<\Ex  (u)\>}(B_i) \\
&\lesssim  &\int_{x\in\partial D}\int_{y\in \partial D} \sum_{k\in \Z} \1_{\{ d(x,y)<C_1\eta^k\}} 
\frac{\Theta_{\Psi, \sigma}  (x, d(x, y))}{\Theta_{\Psi, \sigma}  (x,C_1\eta^k)}
\frac{\big(u(x)-u(y)\big)^2}{\Theta_{\Psi, \sigma}  (x, d(x, y))  V_\sigma (x, d(x, y)) }\sigma(dx)\sigma(dy) \\
&\lesssim  &\int_{x\in\partial D}\int_{y\in \partial D} \sum_{k\in \Z} \1_{\{ d(x,y)<C_1\eta^k\}} 
  \left( \frac{d(x, y))}{ C_1\eta^k } \right)^\beta
\frac{\big(u(x)-u(y)\big)^2}{\Theta_{\Psi, \sigma}  (x, d(x, y) )  V_\sigma  (x, d(x, y)) }\sigma(dx)\sigma(dy) \\
&\lesssim  &\int_{x\in\partial D}\int_{y\in \partial D}  \frac{\big(u(x)-u(y)\big)^2}{\Theta_{\Psi, \sigma}  (x, d(x, y)) 
V_\sigma  (x, d(x, y)) }\sigma(dx)\sigma(dy) \, =\, \lb u \rb^2_{\Lambda_{\Psi,  \sigma} }.
\end{eqnarray*}
This proves the proposition.
\end{proof}

We can also show  that $\Ex:C_c(\partial D)\cap \Lambda_{\Psi,  \sigma} \to \bar \sF_e$ 
under  the same conditions as in Proposition \ref{prop312}. 

\begin{proposition}\label{prop313} Suppose that   {\rm Cap$_\leq(\Psi;D)$} hold for $(D, d, m, \sE^0,\sF^0)$,  and  $\sigma$ is  a     Radon measure  with full support on $\partial D$  satisfying {\rm (VD)}  property so that  {\rm (LS)} holds for     $\Theta_{\Psi, \sigma}$. 
   \begin{enumerate} [\rm (a)] 
\item For each  $\delta>0$,   $u\in C_c(\partial D)\cap \Lambda_{\Psi,  \sigma} $ and $\psi\in C_c(\tD)\cap \bar \sF $ such that $0\leq\psi\leq 1$ and $\psi=1$ on $\{x\in \partial D: d(x, {\rm supp} [u])<\delta \}$,  we have that $\psi  \, \Ex (u)\in\bar \sF  \cap C_c(\tD)$ and 
 $$  \bar \sE_1(\psi \, \Ex (u),\psi \, \Ex (u))^{1/2}  \leq C  \| u  \|_{\Lambda_{\Psi,  \sigma} }
$$ 
for some $C$ depending on $\psi$ and $ \delta$ but not on $u$. 

\item  
\begin{description}
 \item[\rm (i)]  If  $D$ is  bounded, then $\Ex : C(\partial D)\cap \Lambda_{\Psi,  \sigma} \to \bar \sF \cap C(\tD)$. 

 \item[\rm (ii)]  If   $\partial D$  is  unbounded, then $\Ex : C_c(\partial D)\cap \Lambda_{\Psi,  \sigma} \to \bar \sF_e\cap C_0(\tD)$.

\item[\rm (iii)] If $\partial D$ is  bounded and $D$ is unbounded, then for each $u\in C (\partial D)\cap \Lambda_{\Psi,  \sigma} $ such that $\int_{\partial D} u(x)\sigma(dx)=0$, we have $\Ex (u)\in \bar \sF_e\cap  C_c(\tD)$. Moreover, if $(\bar \sE , \bar \sF )$ is recurrent, 
then $\Ex : C (\partial D)\cap \Lambda_{\Psi,  \sigma} \to \bar \sF_e\cap C(\tD)$.
\end{description}
\end{enumerate} 
\end{proposition}

\begin{proof}	
(a). First, we show that $\Ex u\in C(\tD)$.  
Recall that $B_i=B(x_i, d_D(x_i)/2)$ and $F_i =B(x_i,2d_D(x_i))\cap\partial D$. So
$$
d(x_i,y)\leq d_D(y) \quad \hbox{and} \quad 
\frac23d_D(y)\leq d_D(x_i)\leq 2d_D(y) \quad\hbox{for }  y\in B_i.
$$
 Thus  for each $y\in D$, 
\begin{equation}\label{eqn3.a}
\begin{aligned}
\Ex u(y) & =\sum_{i\in I}[u]_i\psi_i(y) =\sum_{i\in I, y\in B_i}[u]_i\psi_i(y)\\
&\in \Big[\inf_{z\in B(y,5d_D(y))\cap \partial D}u(z),  \ \sup_{z\in B(y,5d_D(y))\cap \partial D}u(z) \Big].
\end{aligned}
\end{equation}
It follows that $\Ex u$ is continuous at each point on $\partial D$. 

By Proposition \ref{prop312},   $\mu_{\<\Ex u\>}(D)\lesssim\lb u\rb^2_{\Lambda_{\Psi,  \sigma} }$. Define $O:=\{x\in \tD:\,\psi(x)=1\}$ and let $y\in D\setminus O$. 
  By \eqref{eqn3.a}  $\Ex u(y)=0$  if   $ 5d_D(y)\leq d(y,\operatorname{supp}[u])$.
  For $i\in I$ so that $B_i \ni y$,  let $\xi_i\in\partial D$ be such that $d(x_i,\xi_i)=d_D(x_i)$.
  Note that 
  $d(\xi_i,y)\leq d(\xi_i,x_i)+d(x_i,y)\leq 3d_D(y)$ and 
$$
F_i \supset \partial D\cap B(\xi_i,   d_D(x_i))\supset \partial D\cap B(\xi_i,2d_D(y)/3).
$$
   If $ 5d_D(y)>d(y,\operatorname{supp}[u])>\delta$,   
then  by the (VD) property of $\sigma$
\begin{align*}
\sigma(F_i)&\geq \sigma(B(\xi_i,2d_D(y)/3))\\
&\geq c_1\sigma \big( B(\xi_i,(8+5\tfrac{\diam(O)}{\delta})d_D(y)) \big)\\
&\geq 
c_1\sigma(B(\xi_i,d(\xi_i,y)+d(y,\operatorname{supp}[u])+\diam(O))\geq c_1\sigma(O)
\end{align*}
where $c_1>0$  depends on   $\psi$ and $ \delta$ but not on $u$.
This implies
\[
|\Ex u(y)|\leq\sum_{i\in I,y\in B_i}\psi_i(y)|[u]_i|\leq \sum_{i\in I,y\in B_i}\psi_i(y)\frac{\|u\|_{L^2(\partial D;\sigma)}}{\sqrt{\sigma(F_i)}}\leq  \frac{\|u\|_{L^2(\partial D;\sigma)}}{\sqrt{c_1\sigma(O)}}.
\]
This proves that   $\Ex u(y)\lesssim \|u\|_{L^2(\partial D;\sigma)}$ for every $y\in D\setminus O$, where the constant in $\lesssim$ depends on $\psi$ and $\delta$  but not on $u$. Hence by the derivation property of the energy measure (see, e.g.,  \cite[Lemma 4.3.6]{CF}),   
\begin{align*}
	\mu_{\<\psi\cdot \Ex u\>}(D)&=\mu_{\<\psi\cdot \Ex u\>}(D\cap O)+\mu_{\<\psi\cdot \Ex u\>}(D\setminus O)\\
	&=\mu_{\<\Ex u\>}(D\cap O)+\mu_{\<\psi\cdot \Ex u\>}(D\setminus O)\\
	&\leq \mu_{\<\Ex u\>}(D\cap O)+2\mu_{\<\Ex u\>}(D\setminus O)+2\|\Ex _{\eta}(u)\|_{L^\infty(D\setminus O)}^2\mu_{\<\psi\>}(D\setminus O)\\
	&\lesssim \| u \|^2_{\Lambda_{\Psi,  \sigma} },
\end{align*} 
where the constant in $\lesssim$ depends on $\psi$ and  $( D,d,m , \sE,  \sF)$.

 Next, we derive  an upper bound estimate of $\|\psi \Ex u\|_{L^2(\tD;m_0)}$. 
 Let $r=\diam(\operatorname{supp}[\psi])/2$. We fix $   i_0 \in I$ so that 
\[
B_{i_0}\cap D_{r/(4A)}\neq\emptyset 
\quad \hbox{ and } \quad 
d\big(B_{i_0},\operatorname{supp}[\psi]\big)<r.
\]
 For each $i\in I$ such that $B_i\cap\operatorname{supp}[\psi]\neq\emptyset$, by Lemma \ref{lemma312+},  there is a finite sequence $ j_{i,0} = i_0, \, j_{i,1},\cdots, \, j_{i,k_i}=i$ such that 
  $B_{j_{i,l-1}}\cap B_{j_{i,l }}\neq\emptyset$ for $1\leq l\leq k_i$,  
 and there is $\alpha>0$ such that 
\begin{equation}\label{eqn318}
\frac{r_{j_{i,l}}}{r_{j_{i,l'}}}\gtrsim e^{\alpha(l'-l)}  \quad \hbox{ for every }0\leq l<l'\leq k_i.
\end{equation} 
Note that $r_{j_{i, 0}} =r_{i_0} \gtrsim r$. Then,
\begin{eqnarray*}
&& \frac{1}{\Psi(r)}\sum_{i\in I, B_i\cap \operatorname{supp}[\psi]\neq\emptyset}([u]_i-[u]_{i_0})^2m_0(B_i)
\\
&\leq  &\frac{1}{\Psi(r)}\sum_{i\in I, B_i\cap \operatorname{supp}[\psi]\neq\emptyset}\Big(\sum_{l=0}^{k_i-1}([u]_{j_{i,l}}-[u]_{j_{i,l+1}})\Big)^2m_0(B_i)
\\
& \leq &\sum_{i\in I, B_i\cap \operatorname{supp}[\psi]\neq\emptyset}\Big(\sum_{l=0}^{k_i-1}\frac{m_0(B_i)}{\Psi(r_{j_{i,l}})}([u]_{j_{i,l}}-[u]_{j_{i,l+1}})^2\Big)\big(\sum_{l=0}^{k_i-1}\frac{\Psi(r_{j_{i,l}})}{\Psi(r)}\big)\\
& \lesssim&\sum_{i\in I, B_i\cap \operatorname{supp}[\psi]\neq\emptyset}\sum_{l=0}^{k_i-1}\frac{m_0(B_i)}{\Psi(r_{j_{i,l}})}([u]_{j_{i,l}}-[u]_{j_{i,l+1}})^2\\
& \leq& \sum_{i_1,i_2\in I\atop B_{i_1}\cap B_{i_2}\neq\emptyset}\sum_{i\in I,B_i\cap\operatorname{supp}[\psi]\neq\emptyset\atop i_1\in \{j_{i,0},\cdots j_{i,k_i-1}\}}\frac{m_0(B_i)}{\Psi(r_{i_1})}([u]_{i_1}-[u]_{i_2})^2\\
& \lesssim&\sum_{i_1,i_2\in I\atop B_{i_1}\cap B_{i_2}\neq\emptyset}\frac{m_0(B_{i_1})}{\Psi(r_{i_1})}([u]_{i_1}-[u]_{i_2})^2.
\end{eqnarray*}
where the third inequality is due to \eqref{eqnpsi} and \eqref{eqn318},  and the last inequality is due to   the pairwise disjointness of $\{B(x_i,r_i),i\in I\}$, (VD) of $m_0$
and the fact that there is $C_3>0$ such that $B(x_i,r_i)\subset B(x_{i_1},C_3r_{i_1})$ for any $i_1\in \{j_{i,0},\cdots,j_{i,k_i-1}\}$ because of \eqref{eqn318}.   Combining  this with a similar argument as that for   \eqref{e:3.20+} and \eqref{eqn310},  we get 
\begin{eqnarray*}
& & \frac{1}{\Psi(r)}\sum_{i\in I, B_i\cap \operatorname{supp}[\psi]\neq\emptyset}([u]_i-[u]_{i_0})^2m_0(B_i)\\
&\lesssim &\sum_{i,j\in I\atop B_{i}\cap B_{j}\neq\emptyset}\frac{m_0(B_{i})}{\Psi(r_{i})}([u]_{i}-[u]_{j})^2 \\
&=& \sum_{k\in\Z}\sum_{i\in I\atop \eta^{k+1}\leq r_i<\eta^{k}}\sum_{j\in I\atop B_{i}\cap B_{j}\neq\emptyset}\frac{m_0(B_{i})}{\Psi(r_{i})}([u]_{i}-[u]_{j})^2\\
&\lesssim & \sum_{k\in\Z}\int_{x\in\partial D}\int_{y\in \partial D\atop d(x,y)<C_4\eta^k} \frac{\big(u(x)-u(y)\big)^2}{\Theta_{\Psi, \sigma}  (x,C_4\eta^k)V_\sigma (x,C_4\eta^k) }\sigma(dy)\sigma(dx)\\
&\lesssim & \lb u\rb^2_{\Lambda_{\Psi,  \sigma} },
\end{eqnarray*} 
where we used the (LS) property of  $\sigma$ in the last inequality.
Hence,  
\begin{align*}
\|\psi \Ex u\|_{L^2(\tD;m_0)}^2 &\lesssim \sum_{i\in I, B_i\cap \operatorname{supp}[\psi]\neq\emptyset}[u]_i^2m_0(B_i)\\
&\leq 2\sum_{i\in I, B_i\cap \operatorname{supp}[\psi]\neq\emptyset}([u]_i-[u]_{i_0})^2m_0(B_i)+2[u]_{i_0}^2\sum_{i\in I, B_i\cap \operatorname{supp}[\psi]\neq\emptyset}m_0(B_i)\\
&\lesssim \Psi(r)\lb u\rb_{\Lambda_{\Psi,  \sigma} }^2+\frac{\sum_{i\in I, B_i\cap \operatorname{supp}[\psi]\neq\emptyset}m_0(B_i)}{\sigma(F_{i_0})}\|u\|^2_{L^2(\partial D;\sigma)}. 
\end{align*}
This gives an upper bound estimate of $\|\psi \Ex u\|_{L^2(\tD;m_0)}$ in terms of $\|u\|_{\Lambda_{\Psi,  \sigma} }$. 

Combining the above $L^2$ norm estimate with  the energy estimate part in the previous paragraph establishes the inequality in  (a).

\medskip

(b)  Let $u\in \Lambda_{\Psi,  \sigma} \cap C_c(\partial D)$.  

\smallskip

(i) If $D$ is bounded, then   $\Ex (u)\in \bar \sF\cap C(\tD)$  by (a). 

 \smallskip

(ii) Suppose  $\partial D$ is unbounded. For $y\in\tD$, 
 if $ 5d_D(y)\leq d(y,\operatorname{supp}[u])$, then $\Ex u(y)=0$ by \eqref{eqn3.a}. Fix some $\xi \in \operatorname{supp}[u]$.
If $5d_D(y)>d(y,\operatorname{supp}[u])$ and $ y\in  B_i := B(x_i, d_D(x_i)/2)$, then   
\begin{equation} \label{e:3.23} 
d_D(x_i)\geq\frac{2}{3}d_D(y)>\frac{2}{15}d(y,\operatorname{supp}[u]).
\end{equation}
Clearly $[u]_i=0$ if $F_i\cap \operatorname{supp}[u] =\emptyset$. When $F_i\cap \operatorname{supp}[u] \not= \emptyset$,  
$$
d(x_i, \xi) <  2  d_D(x_i)   + \diam (\operatorname{supp}[u]),
$$
where $\xi_i\in \partial D$  so that $d(x_i, \xi_i)=d_D(x_i)$. Suppose 
 $d(y,\operatorname{supp}[u])\geq \frac{15}2\diam(\operatorname{supp}[u])$. 
 Then $ d_D(x_i) \geq \frac23d_D(y)\geq \frac2{15}d(y,\operatorname{supp}[u])\geq\diam(\operatorname{supp}[u])$.
Hence $B(\xi,  d_D(x_i))\subset  (B(\xi_i,  4d_D(x_i))$.  By the VD of $\sigma$, 
  \begin{eqnarray*} 
  \sigma (F_i)  &\geq&  \sigma (B(\xi_i, d_D(x_i)))
  \gtrsim  \sigma (B(\xi_i, 4 d_D(x_i)))  \\
  &\geq &    \sigma (B(\xi,  d_D(x_i)))
  \gtrsim  \sigma (B(\xi,  d(y,\operatorname{supp}[u]) )).
  \end{eqnarray*} 
  As there are at most $N$ number of $i\in I$ so that $B_i\ni y$, we have  
  \begin{eqnarray*} 
\Ex u(y)  &=&  \sum_{i\in I,   B_i \ni y} [u]_i\psi_i(y)
\leq \sum_{i\in I,   B_i \ni y}  \psi_i(y)   \sigma(F_i)^{-1}\|u\|_{L^1(\partial D;\sigma)} \\
&\lesssim & N  V_\sigma (\xi,d_D(y))^{-1}\|u\|_{L^1(\partial D;\sigma)}\to 0
\qquad \hbox{as }  d(y,\operatorname{supp}[u])\to\infty. 
\end{eqnarray*} 
 This proves that $\Ex u \in C_0 (\tD)$. 
This together with (a) yields that  $\varphi_\delta\circ u\in C_c(\tD)\cap\bar\sF$,  where 
\[
\varphi_{\delta}(t)=\begin{cases}t-\delta&\hbox{ if }t>\delta,\\
	0&\hbox{ if }-\delta\leq t\leq \delta,\\
	t+\delta &\hbox{ if }t<-\delta.
\end{cases}
\]	
Moreover, by Proposition \ref{prop312},   $\bar \sE (\varphi_{\delta}\circ\Ex (u),\varphi_{\delta}\circ\Ex (u))$ is bounded in $\delta >0$ .  As $\lim_{\delta \to 0} \varphi_{\delta}\circ\Ex (u)=\Ex (u)$, it follows that  $\Ex (u)\in \bar \sF_e$ (see, e.g., \cite{CF}). 

\smallskip

 (iii)  Suppose that $\partial D$ is bounded, $D$ is unbounded,  and $u\in C (\partial D)\cap \Lambda_{\Psi,  \sigma} $ having  
 \hfill \\
  $\int_{\partial D} u(x)\sigma(dx)=0$.
  If $d_D(y)>\frac{3}{2}\diam(\partial D)$, then for any $i\in I$ having    $y\in B_i$,   we have $d_D(x_i)\geq\frac23d_D(y)>\diam(\partial D)$. This implies that  $F_i=\partial D$ and so $\Ex u(y) = \sum_{j \in I} [u]_j \psi_j  (y) = \sum_{j \in I:  B_j  \ni y} [u]_j \psi_j (y) =0$. This together with (a) proves that $\Ex(u)\in C_c(\tD) \cap \bar \sF$.  
Moreover, if $(\bar\sE ,\bar\sF)$ is recurrent, then $\Ex ({\mathbbm 1}_{\partial D})=1 \in \sF_e$. This implies that for any $u\in C (\partial D)\cap \Lambda_{\Psi,  \sigma} $, with $c(u):= \int_{\partial D}u(x)\sigma(dx)$, $\Ex(u)=\Ex(u-c(u))+c(u)\in \bar \sF_e\cap C(\tD)$.
\end{proof}

\begin{proposition}\label{P:3.15} 
 Suppose that    {\rm Cap}$_\leq(\Psi; D)$  and {\rm PI}$(\Psi;D)$ hold for  $(D, d, m,  \sE^0,\sF^0)$, 
 and  $\sigma$ is  a     Radon measure  with full support on $\partial D$ satisfying {\rm (VD)}  property so that 
 {\rm (LS)} holds for     $\Theta_{\Psi, \sigma}$. 
Then $\sigma$ is a smooth measure  on $\partial D$ with full $\bar\sE$-quasi-support.
\end{proposition}

\begin{proof}
It is proven in Theorem \ref{thmrestriction2} that $\sigma$ is a smooth measure. 
We prove the rest of the claim by contradiction. Assume that $\partial D$ is not a quasi-support of $\sigma$, then by \cite[Theorem 4.6.2]{FOT}, there exists $f\in \bar\sF$ such that $f|_{\partial D}=0$ $\sigma$-a.e. and $f|_{\partial D}$ does not equal to $0$ q.e. on $\partial D$. Moreover, we can assume that $f$ is compactly supported. Take $\psi\in. C_c(\tD)\cap\bar{\sF}$ such that   $0\leq \psi\leq 1$ and $\psi=1$ on a neighborhood of the support of $f$.  

By the regularity of $\big(\bar \sE ,\bar \sF \big)$, we can find a sequence $f_n\in C_c(\tD)\cap \bar \sF $ 
with ${\rm supp}[f_n] \subset \{ \psi=1\}$
so that $f_n \to \tilde{f}$ $\bar \sE$-q.e. and $f_n\to f$ in $\bar \sF $ as $n\to\infty$. 
Next, for $n\geq 1$, we let $g_n=\psi\cdot\Ex (f_n|_{\partial D})$. Then, by Theorems \ref{thmrestriction1}, \ref{thmrestriction2} and Proposition \ref{prop313} (a),  for $n,m\geq 1$
\begin{eqnarray}
	&\label{eqn311}g_n\in C_c (\tD)\cap \bar\sF,\quad g_n|_{\partial D}=f_n|_{\partial D}\\
	&\label{eqn312}\|g_n-g_m\|_{\bar\sF}\lesssim \|(g_n-g_m)|_{\partial D}\|_{\Lambda_{\Psi,  \sigma} }=\|(f_n-f_m)|_{\partial D}\|_{\Lambda_{\Psi,  \sigma} }\lesssim \|f_n-f_m\|_{\bar\sF}\\
	&\label{eqn313}\|g_n\|_{\bar\sF}\lesssim\|g_n|_{\partial D}\|_{\Lambda_{\Psi,  \sigma} }=\|f_n|_{\partial D}\|_{\Lambda_{\Psi,  \sigma} }
\end{eqnarray}
Combining \eqref{eqn311} and \eqref{eqn312}, by passing to a subsequence if necessary, we can find $g\in \bar\sF$ such that 
\begin{eqnarray}\label{eqn314}
	g_n\to g\hbox{ in }\bar\sF,  \quad   g_n\to g \  \ \bar \sE  \hbox{-q.e.} 
	\quad \hbox{and} \quad g_n|_{\partial D}\to g|_{\partial D} \  \hbox{ in }\Lambda_{\Psi,  \sigma} .
\end{eqnarray}
Combining with the fact that $f_n\to f$ $\bar \sE$-q.e. and $g_n|_{\partial D}=f_n|_{\partial D}$ for each $n\geq 1$, we see that 
\[
g|_{\partial D}=f|_{\partial D}\quad\hbox{q.e.}
\]
Hence, by the assumption $f$ does not equal to $0$ q.e. on $\partial D$, we know that $g$ does not equal to $0$ q.e. on $\partial D$; on the other hand, by taking the limit of \eqref{eqn313} (using \eqref{eqn312}), we see that $\|g\|_{\sF}\lesssim \|g|_{\partial D}\|_{\Lambda_{\Psi,  \sigma} }=\|f|_{\partial D}\|_{\Lambda_{\Psi,  \sigma} }=0$ so that $g=0$ q.e. on $\partial D$. This leads to a contradiction.
\end{proof}

We conclude this section with the following theorem about $\big(\check \sE,\check  \sF \big)$. 

\begin{theorem}\label{T:3.16}
Suppose that {\rm Cap}$_\leq(\Psi; D)$  and {\rm PI}$(\Psi;D)$ hold  for  $(D, d, m,  \sE^0,\sF^0)$,  and  $\sigma$ is  a     Radon measure  with full support on $\partial D$ satisfying {\rm (VD)}  property so that 
 {\rm (LS)} holds for     $\Theta_{\Psi, \sigma}$.  
\begin{enumerate}[\rm (a)]
\item If $D$ is bounded, or if  $\partial D$ is unbounded, or if $(\bar \sE ,\bar \sF )$ is recurrent, we have 
\[
\check \sE (u,u)\asymp \lb u \rb_{\Lambda_{\Psi,  \sigma} }^2  \quad \hbox{for  }  u\in \check \sF    .
\]
  Moreover, $\Lambda_{\Psi,  \sigma} \cap C_c(\partial D)$ is a core of $\check \sF $. 

\item If $D$ is unbounded, $\partial D$ is bounded and $(\bar \sE ,\bar \sF )$ is transient, then
\[
\check \sE (u,u)\asymp \| u \|_{\Lambda_{\Psi,  \sigma} }^2=\lb u \rb_{\Lambda_{\Psi,  \sigma} }^2+\|u\|_{L^2(\partial D;\sigma)}^2
 \quad \hbox{for  }  u\in \check \sF    .
\]
 Moreover, $\Lambda_{\Psi,  \sigma} \cap C_c(\partial D)\subset \check \sF $ is a core of $\check \sF$. 
\end{enumerate}
\end{theorem}
\begin{proof}

(a). First, for each $u\in C_c(\partial D)\cap \Lambda_{\Psi,  \sigma} $, 
\[
\lb u \rb_{\Lambda_{\Psi,  \sigma} }^2\lesssim \check \sE(u,u)\leq \bar \sE \big(\Ex (u),\Ex (u)\big)\lesssim \lb u \rb_{\Lambda_{\Psi,  \sigma} }^2,
\]
where   the first inequality is due to  Theorem \ref{thmrestriction1},  
and the second inequality is due to Proposition \ref{prop312} and Proposition \ref{prop313}(b). 
Moreover, by Theorem \ref{thmrestriction1}, we know that $u|_{\partial D}\in\Lambda_{\Psi,  \sigma} \cap C_c(\partial D)$ for each $u\in \bar \sF \cap C_c(\tD)$. So by \cite[Theorem 5.8]{CF}, $\Lambda_{\Psi,  \sigma} \cap C_c(\partial D)$ is a core of $\check \sF$. 

(b).  Choose $\psi\in C_c(\tD)$ such that $\psi=1$ on a neighborhood of $\partial D$ and $0\leq\psi\leq 1$. Then, for each $u\in C_c(\partial D)\cap \Lambda_{\Psi,  \sigma} $, we have  
\begin{equation}\label{eqn315}
\check \sE(u,u)=\bar \sE (\mathcal{H}u,\mathcal{H}u)\leq \bar \sE \big(\psi\Ex (u),\psi\Ex (u)\big)\lesssim \lb u \rb_{\Lambda_{\Psi,  \sigma} }^2+\|u\|_{L^2(\partial D;\sigma)}^2
\end{equation}
by Proposition \ref{prop313} (a).

Moreover, by Theorem \ref{thmrestriction1}, 
\begin{equation}\label{eqn316}
\check \sE(u,u)\gtrsim \lb u \rb_{\Lambda_{\Psi,  \sigma} }^2. 
\end{equation}

Finally,  we show 
\begin{equation}\label{eqn317}
\check \sE(u,u)\gtrsim \|u\|_{L^2(\partial D;\sigma)}^2
\end{equation}
We prove it by contradiction. Assume that \eqref{eqn317} is not true, then there is a sequence $u_n\in \check\sF,\,n\geq1$ such that $\lim\limits_{n\to\infty}\check\sE (u_n,u_n)=0$ and $\|u_n\|_{L^2(\partial D;\sigma)}=\sigma(\partial D)^{1/2}$ for every $n\geq 1$. By \eqref{eqn316} and {\rm (LS)} for $\Theta_{\Psi, \sigma}$, we have  
\[
\Big\|  u_n|_{\partial D}-\oint_{\partial D}u_nd\sigma \Big\|^2_{L^2(\partial D;\sigma)}\leq \int_{\partial D}\int_{\partial D}(u_n(x)-u_n(y))^2\sigma(dx)\sigma(dy)
\lesssim \lb u_n|_{\partial D}\rb^2_{\Lambda_{\Psi,  \sigma} }\to 0
\]
and so $u_n|_{\partial D}$ converges in $L^2(\partial D;\sigma)$ to the constant function $\1_{\partial D}$
as $n\to \infty$. Hence, by the lower-semicontinuity of $\check\sE$, 
\[
\bar\sE\big(\sH \1_{\partial D},\sH\1_{\partial D})=\check\sE(\1_{\partial D},\1_{\partial D}\big)
\leq \liminf_{n\to\infty}\check\sE(u_n,u_n)
=0.
\]
However, $\bar\sE\big(\sH \1_{\partial D},\sH\1_{\partial D})>0$ as $(\bar\sE,\bar\sF)$ is transient. This is  a contradiction. 

Finally, combining \eqref{eqn315}, \eqref{eqn316} and \eqref{eqn317}, we see the desired estimate. $\Lambda_{\Psi,  \sigma} \cap C(\partial D)$ is a core of $\check \sF$ by  the  same reason as that for (a). 
\end{proof}

\begin{remark}\rm  
In \cite{GS}, a similar trace theorem is established on an unbounded uniform domain in a different setting. Let $(D,d,m)$ be an unbounded, locally compact, non-complete, doubling metric measure space  
that supports a $p$-Poincar\'e inequality for some $1\leq p<\infty$, and in addition $D$ be a uniform domain in its completion $(\bar D,d)$ with unbounded boundary $\partial D:=\bar D\setminus D$. 
 Let $D^{1,p}(D)$ be the Dirichlet-Sobolev space defined in terms of upper gradients. Note that in this setting, a capacity upper bound condition is automatically satisfied, in fact, for any $z\in D$ and $r>0$, there is a Lipschitz bump function $\varphi$ for $B(z,r)\subsetneq B(z,2r)$ with Lipschitz constant $1/r$, so 
\[
\|\varphi\|^p_{D^{1,p}}\leq V(x,2r)/r^p. 
\]	
The restriction of $D^{1,p}(D)$ onto the boundary $\partial D$ as functions in the homogeneous Besov space $HB_{p,p}^{1-(\theta/p}(\partial D)$ is considered in \cite{GS}, under the condition that there is  a non-atomic Borel regular measure $\sigma$ on $\partial D$ that satisfies  $\theta$-codimensional condition with respect to $m$ for some $0<\theta<p$ in the sense that 
\[
\frac{V_\sigma (x,r)}{V (x,r)}\asymp r^{-\theta}
\]
 When $p=2$,   the above  $\theta$-codimensional condition  implies  that our (LS)  holds for     $\Theta_{\Psi, \sigma}$  with  $\Psi(r)=r^2$. We will see  Theorem \ref{T:6.1} below that our (LS) condition is natural in the sense that (VD) condition for $\sigma$ plus (LS) condition for $\Theta_{\Psi,\sigma}$ is equivalent to the capacity density condition \eqref{e:6.1}.

\smallskip   

Both \cite{GS} and our paper use the Poincaré inequality along a chain of balls to prove the restriction theorems, though there are some differences in details.  It is likely that the approach in \cite{GS} can also be modified 
 to prove Theorems \ref{thmrestriction1} and \ref{thmrestriction2}. For the extension theorems,
 the use of Whitney cover method  to define the extension map is nowadays standard.
This method can be traced back to Whitney \cite{Wh}, and has been used by many authors, see, e.g.,  \cite{St} on Euclidean spaces and     \cite{BS, GS,HK} on more general state spaces.       \qed
\end{remark}

\medskip

In the {remaining part} of the paper, we will focus on harmonic measures. In particular, in Section \ref{S:5}, we will see that once condition {\rm(VD)} and {\rm (LS)}  hold for some $\sigma$, then it holds for the harmonic measure when $\partial D$ is bounded or the elliptic measure from $\infty$ when $\partial D$ is unbounded.

\section{Volume doubling of harmonic measures}\label{S:4}
Recall that $(D,d)$ is an $A$-uniform domain with $(\tD,d)$ being its completion. For each $x\in D$, we denote by $\omega_x$ the harmonic measure of the reflected diffusion process $\bar X$ on $\partial D$ starting from $x$; that is,  
\[
\int_{\partial D} f(y)\omega_x(dy)=\bE_x[f(\bar X_{\tau_D})]  \quad 
\hbox{ for each }f\in C_c(\partial D).
\]
In this section, we study the doubling property of the harmonic measure. Let $\Psi$ be a continuous bijection from $(0,\infty)$ to $(0,\infty)$ that enjoys the property \eqref{eqnpsi}. 
 
\medskip

\noindent\textbf{HK($\Psi$)}: We say that $(\tD,d,m_0,\bar\sE,\bar\sF)$ satisfies the heat kernel estimate {\rm HK}($\Psi$) if its associated diffusion process $\bar X$ has a transition density function $\bar p(t, x, y)$ with respect to the measure $m_0$ on $\tD$, and that there are positive constants  $c_1,c_2,c_3,c_4$  so  that for every $t>0$,
\begin{eqnarray}
\bar{p}(t, x,y)&\leq & \frac{c_1}{V(x,\Psi^{-1}(t))}\exp\big(-c_2t\Phi(c_3 d(x,y) / t)\big)
   \  \hbox{ for $m_0$-a.e. }  x,y\in \tD,   \label{e:4.1}\\
\bar{p}(t, x,y)&\geq&  \frac{c_4}{V(x,\Psi^{-1}(t))}
\quad  \hbox{ for $m_0$-a.e. }x,y\in \tD\hbox{ with }d(x,y)\leq \Psi^{-1}(t).   \label{e:4.2}
\end{eqnarray}
Here $\Phi(s):=\sup_{r>0}(\frac{s}{r}-\frac{1}{\Psi(r)})$. 

\smallskip

\begin{remark} \label{R:4.1} \rm
\begin{enumerate} [(i)]
   \item   It follows from \cite[Proposition 3.1(b)]{CKW} that 
  the lower bound estimate \eqref{e:4.2} in condition \textbf{HK($\Psi$)}
    implies that the diffusion process $\bar X$ is conservative. 
  
 \smallskip
 
 \item It is known  (see \cite[Theorem 3.1]{BGK}) 
 that if  \textbf{HK($\Psi$)} holds, then $\bar{p}(t, x, y)$ has a jointly 
  continuous modification on $(0, \infty) \times \tD \times \tD$. Thus
  the heat kernel estimates hold pointwise for this jointly continuous modification.  In particular, the diffusion process can be modified to start from every point from $\tD$, which is a Feller process on $\tD$ having strong Feller property. 
 
 \smallskip
 
 \item  It is known that under (VD),  {\rm HK}($\Psi$) is equivalent to {\rm PI}($\Psi$) and {\rm CS}($\Psi$),
where {\rm CS}($\Psi$) is a cutoff Sobolev inequality condition that implies {\rm Cap}$_\leq (\Psi)$;
see \cite{AB, BB6, BBK, GHL}  and \cite[Remark 2.9]{Mathav}.
When $\tD$ is unbounded, it is shown in \cite[Theorem 1.2]{GHL} that 
{\rm HK}($\Psi$)  implies  {\rm Cap}$ (\Psi)$.   \qed
\end{enumerate} 
\end{remark}

\begin{theorem}\label{thm41}
Assume that the reflected Dirichlet form 
$(\tD,  d, m_0, \bar \sE, \bar \sF)$ satisfies the heat kernel estimate condition {\rm HK($\Psi$)}, and assume that $\overline  {\rm Cap}\big(B(x,r)\cap\partial D,B(x,2r)\big)>0$ for each $x\in \partial D$ and $r\in (0,\diam(\partial D)/3)$. The following conditions are equivalent:
\begin{enumerate} [\rm (i)]   
\item  (Relative boundary capacity doubling property) There are  constants $C_1>1$ and $C_2>0$ so that 
\[
	 \overline  {\rm Cap}\big(B(x,2r)\cap\partial D,B(x,4r)\big)
	 \leq C_1 \overline  {\rm Cap}\big(B(x,r)\cap\partial D,B(x,4r)\big) 
\]	
for each $0<r<C_2\,\diam(\partial D)$ and $x\in \partial D$.
		
\item (Harmonic measure doubling property) There are constants $C_1>1$ and $C_2>0$ so that 
\[
\omega_{x_0}\big(B(x,2r)\big) \leq C_1\omega_{x_0}\big(B(x,r)\big) 
\]
for each $x\in\partial D$, $x_0\in D$ and $0<r<\frac{d(x,x_0)}4\wedge (C_2\,\diam(\partial D))$. 
\end{enumerate}
 The constants $C_1,C_2$ depend only on the constants of the other condition and the parameters in \eqref{eqnpsi}, {\rm(VD)} of $m_0$ and {\rm HK($\Psi$)}.
\end{theorem}

 Here, we remark that if (ii) holds and $d(x_0,\partial D)\geq c\,\diam(\partial D)$ for some $c>0$, then $\omega_{x_0}$ is a doubling measure with the doubling constant depending only on $C_1,C_2,c$ and the constant of (VD) for $m_0$. Indeed, if $r>r_0:=(\frac{c}{4}\wedge C_2)\diam(\partial D)$, we can find a finite cover $B(z_i,r_0/2),1\leq i\leq N$ of $\partial D$ where $z_i\in\partial D$ for $1\leq i\leq N$, as $(\tD,d)$ is metric doubling. Then, for any $r\geq r_0$ and $x\in\partial D$,
\[
\omega_{x_0}(B(x,r))\geq \inf_{1\leq i\leq N}\omega_{x_0}(B(z_i,r_0/2))\geq \frac{\inf_{1\leq i\leq N}\omega_{x_0}(B(z_i,r_0/2))}{\omega_{x_0}(\partial D)}\cdot\omega_{x_0}(B(x,2r)).
\]

\medskip

The proof of Theorem \ref{thm41} will be given in \S \ref{S:4.2}.

\subsection{Harnack principles} \label{S:4.1}
We need the elliptic Harnack principle and  the scale invariant boundary Harnack principle for harmonic functions. We first recall
 the definition of harmonicity and Dirichlet boundary condition. See papers \cite{BCM,BM,GHL,L} for a reference. \smallskip  

\begin{definition} 
Let $V$ be a proper open subset of $\tD$. 

\begin{enumerate} 
\item[\rm (a)] Let $V\subset\tD$ be an open subset and $f\in\bar\sF_{loc}(V)$. We say $f$ is harmonic in $V$ (with respect to $\bar{X}$) if for each relatively compact open subset $O$ of  $V$,
\[
\bE_x[|f(\bar{X}_{\tau_O})|; \, \tau_O<\infty]<\infty 
\quad \hbox{and} \quad 
f(x)=\bE_x[f(\bar{X}_{\tau_O});\tau_O<\infty]
\quad \hbox{ for $\bar\sE$-q.e. }x\in O
\]
We say $f$ is regular harmonic in $V$ if \[\bE_x[|f(\bar{X}_{\tau_V})| ;  \, \tau_V<\infty ]<\infty\hbox{ and }f(x)=\bE_x[f(\bar{X}_{\tau_V});  \, \tau_V<\infty ]\hbox{ for $\bar\sE$-q.e. }x\in V.\]

\item[\rm (b)]  We say $f\in\sF_{\rm loc}(D)$ satisfies the Dirichlet boundary condition along $V\cap \partial D$ if for every open subset $O\subset V\cap D$ relatively compact in $\tD$ such that $d(O,D\setminus V)>0$, there is $u\in  \sF^0$ such that $f =u$ $m$-a.e. on $O$.  
\end{enumerate} 
\end{definition}

\begin{remark} \label{remark44}\rm 
If $f\in \bar\sF_e\cap L^\infty(\tD;m_0)$ and $f=0$ $\bar{\sE}$-q.e. on $V\cap\partial D$ for some open $V\subset \tD$, then $f$ satisfies the Dirichlet boundary condition along $V\cap \partial D$. In fact, for each open $O\subset V\cap D$ relatively compact in $\tD$ such that $d(O,D\setminus V)>0$, we can find $\psi\in \bar\sF\cap C_c(\tD)$ such that $0\leq \psi\leq 1$, $\psi|_O=1$ and $\psi|_{\tD\setminus V}=0$, then $\psi\cdot f\in \sF^0$ and $(\psi\cdot f)|_O=f|_O$.   \qed
\end{remark}

\medskip

\begin{definition}\label{D:4.5}  \rm  \begin{enumerate}
\item[(i)] We say that $(\tD,d,m_0,\bar\sE,\bar\sF)$ satisfies the elliptic Harnack principle (EHP) if there are $C_{1,e},C_{2,e}\in (1,\infty)$ so that 
\[
h(x)\leq C_{1,e}\,h(y) \quad \hbox{ for each } x, y\in B(x_0 ,r/ C_{2,e})
\]
for each $x_0\in \tD$,   $r>0$ and  non-negative function $h$ that is harmonic in $B(x,r)$.\smallskip

\item[(ii)] We say that the scale invariant boundary Harnack principle (BHP) on $D$ holds for $(\tD,d,m_0,\bar\sE,\bar\sF)$ if there are $C_{1,b},C_{2,b}\in (1,\infty)$ so that 
\begin{align}\label{BHP}
f(y)g(z)\leq C_{1,b}f(z)g(y)  \quad \hbox{ for $\bar\sE$-q.e. }   y,z \in B(z ,r/C_{2,b}),
\end{align}
where $x\in \partial D$, $r\in (0,\operatorname{diam}(\partial D)/2)$ and $f,g$ are non-negative harmonic in $D\cap B(x,r)$ that satisfy the Dirichlet boundary condition along $\partial D\cap B(x,r)$. 
\end{enumerate}
\end{definition}

If (EHP)   holds for $(D, d, m, \sE^0, \sF^0)$, 
then every harmonic function has a  locally H\"older continuous version. Representing each harmonic function by its continuous version,  then \eqref{BHP} is equivalent  to  holding  for every $y,z\in B(z,r/C_{2,b})$.

\medskip 

It is known that under the  {\rm HK}($\Psi$)  condition  for $(\tD, d, m_0, \bar \sE, \bar \sF)$, (EHP) and (BHP) hold on $D$; see \cite[Theorem 1.2]{GHL} and  \cite[Theorem 1.1]{AC}.  Conversely, if (EHP) holds for  $(\tD, d, m_0, \bar \sE, \bar \sF)$, then by \cite[Theorem 7.9]{BCM} (noting that 
$(\tD, d, m_0)$ is (VD)), there is a metric $\bar d$ on $\sX$ that is quasisymmetric to $d$ and a smooth Radon measure $\mu$ having 
full $\bar \sE$-support on $\sX$ so that the time-changed Dirichlet space $(\tD, \bar d, \mu, \bar \sE, \bar \sF_e \cap L^2(\sX; \mu))$
 has property {\rm HK}($\Psi$) for $\Psi(r)=r^\beta$ for some $\beta \geq 2$. 
 It is well known \cite{CF, FOT} that the trace Dirichlet spaces of $(\tD, d, m_0, \bar \sE, \bar \sF)$  and $(\tD, \bar d, \mu, \bar \sE, \bar \sF_e \cap L^2(\sX; \mu))$ on $\partial D$ are related through a strictly increasing continuous time-change. Thus as far as trace Dirichlet spaces are concerned,
 assuming (EHP) holds for  $(\tD, d, m_0, \bar \sE, \bar \sF)$ is essentially equivalent to assuming 
  {\rm HK}($\Psi$)  holds for  $(\tD, d, m_0, \bar \sE, \bar \sF)$ up to a time change.

\medskip

\begin{lemma}\label{lemma44}
Assume that {\rm HK($\Psi$)} holds for $(\tD,  d, m_0, \bar \sE, \bar \sF)$. Let $s>A+\frac12$ and $t\in (0,1)$. 
	
\begin{enumerate} [\rm (a)] 
\item There is a constant $C_{s,t}\in (0,1)$ so that
\[
f(y)\geq C_{s,t} f(z) \quad \hbox{for every }x\in \partial D\hbox{ and }y,z\in \overline{B(x,r)}\cap D_{tr}
\]
for $r\in (0,\diam(D)/2s)$ and  non-negative  $f$ that is harmonic in $B(x,2sr)\cap D$.

\item There is a constant $C_s\in (0,1)$  so that 
\[
f(y)\geq C_sf(z)   
\quad \hbox{for each }x\in\partial D\hbox{ and }y,z\in    \partial B(x,r) 
\] 
for $r\in (0,\diam(D)/2s)$ and non-negative function $f$ that is harmonic in 
\[
E=\big(D\cap B(x,2sr)\big)\cup \Big( \partial D\cap \big(B(x,3r/2)\setminus\overline{B(x,r/2)}\big)\Big).
\]

\item  There is a constant $C_s\in (0,1)$  so that 
\[
f(y)g(z)\geq C_s\,f(z)g(y)\hbox{ for each }x\in\partial D,\,y,z\in \partial B(x,r)
\] 
for $r\in (0,\diam(D)/2s)$ and non-negative functions $f, g$ that are harmonic in $B(x,2sr)\cap D$ and satisfy the Dirichlet boundary condition along $\partial D\cap \big(B(x,2sr)\setminus \overline{B(x,r/2)}\big)$.
\end{enumerate} 
\end{lemma}

\begin{proof}
(a). For each $y,z\in \overline{B(x,r)}\cap D_{tr}$,  by Lemma \ref{lemma21}(b), 
there is  a path $\gamma$ in $\overline{B(x,(2A+1)r)}\cap D_{tr/(1+A)}\subset B(x,2sr)\cap D$ connecting $y,z$. Conclusion (a) then follows from   Lemma \ref{lemma2path}(b),
 a routine Harnack chain argument and (EHP). 
	
(b). For each $y,z\in D\cap\partial B(x,r)$, by the property of $D$ being an
$A$-uniform domain, there is  a path $\gamma$ connecting $y,z$ in $  D\cap  \overline{B(x,(2A+1)r)}\subset   D\cap  B(x,2sr)$, 
 and   that for every $\omega\in\gamma$,
\begin{align*}
d(w,\partial D\setminus E)
&\geq \max\Big\{\frac{r}{2}-d(y,w), \frac{r}{2}-d(z,w),\frac{d(w,y)\wedge d(z,w)}{A}\Big\}\geq \frac{r}{2(A+1)}.
\end{align*}
The inequality of (b) then holds for $y,z$ by Lemma \ref{lemma2path}(b), a routine Harnack chain argument and (EHP). Notice that $f$ is continuous in $E$, the inequality extends to $\partial B(x,r)$.  

(c). Let $0<\lambda<\min\{\frac{s}{A+1/2}-1,1\}$ and $t=\frac{\lambda}{16AC_{2,b}}$. By (a), we know that 
\begin{equation}\label{eqn41}
f(y)\geq C f(z), g(y)\geq C g(z) \hbox{ for every }y,z\in \overline{B\big(x,(1+\lambda)r\big)}\cap D_{tr}
\end{equation}
for some $C$ independent of $f,g,r,x$. By (BHP), we can show that for each $y\in \partial B(x,r)\setminus D_{tr}$
\begin{equation}\label{eqn42}
 C_{1,b}^{-1}\,\frac{f(z)}{g(z)}\leq \frac{f(y)}{g(y)}\leq C_{1,b}\,\frac{f(z)}{g(z)} \hbox{ for some }z\in \overline{B\big(x,(1+\lambda)r\big)}\cap D_{tr}.
\end{equation}
In fact, we can find $y'\in \partial D$ such that $d(y,y')<tr$, so $r-tr<d(x,y')<r+tr$. Then, we apply (BHP) to the ball $B(y',4AC_{2,b}tr)=B(y',\lambda r/4)\subset B(x,(1+\lambda)r)\setminus \overline{B(x,r/2)}$ to find $z\in B(y',4Atr)\cap D_{tr}\subset \overline{B\big(x,(1+\lambda)r\big)}\cap D_{tr}$ such that \eqref{eqn42} holds. Note that $B(y',4Atr)\cap D_{tr}\neq \emptyset$ by Lemma \ref{lemma21}. 
    
The desired estimate follows immediately from \eqref{eqn41} and \eqref{eqn42}. 
\end{proof}

\subsection{Proof of Theorem \ref{thm41}}\label{S:4.2}
We prove Theorem \ref{thm41} in this subsection. First, we introduce some more notations.

For open $U\subset \tD$ such that $\tD\setminus U$ is not $\bar\sE$-polar, we let $\bar{p}_U (t, x,y)$ be the transition density function  of the subprocess $\bar{X}^U$ of $\bar{X}$  killed upon leaving $U$: $\bar\bP_x(\bar{X}_t\in E;t<\tau_U)=\int_E \bar{p}_U(t,x,y)m_0(dy)$. Since we always assume HK$(\Psi)$, $\bar{p}_U (t, x,y)$ is well defined and jointly continuous on $(0,\infty)\times U^2$. Let 
\[\bar{g}_U(x,y)=\int_0^\infty \bar{p}_U(t, x,y)dt\] 
be the Green's function on $U$. By \cite[Theorem 4.4, Remark 2.7(ii), Proposition 2.9(iii)]{BCM}, the green's function $\bar{g}_U(x,\cdot)$ is in $\bar\sF_{\rm loc}(U\setminus \{x\})$ and satisfies the Dirichlet boundary condition along $\partial U$.

\begin{lemma}\label{lemma45}
Suppose that {\rm HK($\Psi$)} holds for $(\tD,  d, m_0, \bar \sE, \bar \sF)$. Let $U$  be an open subset of $\tD$, $x,y\in U$ and $r>0$. Suppose  that  $\gamma$ in $U$ is a continuous curve with $\gamma(0)=x$, $\gamma(1)=y$ and $d(\gamma,\tD\setminus U)\geq r$. Then there is a positive constant $C_s$ depending  only on the bounds of {\rm HK($\Psi$)}, {\rm (VD)} 
and $s:=\diam(\gamma)/r$
so that 
\[
\bar\bP_x(\sigma_{B(y,r)}<\tau_{U})>C_s  . 
\] 
Moreover, there is $\eta>1$ depending only on the bounds of {\rm HK($\Psi$) and (VD)} so that $\bar{g}_{U}(u,v)>0$ for $u\in B(x,r/\eta)$ and $v\in B(y,r/\eta)$. 
\end{lemma}

\begin{proof}
By {\rm HK}($\Psi$),  (VD)  and strong Markov property of $\bar X$, there exists $\lambda\in (2,\infty)$  so that for each $x\in D$, $0< \rho<\diam(D)/\lambda$,  $z\in B(x,\rho)$ and $w\in B(z,\rho)$,
\begin{eqnarray}\label{e:4.3} 
 &&  \bar p_{B(x,\lambda\rho)}\big(\Psi(\rho),z,  w\big)=\bar p\big(\Psi(\rho),z,  w\big)-\bE_z[\bar p\big(\Psi(\rho),X_{\tau_{B(x,\lambda\rho)}}, w\big);\tau_{B(x,\lambda\rho)}<\Psi(\rho)]  \nonumber \\
&\geq & \bar p\big(\Psi(\rho),z,  w\big)-\max_{y\in\partial B(x,\lambda\rho)}\bar p\big(\Psi(\rho),y, w\big)   \nonumber  \\
&\geq & \frac{c_4}{V(z,\rho)}-\max_{y\in\partial B(x,\lambda\rho)}\frac{c_1}{V(y,\rho)}\exp\big(-c_2\Psi(\rho)
    \Phi(c_3(\lambda-2)\rho / \Psi(\rho) )\big)  \nonumber \\
&\geq& \frac{c_4}{V(z,\rho)}-\frac{c_1}{V(z, \rho)} \max_{y\in\partial B(x,\lambda\rho)}
\frac{V(y, (1+\lambda)\rho )}{V (y,\rho)}\exp\big(-c_2(c_3(\lambda-2)-1)\big)  \nonumber  \\ 
&\geq& \frac{c_4}{ V (z,\rho) }   -    \frac{c_1}{V(z, \rho)}  \tilde c_1 (1+\lambda)^{d_1}
 \exp\big(-c_2(c_3(\lambda-2)-1)\big)  \nonumber  \\ 
& \geq & \frac{c_4/2}{V (x,\rho)}, 
  \end{eqnarray}
 where $c_1,c_2,c_3,c_4$ are constants of HK$(\Psi)$,  and $\tilde c_1$ and $d_1$ are the parameters in \eqref{e:VD}. 
  In the third to the last inequality above, we used the fact that 
$\Phi (s) \geq \frac{s}{\rho} - \frac{1}{\Psi (\rho)}$.   For notational convenience, set $C_1=c_4/2$. For $a=r/\lambda$, by Lemma \ref{lemma2path}(b), we can find a sequence $x=z_0,z_1,\cdots,z_l=y\in\gamma$ such that $d(z_i,z_{i+1})<a/3$ for $0\leq i\leq l-1$ and $l$ has an upper bound depending only 
on $\diam(\gamma)/a=\lambda s$.  Thus  for $p_0\in B(x,a/3)$ and $p_l\in B(x,a/3)$, there is $C_2>0$ depending only on the parameters of (VD) so that 
\begin{align}
		&\quad\ \bar{p}_U(l\Psi(a),p_0,p_l)\nonumber\\
	&\geq \int_{p_1\in B(z_1,a/3)}\cdots\int_{p_{l-1}\in B(z_{l-1},a/3)} \bar p_U(\Psi(a),p_0,p_1)\cdots \bar p_U(\Psi(a),p_{l-1},p_l)m_0(dp_{l-1})\cdots m_0(dp_1)\nonumber\\
	&\geq \int_{p_1\in B(z_1,a/3)}\cdots\int_{p_{l-1}\in B(z_{l-1},a/3)}\frac{C_1}{V(z_1,a)}\cdots \frac{C_1}{V(z_l,a)} m_0(dp_{l-1})\cdots m_0(dp_1)\nonumber\\
	&\geq \frac{C_1^lC_2^{l-1}}{V (z_l,a) }= \frac{C_1^lC_2^{l-1}}{V (y,a) } ,  	 \label{eqn46}
\end{align}
where the third inequality holds due to \eqref{e:4.3}, and the facts $B(z_i,\lambda a)\subset U$, $p_i\in B(z_i,a/3)$, $d(p_{i-1},p_i)<a$ for $i=1,\cdots,l$. Note that $\bar{p}_U(t,u,v)$ has a jointly continuous modification in $(0,\infty)\times U\times U$ by the parabolic Harnack principle \cite[Theorem 1.3]{BGK}.
It  follows that $\bar{g}_U(p_0,p_l)>0$;  that is, the second statement of the lemma holds with $\eta=3\lambda$. 
The first statement follows by integrating \eqref{eqn46} over $B(z_l,a/3)$
\[
\bar\bP_x(\sigma_{B(y,r)}<\tau_{U})\geq \int_{B(y,a/3)}\bar{p}_U(l\Psi(a),x,p_l)m_0(dp_l)
 \geq  \frac{C_1^lC_2^{l-1}}{  V (y ,a) } V(y, a/3)
\geq C_1^lC_2^l,
\]
where in the last inequality we used the (VD) property that $V(x, a/3)\geq C_2 V(x,a)$. 
\end{proof}

\begin{lemma}\label{lemma46}
Assume that {\rm HK($\Psi$)} holds for $(\tD,  d, m_0, \bar \sE, \bar \sF)$. For each $s>1$, there is $C_s\in (0,1)$ such that 
\[
\overline  {\rm Cap}\big(\partial D\cap B(x,r),B(x,2sr)\big)\geq C_s\overline  {\rm Cap}\big(\partial D\cap B(x,r),B(x,2r)\big)
\]
for each $x\in \partial D$ and $r<\diam(D)/(6s)$.
\end{lemma}
\begin{proof}
  The result can be proved by using the argument of \cite[Lemma 3.12 and Corollary 3.13]{BM2}. We provide a different proof here.  

Define $h_s\in \bar\sF$ by $h_s(x)=\bar \bP_x(\sigma_{\partial D\cap B(x,r)}<\tau_{B(x,2sr)})$ for each $x\in \tD$. 

For each  $y\in D\cap\partial B(x,2r)$,  as $D$ is $A$-uniform domain, there is a a path $\gamma\subset D$ so that $\gamma(0)=y$, $\gamma(1)\in D\cap \partial B(x,3sr)$ and $d_D(z)> ( d(z,\gamma(1)) \wedge d(z, y)) /A$ for each $z\in \gamma$. For  every $z\in\gamma$,
\begin{eqnarray} \label{e:4.4a} 
d(z,\partial D\cap B(x,r))
&\geq & \max\big\{d(z,y)\wedge d(z,\gamma(1))/A,\,r-d(z,y),\,3sr-r-d(z,\gamma(1))\big\}  \nonumber \\
&\geq & \max\big\{d(z,y)\wedge d(z,\gamma(1))/A,\,r-d(z,y)\wedge d(z,\gamma(1))\big\}. \nonumber \\
&\geq & r/(1+A).
\end{eqnarray}
Then, by Lemma \ref{lemma45}, 
\begin{eqnarray} \label{e:4.5a}
h_s(y)&=& 1-\bar\bP_y(\tau_{B(x,2sr)}<\sigma_{B(x,r)\cap \partial D}) \nonumber \\
&\leq & 1-\bar\bP_y(\sigma_{B(\gamma(1),r/(1+A))}<\sigma_{B(x,r)\cap\partial D})\leq C_{2,s}<1.
\end{eqnarray}
Note that by (EHP), $h_s$ is continuous in $B(x,2sr)\setminus (\partial D\cap B(x,r))$.
So the above estimate holds for each $y\in\partial B(x,2r)$.  
Consequently, 
 $\frac{h_s-C_{2,s}}{1-C_{2,s}} \big|_{B(x,r)\cap \partial D}=1$ and $\frac{h_s-C_{2,s}}{1-C_{2,s}}|_{\partial B(x,2r)}\leq 0$. Hence, by the Markov property
\begin{align*}
&\quad\ \overline  {\rm Cap}\big(\partial D\cap B(x,r),B(x,2r)\big)=\bar\sE(h_1,h_1)\\
&\leq \bar\sE\big((\frac{h_s-C_{2,s}}{1-C_{2,s}})^+,(\frac{h_s-C_{2,s}}{1-C_{2,s}})^+\big)=(1-C_{2,s})^{-2}\bar\sE\big((h_s-C_{2,s})^+,(h_s-C_{2,s})^+\big)\\
&\leq (1-C_{2,s})^{-2}\bar \sE (h_s,h_s)=(1-C_{2,s})^{-2}\overline  {\rm Cap}\big(\partial D\cap B(x,r),B(x,2sr)\big).
\end{align*}
This completes the proof of the lemma.
\end{proof}

The main step of the proof of Theorem \ref{thm41} is the following local version of the theorem. 

\begin{proposition}\label{prop47}
Assume that {\rm HK$(\Psi)$} holds for $(\tD,d,m_0,\bar\sE,\bar\sF)$ and $\overline {\rm Cap}\big(B(x,r)\cap\partial D,B(x,2r)\big)>0$ for each $x\in \partial D$ and $r\in (0,\diam(\partial D)/3)$. The following are equivalent.
\begin{enumerate} [\rm (i)] 
\item There are constants $C_1\geq 1$ and $0<C_2<1$   so that 
\[
\overline{\rm Cap}\big(B(x,2r)\cap\partial D,B(x,4r)\big) \leq C_1\overline{\rm Cap}\big(B(x,r)\cap\partial D,B(x,4r)\big) 
\]
for each $0<r<C_2\diam(\partial D)$ and $x\in \partial D$.

\item For each $s>\max\{(2A+1)C_{2,e},8A+4\}$ where $C_{1,e},C_{2,e}$ are the constants in {\rm(EHP)}, there are constants  $C_1\geq 1$ and $0<C_2<1$  so that 
$$
\bar \bP_y\big(\sigma_{\partial D\cap B(x,2r)}=\tau_{B(x,2sr)\cap D}\big)\leq C_1 \bar \bP_y\big(\sigma_{\partial D\cap B(x,r/(6+2A))}=\tau_{B(x,2sr)\cap D}\big)
$$
 for each $x\in \partial D$, $0<r<C_2\diam(\partial D)$ and $y\in \partial B(x,4r)$. 
\end{enumerate}
\end{proposition}

\begin{proof}
Let's fix $s>(2A+1)C_{2,e}$, $x\in \partial D$ and $r<\frac{\diam(D)}{6s}$. For $\lambda\in (0,2]$, we define $f_\lambda$ and $h_\lambda$  by
\begin{align*}
f_\lambda(y)&:=\bar \bP_y(\sigma_{\partial D\cap B(x,\lambda r)}=\tau_{B(x,2sr)\cap D})&\hbox{ for each }y\in \tD,\\
h_\lambda(y)&:=\bar \bP_y\big(\sigma_{\partial D\cap B(x,\lambda r)}<\tau_{B(x,2sr)}\big)&\hbox{ for each }y\in \tD.
\end{align*}
By \cite[Lemma 2.2.10, Theorem 2.2.5]{FOT} and \cite[Lemma 6.5]{GH},  there is  a Radon measure $\sigma_\lambda$ supported on $\overline{B(x,\lambda r)}\cap \partial D$ such that 
\begin{align}
\label{e:4.4}
&\sigma_\lambda\big(\partial D\cap \overline{B(x,\lambda r)}\big)=\overline  {\rm Cap}\big(B(x,\lambda r)\cap\partial D,B(x,2sr)\big),\\
\label{e:4.5}
&h_\lambda(y)=\int_{\partial D\cap \overline{B(x,\lambda r)}} \bar g_{B(x,2sr)}(y,z)\sigma_\lambda(dz)\quad\hbox{ for every }y\in B(x,2sr).
\end{align}
  For every $ u,v\in D\cap\overline{B(x,2C_{2,e}r)}$, there is a path $\gamma$ connecting $u,v$ in $\overline{B(x,(2A+1)2C_{2,e}r)}$. Note that $d(\gamma,\tD\setminus B(x,2sr))\geq \rho:=sr-(2A+1)2C_{2,e}r$. By Lemma \ref{lemma45},  $\bar{g}_{B(x,2sr)}(y,z)>0$ for $y\in B(u,\rho/\eta)$ and $z\in B(v,\rho/\eta)$ with $u,v\in D\cap\overline{B(x,2C_{2,e}r)}$. It follows that  $\bar{g}_{B(x,2sr)}(y,z)>0$ for every $y,z\in \overline{B(x,2C_{2,e}r)}$.
  So by \eqref{e:4.4} and \eqref{e:4.5}. 
\[
h_\lambda(y)>0\quad\hbox{ for each }\lambda\in(0,2],\,y\in\overline{B(x,2C_{2,e}r)}.
\] 

By (EHP), for each $y\in B(x,2sr)\setminus B(x,2C_{2,e}r)$, we have 
\begin{equation}\label{e:4.6}
C_{1,e}^{-1}\bar{g}_{B(x,2sr)}(y,x)\leq \bar{g}_{B(x,2sr)}(y,z)\leq C_{1,e}\bar{g}_{B(x,2sr)}(y,x)\quad\hbox{ for every }z\in \overline{B(x,2r)}.
\end{equation}
Hence, by \eqref{e:4.4}, \eqref{e:4.5}, \eqref{e:4.6} 
\begin{equation}\label{e:4.7}
C_{1,e}^{-1}\frac{\overline  {\rm Cap}\big(B(x,\lambda r)\cap\partial D,B(x,2sr)\big)}{\overline  {\rm Cap}\big(B(x,\lambda'r)\cap\partial D,B(x,2sr)\big)}
\leq \frac{h_\lambda(y)}{h_{\lambda'}(y)}
\leq C_{1,e}\frac{\overline  {\rm Cap}\big(B(x,\lambda r)\cap\partial D,B(x,2sr)\big)}{\overline  {\rm Cap}\big(B(x,\lambda'r)\cap\partial D,B(x,2sr)\big)}
\end{equation}
for each $\lambda,\lambda'\in (0,2]$, $y\in B(x,2sr)\setminus B(x,2C_{2,e}r)$. Moreover, by Lemma \ref{lemma44}(a), 
 \eqref{e:4.7} holds for each $y\in \overline{B(x,4r)}\cap D_{r/A}$.

Next, we fix $\lambda\in [1/(6+2A),2]$ and write 
\[
M:=\max_{y\in \partial B(x,\lambda r/(3+A))} h_{\lambda/(6+2A)}(y).
\]
For each $z\in  D\cap  (B(x,2sr)\setminus B(x,\lambda r))$, 
there is  a path $\gamma$  in $D$ so that
 $\gamma(0)=z$, $\gamma(1)\in D\cap \partial B(x,5sr/2)$ and satisfies the defining property of $A$-uniform domain. For each $w\in\gamma$, 
\begin{eqnarray*}
&& \ d(w,\overline{B(x,\lambda r/(3+A))})\\
&\geq&\max\Big\{d(z,\overline{B(x,\lambda r/(3+A))})-d(w,z), \, d(\gamma(1),\overline{B(x,\lambda r/(3+A))})-d(w,\gamma(1)), \, d_D(w)-\frac{\lambda r}{3+A}\Big\}  \\
&\geq&  \max \Big\{\frac{2+A}{3+A}\lambda r-d(w,z),\frac52sr-\frac{\lambda r}{3+A}-d(w,\gamma(1)),\frac{d(w,z)\wedge d(w,\gamma(1))}{A}-\frac{\lambda r}{3+A}\Big\}  \\
&\geq & \max \Big\{\frac{2+A}{3+A}\lambda r-d(w,z)\wedge d(w,\gamma(1)),\frac{d(w,z)\wedge d(w,\gamma(1))}{A} \Big\}-\frac{\lambda r}{3+A}\\
&\geq & \frac{\lambda r }{1+A}\frac{2+A}{3+A} -\frac{\lambda r}{3+A}=\frac{A\lambda r}{(1+A)(3+A)}. 
\end{eqnarray*}
Hence, by Lemma \ref{lemma45}, for $ z\in D\cap(B(x,2sr)\setminus B(x,\lambda r))$,
\[
\bar \bP_z \left(\tau_{B(x,2sr)}<\sigma_{\overline{B(x,\lambda r/(3+A))}} \right)\geq C_3.
\]
Note that $\bar\bP_z \big(\tau_{B(x,2sr)}<\sigma_{\overline{B(x,\lambda r/(3+A))}} \big)$, as a function of $z$, is continuous in $B(x,2sr)\setminus\overline{B(x,\lambda r/(3+A))}$ by (EHP). So the estimate holds for $z\in \partial D\cap(B(x,2sr)\setminus B(x,\lambda r))$ as well. Next, the above estimate extends to 
$z\in \tD\setminus B(x,\lambda r)$ by the strong Markov property and continuity of $\bar X$. So 
\begin{equation}\label{e:4.8}
\begin{aligned}
h_{\lambda/(6+2A)}(z)&\leq \bar \bP_z(\tau_{B(x,2sr)}>\sigma_{\overline{B(x,\lambda r/(3+A)}})\sup_{y\in \partial B(x,\lambda r/(3+A))} h_{\lambda/(6+2A)}(y)\\
&\leq (1-C_3)M
\end{aligned}
\end{equation}
for every $z\in \tD\setminus B\big(x,\lambda r\big)$. Define 
\begin{equation}\label{e:4.9}
v_{\lambda/(6+2A)} :=\big(h_{\lambda/(6+2A)}-(1-C_3)M\big)^+ . 
\end{equation}
Note that  $v_{\lambda/(6+2A)}=0$ on $\tD\setminus B(x,\lambda r)$ by \eqref{e:4.8}.   Let 
\[
E :=\big(D\cap B(x,2sr)\big)\cup \big(B(x,\lambda r)\setminus B(x,\lambda r/(6+2A))\big).
\]
and define $u_{\lambda/(6+2A)}(z):=\bar{\mathbb{E}}_z [v_{\lambda/(6+2A)}(\bar X_{\tau_E})]\hbox{ for each }z\in\tD$. 
 Then, for every $y\in B(x,2sr)$, 
\begin{align*}
v_{\lambda/(6+2A)}(y)&=\big(h_{\lambda/(6+2A)}(y)-(1-C_3)M\big)^+\\
&=0\vee\bar{\mathbb{E}}_z[h_{\lambda/(6+2A)}(\bar X_{\tau_E})-(1-C_3)M] \\
& \leq \bar{\mathbb{E}}_z[v_{\lambda/(6+2A)}(\bar X_{\tau_E})]=u_{\lambda/(6+2A)}(z).
\end{align*}
Hence,  
\begin{equation}\label{e:4.10}
f_\lambda(y)\geq u_{\lambda/(6+2A)}(y)\geq v_{\lambda/(6+2A)}(y)\hbox{ for every }y\in B(x,2sr),
\end{equation}
where the first inequality is due to a comparison of boundary value along $\partial \big(B(x,2sr)\cap D\big)$.
Let $y_0\in \partial B(x,\lambda r/(3+A))$ such that $ h_{\lambda/(6+2A)}(y_0) \geq (1-C_3/2)M$.
By  \eqref{e:4.9} and \eqref{e:4.10},
\[
u_{\lambda/(6+2A)}(y_0)\geq v_{\lambda/(6+2A)}(y_0)\geq \frac{C_3/2}{1-C_3/2}h_{\lambda/(6+2A)}(y_0).
\]
Hence  by Lemma \ref{lemma44} (a), (b) and the facts that $\lambda\in [1/(6+2A),2]$, $y_0\in \partial B(x,\lambda r/(3+A))$,
 and $h_{\lambda/(6+2A)}$ and $u_{\lambda/(6+2A)}$ are harmonic in $E$,    there is some constant $C_4>0 $ depending on $s$
so that 
\[
u_{\lambda/(6+2A)}(y)\geq C_4h_{\lambda/(6+2A)}(y) \quad \hbox{ for every }y\in \overline{B(x,4r)}\cap D_{r/A}.
\]
 Combining  this with \eqref{e:4.10}, we conclude that 
\begin{equation}\label{e:4.11}
f_\lambda(y)\geq C_4\,h_{\lambda/(6+2A)}(y)\hbox{ for every }y\in \overline{B(x,4r)}\cap D_{r/A}
\end{equation}\smallskip

(i)$\Longrightarrow$(ii): Assume (i) holds, then for $y\in\overline{B(x,4r)}\cap D_{r/A}$ and $r<(C_2\wedge\frac1{6s})\diam(\partial D)$, 
\[
\frac{f_{1/(6+2A)}(y)}{f_2(y)}\geq  C_4\frac{h_{1/(6+2A)^{2}}(y)}{h_2(y)}\geq C_4C_{1,e}^{-1}\frac{\overline  {\rm Cap}\big(B(x,r/(6+2A)^2)\cap\partial D,B(x,2sr)\big)}{\overline  {\rm Cap}\big(B(x,2r)\cap\partial D,B(x,2sr)\big)}\geq C_5
\]
where the first inequality is due to \eqref{e:4.11}  and $f_2\leq h_2$, the second inequality is due to \eqref{e:4.7}, and the last inequality is due to (i) and Lemma \ref{lemma46}. Property (ii)  follows by using Lemma \ref{lemma44} (c).\medskip
 
(ii)$\Longrightarrow$(i): Assume (ii) holds, then for $y\in\overline{B(x,4r)}\cap  D_{r/A}$ and $r<C_2\diam(\partial D)\wedge\frac{\diam(\partial D)}{6s}$, 
\[
\frac{\overline  {\rm Cap}\big(B(x,r/(6+2A))\cap\partial D,B(x,2sr)\big)}{\overline  {\rm Cap}\big(B(x,2r/(6+2A))\cap\partial D,B(x,2sr)\big)}\geq C_{1,e}^{-1}\frac{h_{1/(6+2A)}(y)}{h_{2/(6+2A)}(y)}\geq C_{1,e}^{-1}C_4\frac{f_{1/(6+2A)}(y)}{f_2(y)}\geq C_6,
\]
where the first inequality is tue to \eqref{e:4.7}, the second inequality is due to \eqref{e:4.11} and $h_{1/(6+2A)}\geq f_{1/(6+2A)}$, and the last inequality is due to (ii). (i)  follows by using Lemma \ref{lemma46}.
\end{proof}

\begin{proof}[Proof of Theorem \ref{thm41}]
Fix $x\in \partial D$, $s>\max\{(2A+1)C_{2,e},8A+4\}$ and  $r<\frac{\diam(D)}{6s}$, and define a sequence of hitting times as follows: let 
\[
S_1=\sigma_{\overline{B(x,4r)}},\quad T_1=\tau_{B(x,2sr)}\circ\theta_{S_1}+S_1;
\]
and for $i\geq 2$,   define 
\[
S_i=\sigma_{\overline{B(x,4r)}}\circ T_{i-1},\quad T_i=\tau_{B(x,2sr)}\circ\theta_{S_i}+S_i.
\]
Then by the strong Markov property, for $\lambda\in\{\frac{1}{2A+6},2\}$ and $x_0\in D\setminus \overline{B(x,4r)}$, we have
\begin{align*}
\omega_{x_0}\big(\partial D\cap B(x,\lambda r)\big)&=\bar \bP_{x_0}(\tau_D=\sigma_{\partial D\cap B(x,\lambda r)})\\
&=\sum_{i=1}^\infty\bar \bP_{x_0}(S_i<\tau_D=\sigma_{\partial D\cap B(x,\lambda r)}<S_{i+1})\\
&=\sum_{i=1}^\infty\bar \bP_{x_0}(S_i<\tau_D=\sigma_{\partial D\cap B(x,\lambda r)}<T_i)\\
&=\sum_{i=1}^\infty \bar \bE_{x_0}
\left[S_i<\tau_D;\bar \bP_{\bar X_{S_i}}(\sigma_{\partial D\cap B(x,\lambda r)}=\tau_{D\cap B(x,2sr)}) \right].
\end{align*} 
Hence, we see that 
\begin{equation}\label{e:4.16}
\begin{split} 
\inf_{y\in \partial B(x,4r)}&\frac{\bar \bP_y(\sigma_{\partial D\cap B(x, r/(6+2A))}=\tau_{D\cap B(x,2sr)})}{\bar \bP_y(\sigma_{\partial D\cap B(x,2r)}=\tau_{D\cap B(x,2sr)})}
\leq 
\frac{\omega_{x_0}\big(\partial D\cap B(x,r/(6+2A))\big)}{\omega_{x_0}\big(\partial D\cap B(x,2r)\big)}\\
&\leq \sup_{y\in\partial B(x,4r)}\frac{\bar \bP_y(\sigma_{\partial D\cap B(x, r/(6+2A))}=\tau_{D\cap B(x,2sr)})}{\bar \bP_y(\sigma_{\partial D\cap B(x,2r)}=\tau_{D\cap B(x,2sr)})}.
\end{split} 
\end{equation}

For any non-trivial $f_1,f_2\in\bar\mcF\cap C_c(\tD)$ with  $0\leq f_1,f_2\leq \1_{B(x,2r)}$, define  
\[
h_{f_i}(y):=\bar\bE_y[f_i(X_{\tau_{D\cap B(x,2sr)}});\tau_D=\tau_{D\cap B(x,2sr)}]\hbox{ for }i=1,2\hbox{ and }y\in\tD. 
\]
Then, $f_1,f_2\in\bar\mcF$, and they are non-negative and harmonic in $B(x,2sr)\cap D$. Moreover, $h_{f_1},h_{f_2}$ satisfy the Dirichlet boundary condition along $\partial D\cap (B(x,2sr)\setminus \overline{B(x,2r)})$ by Remark \ref{remark44}. Hence,  $Ch_{f_1}(y)h_{f_2}(z)\leq h_{f_1}(z)h_{f_2}(y)$ for $y,z\in\partial B(x,4r)$, where $C$ is the constant of Lemma \ref{lemma44}(c). By letting  $f_1\to \1_{B(x, r/(6+2A))}$ and $f_2\to \1_{B(x,2r)}$ pointwise, we get 
\[
C\frac{\bar \bP_y(\sigma_{\partial D\cap B(x, r/(6+2A))}=\tau_{D\cap B(x,2sr)})}{\bar \bP_y(\sigma_{\partial D\cap B(x,2r)}=\tau_{D\cap B(x,2sr)})}\leq \frac{\bar \bP_z(\sigma_{\partial D\cap B(x, r/(6+2A))}=\tau_{D\cap B(x,2sr)})}{\bar \bP_z(\sigma_{\partial D\cap B(x,2r)}=\tau_{D\cap B(x,2sr)})}
\]
for $y,z\in\partial B(x,4r)$.   
Hence, by \eqref{e:4.16},  
\begin{align*}
C\sup_{y\in \partial B(x,4r)}&\frac{\bar \bP_y(\sigma_{\partial D\cap B(x, r/(6+2A))}=\tau_{D\cap B(x,2sr)})}{\bar \bP_y(\sigma_{\partial D\cap B(x,2r)}=\tau_{D\cap B(x,2sr)})}
\leq 
\frac{\omega_{x_0}\big(\partial D\cap B(x,r/(6+2A))\big)}{\omega_{x_0}\big(\partial D\cap B(x,2r)\big)}\\
&\leq C^{-1}\inf_{y\in\partial B(x,4r)}\frac{\bar \bP_y(\sigma_{\partial D\cap B(x, r/(6+2A))}=\tau_{D\cap B(x,2sr)})}{\bar \bP_y(\sigma_{\partial D\cap B(x,2r)}=\tau_{D\cap B(x,2sr)})}.
\end{align*}
  The theorem follows from Proposition \ref{prop47}. 
\end{proof}

\section{Local comparability of harmonic measures}\label{S:5}
In this section, we show that harmonic measures behave locally similarly due to (BHP). However, most results of this section do not use (EHP), and we recall that $\sN$ is a properly exceptional set so that the Hunt process $\bar X =\{\bar X_t, t\geq 0; \bar \bP_x, x\in \tD \setminus \sN\}$ associated with $(\bar\mcE,\bar\mcF)$ is well defined. When   (EHP)  holds for $(\tD, d, m_0, \bar \sE,  \bar \sF)$,  the exceptional set $\sN$ can be taken to be an empty set.

We will use the following notations throughout the rest of the paper.  For each compact $K\subset D$, 
\begin{equation} \label{eqn51} 
e_K(x):=\bar\bP_x(\sigma_K<\tau_D), \quad x\in \tD  setminus \sN .
\end{equation}
In literature, $e_K$ is called the condenser potential of $K$ in $D$. It is the equilibrium potential of $K$ in the extended Dirichlet space $(\sE^0,\sF^0_e)$.

Recall that for $h\in \mcB({\partial } \tD)$, $\sH h$ is its harmonic extension into $D$ defined by
\[
\sH h(x):=\bar\bE_x[h(\bar X_{\tau_D});\tau_D<\infty], \quad  x\in \tD  \setminus \sN.
\] 
For $h\in \check\sF_e\cap C_c(\partial D)$ and $K\subset D$ such that $d(K,\operatorname{supp}[h])>0$ and $K\cup\partial D$ is closed, define 
\begin{equation}\label{eqn52}
\sH^K h (x):=\bar\bE_x[h(\bar X_{\tau_D});\tau_D<\sigma_K], \quad  x\in \tD \setminus \sN.
\end{equation}
Note that $\sH^K h\in \bar \sF_e$. For $h\in\check\sF_e\cap C_c(\partial D)$ and compact $K\subset D$, define 
\begin{equation}\label{eqn53} 
c(h, K) := - \bar \sE (\sH^K h, e_K) / \bar \sE (e_K, e_K). 
\end{equation} 
 Observe that  $c(h, K)$ is the constant $a $  where  $\min\{\bar \sE (\sH^K h+a e_K,\sH^K h +a e_K):\,a\in\R\}$ is achieved. More precisely,  $a\mapsto \bar \sE (\sH^K h+a e_K,\sH^K h +a e_K)$ is decreasing on $(-\infty, c(h, K)]$ and increasing on $[c(h, K), \infty)$. 
	
Since $\sH h \in \bar \sF_e$ is harmonic in $D$, $\bar \sE (\sH h, e_K)=0$.  Thus 
\begin{equation}\label{eqn54}
c(h, K) =  \bar \sE (\sH h- \sH^K h, e_K) / \bar \sE (e_K, e_K). 
\end{equation} 
Note that $\sH h-\sH^K h\in\bar\sF_e$ and
\begin{equation}\label{eqn55}
\sH h (x)- \sH^K h(x)=\bar\bE_x[h(\bar X_{\tau_D});\sigma_K<\tau_D], \quad x\in \tD  \setminus \sN.
\end{equation} 
Since $e_K$ is the equilibrium potential of $K$ in $(\bar\sE, \bar\sF_e)$, $c(h, K)\geq 0$ for any non-negative $h\in \check \sF_e\cap C_c(\partial D)$. On the other hand, for any relatively compact open subset $O  \subset \partial D$,  by the regularity of the Dirichlet form $(\bar \sE, \bar \sF)$, there is $\varphi \in C_c (\partial D)\cap \check \sF_e$ so that $0\leq \varphi \leq 1$ and $\varphi =1$ on $O$. For any $h\in \check \sF_e \cap C_c(\partial D)$ with ${\rm supp}[h]\subset O$, we have by \eqref{eqn55} that 
\[
-\| h\|_\infty ( \sH \varphi - \sH^K \varphi ) \leq \sH h- \sH^K h \leq \| h\|_\infty ( \sH \varphi - \sH^K \varphi ).
\]
It then follows by \eqref{eqn54} that 
$
|c(h, K)| \leq c(\varphi , K) \| h\|_\infty. $
Hence the linear functional $h\mapsto c(h, K)$ over $ \check \sF_e \cap C_c (\partial D)$ uniquely determines a Radon measure $\omega_K$ on $\partial D$, that is, 
\begin{equation}\label{eqn56}
c(h, K) =\int_{\partial D}h(z)\omega_K (dz) \quad \hbox{for every } h\in  \check \sF_e \cap C_c (\partial D).
\end{equation}

\begin{lemma}\label{lemma51}
Let $h\in \check \sF_e \cap C_c(\partial D)$ and let $K\subset D$ be a non-$\sE^0$-polar compact subset. 
\begin{enumerate} [\rm (a)]
\item  $\inf_{x\in K} \sH  h(x)\leq c(h, K) \leq \sup_{x\in K} \sH  h(x)$, where the supremum and infimum are defined in $\bar\sE$-q.e. sense.
	
\item  For any open subset $U$ of $ \tD$ satisfying $ \operatorname{supp}[h]\subset U\subset\tD\setminus K$, 
$$
\bar\sE(e_K,\sH^{D\setminus U}h)=- c(h,K)\bar\sE(e_K,e_K).
$$
\end{enumerate}
\end{lemma}

\begin{proof}
(a). Note that $e_K\in \bar\sF_e$ is the equilibrium potential of $K$ in $D$. Let $a=\sup_{x\in K} \sH  h(x)$, one can check that for any $t>0$
\begin{eqnarray*}
&&  \bar \sE (\sH^K h+({a}+t)  e_K ,\sH^K h+({a}+t)  e_K )-\bar \sE (\sH^K h+{a} e_K ,\sH^K h+{a} e_K)\\
&\geq& 2t\bar \sE (\sH^K h+{a} e_K,e_K)
\, =\,  2t\bar \sE (\sH^K h+{a} e_K-\sH h,e_K)
\,  \geq \,  0,
\end{eqnarray*}
where in the equality we used the fact that $\bar \sE (\sH h,  e_K)=0$ as $\sH h$ is harmonic in $D$ and $e_K \in  \sF^0_e$, while the last inequality is due to the fact that $\sH^K h - \sH h +ae_K \in \sF^0_e$ and is non-negative $\bar \sE$-q.e. on $K$ and $e_K$ is the equilibrium potential of $K$ in $( \sE^0,  \sF^0_e)$.  It follows immediately that $c(h, K)\leq {a} =\sup_{x\in K} \sH  h(x)$.  By a same argument, we have  $c(h, K)\geq \inf_{x\in K}\sH^K h(x)$. 
	
 (b) follows from \eqref{eqn53} and the observation $\bar\sE(e_K,\sH^K h-\sH^{D\setminus U}h)=0$. The latter is due to    the facts that $e_K$ is harmonic in $D\setminus K$ and $(\sH^K h-\sH^{D\setminus U}h)|_{K\cup\partial D}=0$.  
\end{proof}

\begin{lemma}\label{lemma52}
Suppose that $u$ is a non-negative function in $\bar\sF_e$ that is regular $\bar \sE$-harmonic in $\tD\setminus K$ for some closed set $K\subset \tD$. Then $\bar \sE(u,f)\leq 0$  for any  $f\in\bar \sF_e$ such that $f\geq 0$ and $uf=0$ both $\bar \sE$-q.e. on $K$. 
\end{lemma}
\begin{proof}
Note that for  any $\varepsilon>0$,  
$$
\bar \sE (u,u) -2\varepsilon\bar \sE (u,f) +\varepsilon^2\bar \sE (f,f) =\bar \sE (u-\varepsilon f,u-\varepsilon f)
\geq \bar \sE \big((u-\varepsilon f)^+ ,(u-\varepsilon f)^+ \big)\geq \bar \sE (u,u),
$$
where the first inequality is due to the normal contraction property of the Dirichlet form, and the second inequality is due to  the facts
 that $(u-\varepsilon f)^+ =u $ on $K$ and that   the function $u$   minimizes $\bar\sE(g,g)$ among all $g\in\bar\sF_e$ with $g=u$ $\bar\sE$-q.e. on $K$. It follows then $\bar\sE(u,f)\leq 0$ since $\varepsilon >0 $ is arbitrary.
 \end{proof}

 	\begin{lemma}\label{L:5.3}
 		Let $h\in\check\sF_e\cap C_c(\partial D)$ and $U$ be an open subset of $\tD$. Suppose that $\operatorname{supp}[h]\subset U$. Then $\sH^{D\setminus U}h$ is $\bar\sE$-regular harmonic in $\tD\setminus K$ with $K=\partial U\cup (U\cap \partial D)$. 
 	\end{lemma}
	
 	\begin{proof}
 	Let $u:=\sH^{D\setminus U}h$. Note that $u$ is bounded with $u=h$  $\bar\sE$-q.e. on $\partial D$ and $u=0$ $\bar \sE$-q.e. on 
 	$\tD\setminus U=(D\setminus U)\cup(\partial D\setminus U)$, and   that $\partial(\tD\setminus\overline{U})\subset \tD\setminus U$.
 		
 	For $\bar\mcE$-q.e. $x\in  D\cap U$, by the quasi-continuity of $u$ and  the continuity of the sample paths of $\bar X$, 
 		\begin{align*}
 			u(x)&=\bar \bE_x [u(\bar X_{\tau_D});\tau_D<\sigma_{D\setminus U}]=\bar\bE_x [u(\bar X_{\tau_D});\tau_D\leq\tau_U<\infty]\\
 			&=\bar\bE_x [u(\bar X_{\tau_D});\tau_D\leq\tau_U<\infty]+\bar\bE_x[u(\bar X_{\tau_U});\tau_U<\tau_D]\\
 			&=\bar\bE_x[u(\bar{X}_{\tau_{D\cap U}});\tau_{D\cap U}<\infty]=\bar\bE_x[u(\bar{X}_{\sigma_K});\sigma_K<\infty],
 		\end{align*} 
 		where the last equality holds as $\partial(D\cap U)\subset \overline{U}\cap(\partial D\cup\partial U)=K\subset \tD\setminus (D\cap U)$.
 		
 		 For $\bar\mcE$-q.e.  $x\in\tD\setminus\overline{U}$, by the quasi-continuity of $u$ and  the continuity of the sample paths of $\bar X$,  
 		$$  
 		u(x)=0 
 		=\bar\bE_x[u(\bar X_{\sigma_{\partial (\tD\setminus \overline{U})}});\sigma_{\partial( \tD\setminus \overline{U})}<\infty]
 		=\bar\bE_x[u(\bar X_{\sigma_K});  \sigma_K<\infty], 
 		$$ 
 		where the last equality holds as $\partial(\tD\setminus \overline{U})\subset \partial U\subset K\subset \overline{U}$. 
 		
 		Thus we have $u(x)=\bar\bE_x[u(\bar X_{\sigma_K});  \sigma_K<\infty]$ for $\bar \sE$-q.e. $x$ in  $\tD\setminus K=(D\cap U)\cup(\tD\setminus\overline{U})$; that is, $u$ is $\bar \sE$-regular harmonic in $\tD\setminus K$. 
 \end{proof}

\begin{lemma}\label{L:5.4}
Suppose that {\rm(BHP)} holds for $(\tD,  d, m_0, \bar \sE, \bar\sF)$. Let $U$ be an open subset of $\tD$, $h\in C_c(\partial D)\cap\check\sF_e$ and $g_1,g_2\in\bar\sF_e$. Suppose that $h$ is non-negative and $\operatorname{supp}[h]\subset U$.
If $f_1\leq f_2$ $\bar\mcE$-q.e. on $\partial U$, and $f_1=f_2=0$ $\bar\mcE$-q.e. on $U\cap \partial D$, then 
\[ 
\bar\mcE(\sH^{D\setminus U}h,f_1)\geq  \bar\mcE(\sH^{D\setminus U}h,f_2).
\]
\end{lemma}
\begin{proof}
Let $u=\sH^{D\setminus U}h$, $f=f_2-f_1$ and $K=\partial U\cup(U\cap \partial D)$. Then, $u$ is a non-negative and regular $\bar \sE$-harmonic in $\tD\setminus K$ by Lemma \ref{L:5.3}, and $f\geq 0$ $\bar\mcE$-q.e. on $K=\partial U\cup (U\cap\partial D)$. As $ u=0$ $\bar\mcE$-q.e. on $\partial U$, it  holds 
that $uf=0$ $\bar\mcE$-q.e. on $K=\partial U\cup(U\cap\partial D)$. It then follows from Lemma \ref{lemma52} that $\bar\mcE(\sH^{D\setminus U}h,f_2-f_1)=\bar\mcE(u,f)\leq 0$,  which gives the desired result.
\end{proof}

\begin{lemma}\label{lemma5+3}
Suppose that {\rm(BHP)} holds for $(\tD,  d, m_0, \bar \sE, \bar\sF)$ with comparison constants $C_{1,b},C_{2,b}>1$.
Let $r\in(0,\diam(\partial D)/2)$ and $\xi\in\partial D$. Let $h$ be a non-negative function in $\check\sF_e\cap C_c(\partial D)$ supported in $\partial D\cap B(\xi,r/C_{2,b})$. Suppose that  $f_1,f_2\in\bar\sF_e$ are non-negative, regular harmonic in $D\cap B(\xi,r)$ with $f_1|_{\partial D\cap B(\xi,r)}=f_2|_{\partial D\cap B(\xi,r)}=0$. Then, 
\[
f_1(x)\bar\sE(\sH^{D\setminus U}h,f_2)\geq  C_{1,b}f_2(x)\bar\sE(\sH^{D\setminus U}h,f_1)
\]
for $\bar\mcE$-q.e.  $x\in B(\xi,r/C_{2,b})$ and open set $ U\subset\tD$ such that $\operatorname{supp}[h]\subset U\subset\overline{U}\subset B(\xi,r/C_{2,b})$.
\end{lemma}

\begin{proof} 
By (BHP), $f_1(x)f_2\leq C_{1,b}f_2(x)f_1$  $\bar\mcE$-q.e. on  $\partial U\subset B(\xi,r/C_{2,b})$. The lemma follows immediately from Lemma \ref{L:5.4}. 
\end{proof}

\begin{corollary}\label{coro5+4}
Suppose  that {\rm(BHP)} holds in $D$ for $(\tD,  d, m_0, \bar \sE, \bar\sF)$, and   $C_{1,b},C_{2,b}>1$  are  the constants in
{\rm(BHP)}. Let $r\in(0,\diam(\partial D)/2)$ and $\xi\in\partial D$. Let $h$ be a non-negative function in $\check\sF_e\cap C_c(\partial D)$ supported in $\partial D\cap B(\xi,r/C_{2,b})$. Then, for any compact non-$\sE^0$-polar  compact sets
$K_1,K_2\subset D\setminus B(\xi,r)$ and $\bar\mcE$-q.e. $x\in D\cap B(\xi,r/C_{2,b})$, 
\[
\frac{c(h,K_1)}{c(h,K_2)}\leq C_{1,b}\frac{e_{K_1}(x)\bar\sE(e_{K_2},e_{K_2})}{e_{K_2}(x)\bar\sE(e_{K_1},e_{K_1})}.
\]
\end{corollary}
\begin{proof}
Let $U$ be an open subset of $\tD$ so that $\operatorname{supp}[h]\subset U\subset\overline{U}\subset B(\xi,r/C_{2,b})$. Then, by Lemma \ref{lemma5+3}, for $\bar\mcE$-q.e. $x\in B(\xi,r/C_{2,b})$, 
\[
e_{K_2}(x)\bar\mcE(e_{K_1},\sH^{D\setminus U}h)\geq  C_{1,b}e_{K_1}(x)\bar\mcE(e_{K_2},\sH^{D\setminus U}h).
\]
The corollary then follows, noticing that $\bar\sE(e_{K_i},\sH^{D\setminus U}h)=- c(h,K_i)\bar\sE(e_{K_i},e_{K_i})$ for $i=1,2$ by Lemma \ref{lemma51}(b).  
\end{proof}

\begin{theorem}\label{thm5+5}
Assume that {\rm(BHP)} holds in $D$ for $(\tD,  d, m_0, \bar \sE, \bar \sF)$, and    $C_{1,b},C_{2,b}>1 $   
are  the constants in  {\rm(BHP)}. Then, for $\bar\sE$-q.e. $x_1,  x_2\in D$, 
\[
\omega_{x_1}(E_1)\omega_{x_2}( E_2)\leq C_{1,b}^2\omega_{ x_1}( E_2)\omega_{ x_2}(E_1)
\]
for every $\xi\in\partial D$, $r<d(x_1,\xi)\wedge d(x_2,\xi)\wedge(\diam(\partial D)/2)$ and $ E_1, E_2\subset \partial D\cap B(\xi,r/C_{2,b})$.
\end{theorem}

\begin{proof}
Fix $\xi\in\partial D$ and $r<\diam(\partial D)/2$.  Let $h_1,h_2 \in C_c(\partial D)\cap\check\sF_e$ be non-negative functions 
  supported in $\partial D\cap B(\xi,r/C_{2,b})$.
By Corollary \ref{coro5+4}
\begin{equation}\label{e:5.6}
c(h_1,K_1)c(h_2,K_2)\leq C_{1,b}^2c(h_2,K_1)c(h_1,K_2)
\end{equation}
for every compact non-$\sE^0$-polar sets $K_1,K_2\subset D\setminus\overline{B(\xi,r)}$ and non-negative. 
Since $\sH h_1,\sH h_2$ are $\bar\mcE$-quasi continuous, for $\bar\sE$-q.e. $x_1,x_2\in D\setminus\overline{B(\xi,r)}$, we can find non-exceptional compact sets $K_1,K_2\subset D\setminus\overline{B(\xi,r)}$ so that 
\[
x_i\in K_i\ \hbox{ and }\ C^{-1}\sH h_j(x_i)\leq\sH h_j|_{K_i}\leq C\sH h_j(x_i)\quad\hbox{ for }i,j\in\{1,2\}.
\]
As $C>1$ is arbitrary, it follows from Lemma \ref{lemma51} (a) and \eqref{e:5.6}, 
\begin{equation}\label{e:5.7}
\sH h_1(x_1)\sH h_2(x_2)\leq C_{1,b}^2\sH h_2(x_1 )\sH h_1( x_2). 
\end{equation} 
Since $(\bar\sE,\bar\sF)$ is a regular Dirichlet form on $L^2(\tD; m_0)$,  applying the above argument to countably many pairs of such $h_1$ and $h_2$, we have by \eqref{e:5.7} that
\[
\omega_{x_1}(E_1)\omega_{x_2}(E_2)\leq C_{1,b}^2\omega_{ x_1}(E_2)\omega_{ x_2}(E_1)\quad\hbox{ for every } E_1, E_2\subset \partial D\cap B(\xi,r/C_{2,b})
\]
for $\bar\sE$-q.e. $x_1,x_2\in D\setminus\overline{B(\xi,r)}$. Consequently, the above estimate holds for a countable dense collection of points $\xi$ in $\partial D$,   countably many $r<\diam(\partial D)/2$  and for $\bar\sE$-q.e. $x_1,x_2\in D\setminus\overline{B(\xi,r)}$. The conclusion of the theorem now follows. 
\end{proof}

\begin{remark}\label{remark5+6} \rm 
If we in addition assume {\rm(EHP)} holds in Theorem \ref{thm5+5}, then the inequality in Theorem \ref{thm5+5}
can be improved to  hold for every $x_1,x_2 \in D$, as (EHP) implies the local H\"older regularity of harmonic functions.  \qed
\end{remark}

\begin{theorem}\label{thm5+7}
Suppose both  {\rm(BHP)} in $D$ and {\rm(EHP)} hold  for $(\tD,  d, m_0, \bar \sE, \bar \sF)$. There is a Radon measure $\omega$ on $\partial D$ so that, for every $x_0\in D$,
\[
\omega_{x_0}(E)\omega(F)\leq C_{1,b}^2\omega_{x_0}(F)\omega(E)
\]
for every $\xi\in\partial D$, $r<d(x_0,\xi)\wedge(\diam(\partial D)/(8A))$ and $E,F\subset \partial D\cap B(\xi,r/C_{2,b})$. Here $A>1$ is the characteristic parameter of  the uniform domain $(D, d)$.
\end{theorem}

\begin{proof}  
When  $\partial D$ is bounded, by Lemma \ref{lemma21}(a), there is some 
$x\in D_{\diam(\partial D)/8A}$.  We  take $\omega :=\omega_x$. Then the conclusion of the theorem follows from Theorem \ref{thm5+5} and (EHP). 

\smallskip

When  $\partial D$ is unbounded, we fix $\xi_0\in\partial D$ and $r_0>0$ such that $\omega_x(E_0)>0$ for some and hence for every $x\in D$ by   (EHP), 
where $E_0=\partial D\cap B(\xi_0,r_0)$. Then, by Theorem \ref{thm5+5},
\[
\frac{\omega_x(E)}{\omega_x(E_0)}\leq 
C_{1,b}^2\frac{\omega_{y}(E)}{\omega_{y}(E_0)}
\] 
for every $r>r_0$, $E\subset B(\xi,r)$ and $x,y\in D\setminus B(\xi,C_{2,b}r)$. As a consequence, we can find a sequence $\{x_n\}_{n\geq 1}\subset D$ such that $d(x_n,\xi_0)\to\infty$ such that $\frac{\omega_x}{\omega_x(E_0)}$ converges   vaguely, 
and it suffices to take the measure $\omega$ to be the vague limit. 
\end{proof}

\begin{remark}  \rm A very similar construction of renormalized harmonic measure from $\infty$ was previously done by Kenig and Toro in \cite[Corollary 3.2]{KT}.
\end{remark}

\section{(LS) condition and capacity density condition}\label{S:6}

We call the measure $\omega$ of Theorem \ref{thm5+7} the renormalized harmonic measure. The goal of this section is to prove that $\omega$ 
is a doubling measure with full support on $\partial D$ and $\Theta_{\Psi, \omega}$ satisfies (LS) if and only if the capacity density condition in Theorem \ref{T:6.1}(iii) holds.  

Suppose that there is an ambient  complete  metric measure strongly local Dirichlet space $(\sX, \tilde d,\tilde m, \tilde \sE, \tilde \sF)$ 
that satisfies {\rm HK$(\Psi)$}, and  $D$ is an $A$-uniform domain in $(\sX, d)$.
 Let  $ (\sE^0,\sF^0):= (\tilde{\sE},\tilde{\sF}^D)$ be the part Dirichlet form 
  of $(\sX, \tilde d,\tilde m, \tilde \sE, \tilde \sF)$  in $D$, where $\tilde{\sF}^D=\{f\in\tilde\sF:\, f =0   \  
  \tilde\sE\hbox{-q.e. on } \sX\setminus D\}$.
It is shown  in \cite{GSC, Mathav} that 
 $(\tD,d,m_0,\bar\sE,\bar\sF)=(\bar{D},d,m|_{\bar{D}},\bar\sE,\bar\sF)$ is a regular strongly local MMD space
 and has  HK($\Psi$) and (VD) property.
In such a setting, we define the following relative capacity with respect to $(\tilde{\sE},\tilde{\sF})$: 
\[
\widetilde{\operatorname{Cap}}(O_1,O_2):=\inf\{\tilde{\sE}(f,f):\,f\in  \tilde{\sF}  \hbox{ with }f=1\hbox{ on }O_1\hbox{ and }\operatorname{supp}[f]\subset O_2\}
\]
where $O_1,O_2\subset \sX$ are open subsets with $\bar{O}_1\subset O_2$. We can consider the capacity density condition of $\sX\setminus D$ with respect to $(\tilde\sE,\tilde\sF)$, which is assumed in \cite{KM}.  
The following is the main result of this section. 

\begin{theorem}\label{T:6.1}
Suppose that $(\tD,d,m_0,\bar\sE,\bar\sF)$ satisfies {\rm HK$(\Psi)$}. The following conditions are equivalent.
\begin{enumerate} [\rm (i)]
\item There is a doubling Radon measure $\sigma$   having full support on $\partial D$ so that   {\rm (LS)} holds for $\Theta_{\Psi, \sigma}$.

\smallskip
		
\item  The  renormalized harmonic measure  $\omega$   has full support on $\partial D$ and is {\rm (VD)},  and     {\rm (LS)}
holds for $\Theta_{\Psi, \omega}.$

\smallskip
	
\item There is $C\in (0,\infty)$ so that 
\begin{equation} \label{e:6.1} 
\overline  {\rm Cap}\big(B(x,r)\cap\partial D,B(x,2r)\big)\geq C\,\frac{V(x,r)}{\Psi(r)}
	\end{equation} 
	for each $0<r<\diam(\partial D)/3$ and $x\in \partial D$.
\end{enumerate}
Moreover, when there is a  complete metric measure strongly local regular Dirichlet space $(\sX,\tilde d, \tilde m,\tilde\sE,\tilde\sF)$ 
that satisfies \rm (VD)   and {\rm HK$(\Psi)$}  so that $D$ is an $A$-uniform domain in $(\sX, d)$, $d=\tilde d|_{D\times D}$, $m=\tilde m |_D$ and $(\sE^0,\sF^0):= (\tilde{\sE},\tilde{\sF}^D)$,  any of the above condition is equivalent to the following condition.
\begin{enumerate} 
\item[\rm (iv)] There are positive constants  $C_1>0 $ and $C_2> (1+A)C_{2,e} $ so that 
\begin{equation} \label{e:6.2}
\widetilde{\operatorname{Cap}}\big(B(x,r)\setminus D,B(x,C_2r)\big)\geq C_1\,\frac{\tilde m(B(x,r))}{\Psi(r)}
\end{equation} 
  for each $0<r<\diam(\partial D)/(3C_2)$ and $x\in \partial D$.  Here, with an abuse of the notation,
   $B(x, r)$ denotes the ball in $(\sX, \tilde d)$ centered at $x$ with radius $r$.
 \end{enumerate} 
\end{theorem}

 We call $(\sX,\tilde d,\tilde m,\tilde\sE,\tilde\sF)$ described in the paragraph above (iv)
an ambient strongly local regular Dirichlet space for $(D, d, m, \sE^0, \sF^0)$ or, simply, for $(\sE^0, \sF^0)$. 
The proof of Theorem \ref{T:6.1} will be given in \S \ref{S:6.2}.

\begin{remark} \label{re:6.2} \rm
\begin{enumerate} [(i)]
\item If {\rm HK}($\Psi$) and (VD) condition  hold for $(\sX, \tilde d,\tilde m, \tilde \sE, \tilde \sF)$, then so does ${\rm Cap}(\Psi)$. Thus \eqref{e:6.2} is equivalent to the following capacity density condition:

\smallskip

\begin{itemize}
\item[$\diamond$] There are positive constants  $C_1>0 $ and $C_2> (1+A)C_{2,e} $ so that 
\begin{equation} \label{e:5.3}
\widetilde {\operatorname{Cap}} \big(B(x,r)\setminus D,B(x,C_2r)\big)\geq C_1 \, \widetilde {\operatorname{Cap}}\big(B(x,r),B(x,C_2r)\big)
\end{equation} 
for each $0<r<\diam(\partial D)/(3C_2)$ and $x\in \partial D$.
\end{itemize}

\smallskip

The same remark applies to the condition \eqref{e:6.1} under the   
{\rm HK}($\Psi$)  assumption  for  the reflected Dirichlet form  $(\tD , d, m_0, \bar  \sE, \bar  \sF)$. 

\smallskip

\item  Clearly condition \eqref{e:5.3}    is weaker than 
the following condition where $\diam( \partial D)$ is replaced by $\diam(  D)$:
there are positive constants  $C_1>0 $, $C_2> (1+A)C_{2,e} $ and $C_3>1$ so that 
\begin{equation} \label{e:6.4a}
\widetilde {\operatorname{Cap}}  \big(B(x,r)\setminus D,B(x,C_2r)\big)\geq C_1 \,\widetilde {\operatorname{Cap}}  \big(B(x,r),B(x,C_2r)\big)
\end{equation} 
for each $0<r<\diam( D)/ C_3$ and $x\in \partial D$.
Condition \eqref{e:6.4a} is the CDC condition assumed in  \cite{KM}. 
We show in Proposition \ref{P:6.3}  below that condition \eqref{e:6.4a} fails when $D^c$ is bounded but $(\tilde \sE, \tilde \sF)$ is transient. 

\smallskip

\item  In literature,  when the condition \eqref{e:6.4a} holds for all $r>0$,  $D^c$ is said to be   uniformly 2-fat in $ \sX$. 
It is shown in \cite[Theorem 1.1]{KS} that the uniformly 2-fatness of $D^c$ in $\sX$ is equivalent to $D$ satisfying  the 2-Hardy's inequality 
in $(\sX, \tilde d,\tilde m, \tilde \sE, \tilde \sF)$. 
\qed
   \end{enumerate}
\end{remark}

\begin{proposition}  \label{P:6.3} 
Suppose that {\rm (VD)} and {\rm HK}($\Psi$)   hold  for $(\sX, \tilde d,\tilde m , \tilde \sE, \tilde \sF)$ and $D$ is a uniform domain in $(\sX, d)$.
Then condition  \eqref{e:6.4a} fails  if $(\tilde \sE, \tilde  \sF)$ is transient  and $D^c$ is bounded. 
 \end{proposition}

\begin{proof} Let $\tilde X$ be the conservative Hunt process associated with the regular Dirichlet form $(\sX, \tilde d,\tilde m, \tilde \sE, \tilde \sF)$.
Since $D^c$ is bounded and $(\tilde \sE, \tilde  \sF)$ is transient, the uniform domain $D$ is necessarily unbounded.
 By \cite[Corollary 3.4.3]{CF}, 
 $e  (y):=\tilde  \bP_y(\tau_D <\infty)$ is in $\tilde \sF_e$ and is the $0$-order equilibrium potential of $D^c$.
 Under {\rm HK}($\Psi$), EHP holds  for   $(\sX, \tilde d,\tilde m, \tilde \sE, \tilde \sF)$ and
 every $\tilde \sE$-harmonic function is locally H\"older continuous. 
 Observe that the function $e$ is regular harmonic in $D$.
 By \cite[Theorem 3.5.2 and Corollary 3.5.3]{CF}, for $\tilde \sE$-q.e. $y\in D$,
$$
\tilde\bP_y  \left(\lim_{t\to \infty} d_D (\tilde X_t)=\infty \hbox{ and }  \lim_{t\to \infty} e(\tilde X_t)=0 \right) =1
\quad \hbox{for $\tilde \sE$-q.e. }  y\in \sX.
$$
In particular, there is a sequence $\{y_n; n\geq 1\}\subset D$ so that $\lim_{n\to \infty} e (y_n) =0$.
Since $D$ is a uniform domain, we have by EHP and a Harnack chain  argument,  
\begin{equation}\label{e:6.5a}
\lim_{y\in D \atop  d_D(y)\to \infty}  \tilde\bP_y(\tau_D <\infty) =0.
\end{equation} 
Let $x\in \partial D$. For each $  r>0$, define for $y\in B(x,C_2r)$,
$$
 e_{D, r}  (y):=\tilde\bP_y( \sigma_{  B(x, r)\cap D^c}<\tau_{B(x,C_2r)}  ) 
 \quad \hbox{and} \quad e_{r}  (y):=\tilde \bP_y( \sigma_{B(x, r)}<\tau_{B(x,C_2r)}  ) .
 $$
By \cite[Corollary 3.4.3]{CF}, $ e_{D, r}$ and $e_r$ are the 0-order equilibrium potential of 
$B(x, r)\cap D^c$ and $B(x, r)$ for the part Dirichlet form $(\tilde \sE, \tilde \sF^{B(x, C_2 r)})$.
Denote by $\mu_{D, r}$ and $\mu_r$ the corresponding equilibrium measures.
Since $(\tilde \sE,  \tilde \sF)$ is strongly local, 
it is known that they are  concentrated on $\partial ( B(x, r)\cap D^c )$ and $\partial B(x, r)$.
Note that $e_r\geq e_{D, r} $ and $e_r - e_{D, r} =0$ on $B(x, r)\cap D^c$. 
By \cite[Theorem 2.2.5]{FOT}, 
\begin{eqnarray*}
  \widetilde {\operatorname{Cap}}  \big( B(x, r) \setminus D,B(x,C_2r)\big) 
&=& \tilde \sE (e_{D, r}, e_{D, r}) = \tilde \sE (e_{D, r}, e_{r}) 
= \int_{\partial B(x,r)} e_{D, r}(y) \mu_r (dy) \\
&\leq & \sup_{y\in  \partial B(x,r)} e_{D, r}(y) \, \mu_r (\partial B(x, r)) \\
&=&  \sup_{y\in  \partial B(x, r)} e_{D, r}(y) \,  \widetilde {\operatorname{Cap}}  \big(B(x, r),B(x,C_2r)\big). 
\end{eqnarray*}
This together with \eqref{e:6.5a}  implies that for each $x\in \partial D$, 
$$
\lim_{r\to \infty} \frac{  \widetilde {\operatorname{Cap}}  \big( B(x, r) \setminus D,B(x,C_2r)\big) }
{ \widetilde {\operatorname{Cap}}  \big(B(x, r),B(x,C_2r)\big)}
\leq\lim_{r\to \infty}  \sup_{y\in \partial B(x, r)} e_{D, r}(y)=0.
$$
This in particular proves that  \eqref{e:6.4a}   cannot hold. 
\end{proof} 
  
  \begin{remark}\label{R:6.4} \rm 
  \begin{enumerate} [(i)]
   \item  The same argument shows that, under the condition that  {\rm HK}($\Psi$)    holds for $(\tD, d, m, \bar \sE, \bar \sF)$, 
  $D$ is a uniform domain in $(\tD, d)$ with unbounded complement, and  
$(\bar \sE, \bar  \sF)$ is transient, then for every $x\in \partial D$,
the inequality \eqref{e:6.1} can not hold for all $r>0$.

  \item Observe that under the condition that (VD) and {\rm HK}($\Psi$) hold for $(\sX, \tilde d,\tilde m, \tilde \sE, \tilde \sF)$, 
   and that $D$ is a uniform domain in $(\sX, {\tilde d})$   with bounded complement, 
    the assumption that $(\tilde \sE,  \tilde \sF)$ is transient 
  is equivalent to the reflected Dirichlet form $(\bar \sE, \bar\sF)$ on $\overline D$ being  transient.   \qed
  \end{enumerate} 
  \end{remark}

\subsection{Hitting probability and relative capacity}\label{S:5.1}

In this subsection, we prove some lemmas. We let $(\tilde{X}_t,t\geq 0;\tilde{\bP}_x,x\in\sX)$ be the Hunt process associated with $(\tilde{\sE},\tilde{\sF})$ on $L^2(\sX;m)$, and we let $\tilde{B}(x,r)=\{y\in \sX:d(x,y)<r\}$ be the ball in $(\sX,d)$. Also, recall that we let $(\bar{X}_t,t\geq 0;\bar{\bP}_x,x\in\tD)$ be the Hunt process associated with $(\bar{\sE},\bar{\sF})$.

\begin{proposition}\label{prop63}
Suppose that either {\rm HK$(\Psi)$} and property \eqref{e:6.1} hold for $(\tD,d,m_0,\bar\sE,\bar\sF)$, or  that {\rm (VD)}, 
{\rm HK$(\Psi)$} and  property \eqref{e:6.2} hold   for $(\sX,d,m,\tilde\sE,\tilde\sF)$.
 Then, there is $C\in(0,1)$ such that
\begin{equation}\label{e:6.4}
\bar\bP_y(\tau_D<\tau_{B(x,r)})\geq C
\end{equation}
for each $x\in \partial D$, $r<\diam(\partial D)/3$ and  $y\in \overline{B\big(x,r/(A+1)\big)}$.  Here $A>1$ is the characteristic parameter of  the uniform domain $(D, d)$.
\end{proposition}

\begin{proof}
Assume that {\rm HK$(\Psi)$} holds for $(\tD,d,m_0,\bar\sE,\bar\sF)$ and property \eqref{e:6.1} holds. We fix $x,y,r$ as in the statement of the proposition. Let $C_{1,e},C_{2,e}$ be the constants of (EHP), and we let $U=B(x,\frac{r}{(A+1)C_{2,e}})$ and $\Gamma=U\cap\partial D$. Then, by the assumption of the proposition and by Lemma \ref{lemma46}, we see that
\[
\overline{\rm Cap}\big(\Gamma,B(x,r)\big)\geq 
C_3\overline{\rm Cap} \Big(\Gamma,   B \big( x,\frac{2r}{(A+1)C_{2,e}} \big) \Big)\geq 
C_4\frac{V(x, r)}{\Psi(r)}
\]
for some $C_3,C_4>0$, where we also use (VD) property of $m_0$ and \eqref{eqnpsi}. It is also known that $\overline  {\rm Cap}\big(U,B(x,r)\big)\leq C_5\frac{V(x,r)}{\Psi(r)}$ by Cap$(\Psi)$. Hence,
\[
\overline  {\rm Cap}\big(\Gamma,B(x,r)\big)\geq \frac{C_4}{C_5}\overline  {\rm Cap}\big(U,B(x,r)\big).
\]
Then, by a same argument as the proof of \eqref{e:4.7}, we know that 
\begin{equation}\label{e:6.5}
h_{\Gamma}(y)\geq \frac{C_4}{C_{1,e}C_5}h_U(y)\hbox{ for every }y\in B(x,r)\setminus B(x,r/(1+A))
\end{equation}
where
\[
h_{\Gamma}(y)=\bar \bP_y(\sigma_{\Gamma}\leq\tau_{B(x,r)})\quad\hbox{ and }\quad  h_U(y)=\bar \bP_y(\sigma_U\leq\tau_{B(x,r)})\quad\hbox{ for every }y\in\tD.
\]
For each $y\in  D\cap\partial B\big(x,r/(1+A)\big)$ and $x'\in D\cap B(x,r/(2A+2A^2))$, as $D$ is $A$-uniform domain,  
there is a path $\gamma \subset D$ connecting $x'$ and $y$
 with $\diam(\gamma)\leq (Ar+r/2)/(1+A)$ and  
 $d(\gamma,\partial B(x,r))\geq  r/(2+2A)$;  see \eqref{e:4.4a}.  
 By a similar argument as that for  \eqref{e:4.5a} using Lemma \ref{lemma45} and the continuity of $h_U$ in $B(x,r)\setminus \overline{U}$,  we have 
\begin{equation}\label{e:6.6}
h_U(y)\geq C_6\quad \hbox{ for every }y\in \partial B(x, r/(1+A)).
\end{equation}
Combining \eqref{e:6.5} and \eqref{e:6.6}, we see that 
$\bar \bP_y(\tau_D<\tau_{B(x,r)})\geq h_{\Gamma}(y)\geq \frac{C_4C_6}{C_{1,e}C_5}$ for each $y\in \partial B(x,r/(1+A))$. By the strong Markov property, it also holds that $\bar \bP_y(\tau_D<\tau_{B(x,r)})\geq \frac{C_4C_6}{C_{1,e}C_5}$ for each $y\in B(x,r/(1+A))$.\medskip 
	
Next, we assume that {\rm HK$(\Psi)$} holds for $(\sX,  \tilde d,\tilde m, \tilde\sE,\tilde\sF)$ and property \eqref{e:6.2} holds.  
Denote by $\tilde X=(\tilde X_t, t\geq 0; \tilde \bP_x, x\in \sX)$ the diffusion process associated with $(\sX, \tilde d, \tilde m, \tilde\sE,\tilde\sF)$.
By a same proof as above, we have  
$
\tilde{\mathbb{P}}_y\big(\sigma_{\Gamma'}<\tau_{B(x,r)}\big)\geq C
$
 for some constant $C>0$ by  \eqref{e:5.3}, where $\Gamma'=B(x,\frac{r}{(A+1)C_{2,e}})\setminus D$.
 Since the diffusion process $\tilde X$ has to leave $D$ first before hitting $\Gamma'$, 
we have $\tilde{\mathbb{P}}_y(\tau_D<\tau_{\tilde{B}(x,r)})
\geq \tilde{\mathbb{P}}_y\big(\sigma_{\Gamma'}<\tau_{B(x,r)}\big) \geq C$. Consequently, $\bar{\mathbb{P}}_y(\tau_D<\tau_{B(x,r)})=\tilde{\mathbb{P}}_y(\tau_D<\tau_{\tilde{B}(x,r)})\geq C$  
 as the part processes of $\tilde X$ and $\bar X$ killed upon leaving $D$ have the same distribution. 
\end{proof}

Denote by $\{ \theta_t; t\geq 0\}$ the time-shift operator for the reflected diffusion $\bar X$. 
Under the assumption of  Proposition \ref{prop63},   we have by the Markov property of $\bar X$ that 
for each $x\in \partial D$, 
\begin{eqnarray*}
\bar \bP_x (\sigma_{D^c} = 0 )
&=& \lim_{r\to 0}\bar\bP_x(\sigma_{D^c}<\tau_{B(x,r)}) \\
&\geq & \lim_{r\to 0}  \lim_{s\to 0} \bar\bP_x \left(\sigma_{D^c} \circ \theta_s <\tau_{B(x,r)} \circ \theta_s; \bar X_s \in B(x, r)  \cap D\right) \\
&\geq & \lim_{r\to 0}  \lim_{s\to 0} \int_{B(x, r) \cap D} \bar\bP_y (\sigma_{D^c}<\tau_{B(x,r)}) \bar p(s, x, y) m_0(dy) \\
&\geq &  C>0.
\end{eqnarray*}
So by Blumenthal's zero–one law,  $\bar\bP_x(\sigma_{D^c} =0)=1$ for every $x\in \partial D$. This means that every $x\in \partial D$ is a regular point for $D^c$.

\begin{lemma}\label{lemma64}
Suppose that \eqref{e:6.4} holds. Let $x\in \partial D$, $r\in (0,\diam(\partial D)/3)$,
 and $h$ be  a bounded, non-negative function on $\overline{B(x,r)}$ such that $h|_{\partial D\cap \overline{B(x,r)}}=0$ and 
$h(z)=\mathbb{E}_z[h(\bar X_{\tau_{B(x,r)\cap D}})]$ for $z\in\overline{B(x,r)}$. Then
\[
h(z)\leq (1-C)^{-1}\big( {d(x,z)}/{r}\big)^\gamma\sup_{y\in \partial B(x,r)}h(y)
\quad \hbox{ for }z\in B(x,r).
\]
where $C$ is the constant of \eqref{e:6.4} and $\gamma=-\frac{\log (1-C)}{\log (1+A)}>0$. 
 Here $A>1$ is the characteristic parameter of  the uniform domain $(D, d)$.
\end{lemma}

\begin{proof}
 The idea of the proof is due to \cite[Definition 2 and Lemma 3]{Ancona}.  By Proposition \ref{prop63}, we have 
\begin{eqnarray*}
h(z)\leq\bar\bP_z(\sigma_{\partial D}>\tau_{B(x,r)})\sup_{y\in \partial B(x,r)}h(y)\leq (1-C)\sup_{y\in \partial B(x,r)}h(y)=(1-C)\sup_{y\in \overline{B(x,r)}}h(y)
\end{eqnarray*}
for each $z\in \overline{B(x,r/(1+A))}$. We can iterate the observation to see that
\[
\sup_{y\in\overline{B(x,(1+A)^{-k}r)}}h(y)\leq (1-C)\sup_{y\in \overline{B(x,(1+A)^{-k+1}r)}}h(y)\leq\cdots\leq (1-C)^k\sup_{y\in \partial B(x,r)}h(y).
\]
Finally, for $z\in\overline{B(x,r)}$, we choose $k\geq 0$ such that $(1+A)^{-k-1}r<d(x,z)\leq (1+A)^{-k}r$, then $h(z)\leq (1-C)^k\sup\limits_{y\in \partial B(x,r)}h(y)<(1-C)^{-1}\big(\frac{d(x,y)}{r}\big)^\gamma\sup\limits_{y\in \partial B(x,r)}h(z)$, where the second inequality is because $\big(\frac{d(x,z)}{r}\big)^\gamma>(1+A)^{-(k+1)\gamma}\geq (1-C)^{k+1}$.
\end{proof}

\begin{corollary}\label{coro65}
Suppose  that  either {\rm HK$(\Psi)$} and property \eqref{e:6.1} hold for $(\tD,d,m_0,\bar\sE,\bar\sF)$, or  that {\rm (VD)}, 
{\rm HK$(\Psi)$} and  property \eqref{e:6.2} hold   for $(\sX,  \tilde d, \tilde m,\tilde\sE,\tilde\sF)$. 
Let $x\in \partial D$, $r\in (0,\diam(\partial D)/3)$ and $h$ is bounded, non-negative on $\overline{B(x,r)}$ such that $h|_{\partial D\cap \overline{B(x,r)}}=0$ and 
$h(z)=\bar\bE_z[h(\bar X_{\tau_{B(x,r)\cap D}})]$ for $z\in\overline{B(x,r)}$. Then
\[
h(z)\leq C\big(\frac{d(x,y)}{r}\big)^\gamma h(y)
\quad \hbox{for }z\in B(x,r/C_{2,b})\hbox{ and }y\in B(x,r/C_{2,b})\cap D_{r/(4AC_{2,b})}, 
\]
where $\gamma>0$ is the constant of \eqref{e:6.4} and $C>0$ depends only on the bounds of {\rm HK$(\Psi)$}. 
 Here $A>1$ is the characteristic parameter of  the uniform domain $(D, d)$.
\end{corollary}

\begin{proof}
Let $h,r,x,y,z$ be the objects in the statement of the corollary. Let $h_*$ be defined as $h_*(z)=\bar\bP_z(\tau_D>\tau_{B(x,r)})$ for $z\in \tD$. By the same proof of Lemma \ref{lemma46}, we know that $h_*(y)\geq C_1$ for some $C_1\in(0,1)$; by Corollary \ref{coro65}, we know that $h_*(z)\leq C_2(\frac{d(x,z)}{r})^\gamma$. Hence, $
h_*(z)\leq \frac{C_2}{C_1}(\frac{d(x,z)}{r})^\gamma h_*(y)$.  Finally, by using (BHP), we see $h_*(z)\leq C_{2,b}\frac{C_2}{C_1}(\frac{d(x,z)}{r})^\gamma$, where $C_{2,b}$ is the constant of (BHP). 
\end{proof}

 Recall that $e_K$ is the condenser potential of  $K$ in $D$ as defined in \eqref{eqn51}.

\begin{lemma}\label{lemma66}
Assume \eqref{e:6.4} and that $(D,d,m,  \sE^0,\sF^0)$ satisfies {\rm PI$(\Psi;D)$} and {\rm Cap$_\leq(\Psi;D)$}. Let $x\in\partial D$, $0<r<\operatorname{diam}(\partial D)/3$ and $K=\overline{B(x,r)}\cap D_{r/(4A)}$, 
 where $A>1$ is the characteristic parameter of  the uniform domain $(D, d)$.  Then 
\[
C^{-1}\frac{V (x,r) }{\Psi(r)}\leq \sE^0(e_K,e_K)\leq C\frac{V (x,r) }{\Psi(r)}
\]
for some $C\in (0,\infty)$ independent of $x,r$. 
\end{lemma}
\begin{proof} First, we prove the upper bound. There is $C_1<\infty$ independent of $x,r$ such that we can cover $K$ with $N\leq C_1$ balls of the form $B(y_i,r/8A)$, where $y_i\in K$ for $1\leq i\leq N$. For each $1\leq i\leq N$, we can find $\phi_i\in C_c(\tD)\cap  \sF^0$ such that $0\leq \phi_i\leq 1$, $\phi_i|_{\tD\setminus B(y_i,r/4A)}=0$,  $\phi_i|_{B(y_i,r/8A)}=1$ and
\[
\sE^0  (\phi_i,\phi_i)\leq C_2\frac{V (y_i,\lambda r)}{\Psi(\lambda r)}\leq C_3\frac{V (x,r)}{\Psi(r)}
\] 
for some $C_2,C_3>0$ independent of $x,r$, due to $\operatorname{Cap}_{\leq}(\Psi;D)$, (VD) and \eqref{eqnpsi}. Then, we define $\widehat {\phi}_K$ by $\widehat {\phi}_K(z)=\max_{1\leq i\leq N}\phi_i(z)$ for each $z\in \tD$, then one can see that 
$$
\sE^0 (e_K,e_K)\leq d\sE^0(\widehat {\phi}_K,\widehat {\phi}_K)\leq \sum_{i=1}^N  \sE^0 (\phi_i,\phi_i)\leq C_1C_3\frac{V(x,r)}{\Psi(r)}.
$$ 
	
Next, we prove the lower bound. By \eqref{e:6.4},
\begin{equation}\label{e:6.7}
e_K(y)\leq 1-C_4\hbox{ for every }y\in B(x,\frac{r}{4A(1+A)}).
\end{equation}
Let $\delta=\frac{r}{16A^2(1+A)}$, we have
\begin{align*}
\Psi(r) \sE^0 (e_K,e_K)
&\gtrsim\int_{y\in B(x,r)\cap D_{\delta r}}\frac{1}{m_0(B(y,r))}\int_{z\in B(x,r)\cap D_{\delta r}}\big(e_K(y)-e_K(z)\big)^2m_0(dy)m_0(dz)\\
&\geq \frac{1}{V(x, r)}\int_{y\in B(x,4A\delta r)\cap D_{\delta r}}\int_{z\in K}\big(e_K(y)-e_K(z)\big)^2m_0(dy)m_0(dz)\\
&\gtrsim\frac{m_0\big(B(x,4A\delta r)\cap D_{\delta r}\big)m_0(K)}{V(x, r)}\\
&\gtrsim V(x,r),
\end{align*}
where we use Lemma \ref{lemma36} in the first inequality, we use (VD) property of $m_0$ in the second inequality, we use \eqref{e:6.7} in the third inequality, and we use (VD) property of $m_0$ and Lemma \ref{lemma21} (a) in the last inequality. 
\end{proof}

\subsection{Proof of Theorem \ref{T:6.1}} \label{S:6.2}
  Recall the definition of $e_K$ and $\omega_K$ in \eqref{eqn51} and \eqref{eqn56}. Recall also that $C_{1,b},C_{2,b}$ are constants of (BHP).

\begin{proposition}\label{prop67}
Assume \eqref{e:6.4} and that {\rm HK$(\Psi)$} holds for $(\tD,d,m_0,\bar\sE,\bar\sF)$.  Then there is  a constant $C_1\in (1,\infty)$ such that 
\[
C_1^{-1}e_K(y)\frac{V(x,r)}{\Psi(r)}\leq\bar \sE (e_K,e_K)\,\omega_K(B(x,r))\leq C_1e_K(y)\frac{V(x,r)}{\Psi(r)}
\]
for $x\in\partial D$, $r<\diam(\partial D)/(8AC_{2,b})$, compact $K\subset D\setminus \overline{B(x,C_{2,b}r)}$ and $y\in B(x,r)\cap D_{r/(4A)}$. Here $A>1$ is the characteristic parameter of  the uniform domain $(D, d)$.
\end{proposition}

\begin{proof}
Set 
$ K_x :=\overline{B\big(x,4AC_{2,b}r\big)}\cap D_{C_{2,b}r} .$ 
Let $h\in C_c(\partial D)\cap \check\sF_e$ that satisfies $0\leq h\leq 1$, $\operatorname{supp}[h]\subset B(x,r)$ and $h|_{B(x,r/2)}=1$. Then, by Corollary \ref{coro5+4}, we see that 
\begin{align}\label{e:6.8}
\begin{split}
C_{1,b}^{-1}\frac{e_K(y)}{e_{K_x}(y)}c(h, K_x) \bar \sE (e_{K_x},e_{K_x})
&\leq c(h, K) \bar \sE (e_K,e_K)
\\&\leq C_{1,b}\frac{e_K(y)}{e_{K_x}(y)}c(h, K_x) \bar \sE (e_{K_x},e_{K_x}).
\end{split}
\end{align}
Next, by using (EHI) and Proposition \ref{prop63}, we can show that $\sH h(z)\geq C_3$ for each $z\in K_x$, hence, $c(h, K_x)\geq C_3$ by Lemma \ref{lemma51} (a). Moreover, $e_{K_x}(y)\geq C_4$ by Lemma \ref{lemma45}, noticing that we can find $z\in K_x$ by Lemma \ref{lemma21}(a) such that $B(z,C_{2,br}/3A) \subset K_x$ and there is a path connecting $y,z$ such that
 $d(\gamma,\partial D)\geq \frac{r}{4A(1+A)}$. So \eqref{e:6.8} is simplified to be 
\[
C_5^{-1}e_K(y)\bar \sE (e_{K_x},e_{K_x})\leq c(h, K) \bar \sE (e_K,e_K)\leq C_5e_K(y)\bar \sE (e_{K_x},e_{K_x}).
\]
  Then, by Lemma \ref{lemma66} and \eqref{eqn56}, 
\[
C_6^{-1}e_K(y)\frac{V(x, r)}{\Psi(r)}\leq 
\bar\sE(e_K,e_K)\int_{\partial D}h(z)\omega_K(z)\leq C_6e_K(y)\frac{V(x,r)}{\Psi(r)}.
\]
The proposition follows immediately by the regular property of $(\bar\mcE,\bar\mcF)$. 
\end{proof}

  The following estimate of harmonic measure was obtained by Aikawa and Hirata \cite[Lemmas 3.5 and 3.6]{AH} in the Euclidean setting.

\begin{proposition}\label{P:6.8}
Suppose that either {\rm HK$(\Psi)$} and property \eqref{e:6.1} hold for $(\tD,d,m_0,\bar\sE,\bar\sF)$, or  that {\rm (VD)}, 
{\rm HK$(\Psi)$} and  property \eqref{e:6.2} hold   for an ambient strongly local regular Dirichlet space 
$(\sX, \tilde d, \tilde  m,\tilde\sE,\tilde\sF)$  for $(\sE^0,\sF^0)$. 
 Then, there are $C_1,C_2\in (1,\infty)$ such that 
\begin{equation}\label{e:6.10a}
C_1^{-1}\frac{\bar{g}_D(x_0,y)V(x,r)}{\Psi(r)}\leq \omega_{x_0}\big(B(x,r)\big)\leq C_1\frac{\bar{g}_D(x_0,y)V(x,r)}{\Psi(r)}
\end{equation}
for $x\in\partial D$, $r<\diam(\partial D)/(8AC_{2,b})$, $x_0\in D\setminus \overline{B(x,C_{3,b}r)}$ and $y\in B(x,r)\cap D_{r/(4A)}$. 
 Here $A>1$ is the characteristic parameter of  the uniform domain $(D, d)$.
\end{proposition}

\begin{proof}
This result is an easy consequence of Proposition \ref{prop67}.  
Let $K_n=\overline{B(x_0,\frac{1}{n})}$. Then for large enough $n$,  $e_{K_n}(z)=\int_{K_n}.  \bar{g}_D (z,w)\sigma_n(dw)$, where $\sigma_n$ is the equilibrium measure on $K_n$ and $\sigma_n(K_n)=\sE(e_{K_n},e_{K_n})$. By (EHP), we know that $\bar{g}_D(\cdot,y)$ is continuous at $x_0$, so 
\begin{equation}\label{e:6.9}
\lim\limits_{n\to\infty}\frac{e_{K_n}(y)}{\sE(e_{K_n},e_{K_n})}=\bar{g}_{D}(x_0,y).
\end{equation}
By Proposition \ref{prop67},  
\begin{equation}\label{e:6.10}
C_1^{-1}\frac{e_{K_n}(y)}{\sE(e_{K_n},e_{K_n})}\frac{V(x,r)}{\Psi(r)}\leq\,\omega_{K_n}\big(B(x,r)\big)\leq C_1\frac{e_{K_n}(y)}{\sE(e_{K_n},e_{K_n})}\frac{V(x,r)}{\Psi(r)}
\end{equation}
The proposition then follows from \eqref{e:6.9}, \eqref{e:6.10}  and Lemma \ref{lemma51}(a).  
\end{proof}

\begin{remark} \label{R:6.9} \rm  
\begin{enumerate}
\item When there is an ambient complete metric measure strongly local Dirichlet space $(\sX,d,m,\tilde\sE,\tilde\sF)$ that satisfies {\rm HK$(\Psi)$}
so that   $D$ is an $A$-uniform domain in $(\sX, d)$  and $ (\sE^0,\sF^0) :=(\tilde{\sE},\tilde{\sF}^D)$, the two-sided   harmonic measure  estimates \eqref{e:6.10a}   has also been proved recently in  \cite[Theorem 4.6]{KM} by Kajino and Murugan under a stronger  condition  \eqref{e:6.4a} than   \eqref{e:6.2}. In a recent updated version, the authors  outlined in \cite[\S 5.4]{KM} how their arguments can be modified to establish the estimates \eqref{e:6.10a} under condition \eqref{e:6.2}.

\item	 Our approach to Proposition \ref{P:6.8} is based on the capacity estimate and Corollary \ref{coro5+4}. 
We remark that the method of Aikawa and Hirata \cite{AH}  can be modified to provide another proof of Proposition \ref{P:6.8}. The idea is to first establish \eqref{e:6.10a} for $x\in \partial B(y,r/(8A))$ by using the Green function estimates and the hitting probability estimates from Proposition \ref{prop63}. Then, one can apply the maximal principle, together with the EHP and BHP, to extend the estimates to $D\setminus\overline{B(x,C_{3,b}r)}$. In the original paper \cite{AH}, a box argument instead of the BHP was applied to prove the upper bound, in the setting of John domains.  While this modified approach of Aikawa and Hirata is more direct, our approach has the advantage that the techniques developed in  Section \ref{S:5} of this paper do not rely on the corkscrew condition of the uniform domain directly, and allow us to establish (LS) property
without using the Green's function. In a forthcoming paper, we apply a similar idea to prove (LS) for reflected jump processes on domains that does not satisfy the corkscrew condition.
\qed
\end{enumerate} 
	\end{remark}

We can now present the proof for  Theorem \ref{T:6.1}.

\begin{proof}[Proof of Theorem \ref{T:6.1}]
(ii)$\Rightarrow$(i) is trivial. \medskip
	
(i)$\Rightarrow$(iii). Let $x\in \partial D$, $r\in (0,\operatorname{diam}(\partial D)/3)$ and $h(y)=\bar \bP_y(\sigma_{B(x,r)\cap \partial D}<\tau_{B(x,2r)})$ for each $y\in \tD$, so that 
\begin{equation}\label{e:6.11}
\bar\sE(h,h)=\overline{\rm Cap}\big(B(x,r)\cap \partial D,B(x,2r)\big).
\end{equation}
  Recall that $h_r(y)=\frac{1}{m_0(B(y,r)\cap D_{r/(4A)})}\int_{B(y,r)\cap D_{r/(4A)}}h(z)m_0(dz)$ as defined in Section \ref{S:3}. By \eqref{eqn36}, 
  there is a  positive constant $C_1,$ depending on $\sigma$,  so that 
\begin{align*}
\sqrt{\int_{y\in \partial D }\frac{\big(h(y)-h_{r}(y)\big)^2}{\Theta_{\Psi,\sigma}(y,r)}\sigma(dy)}
&\leq \sum_{k=0}^\infty\sqrt{\int_{y\in\partial D}\frac{\big(h_{\theta^kr}(y)-h_{\theta^{k+1}r}(y)\big)^2}{\Theta_{\Psi, \sigma} (y,r)}\sigma(dx) }
\\
&\leq C_1\sum_{k=0}^\infty\sqrt{\theta^{\beta k}\bar\mcE(h,h)}=\frac{C_1}{1-\theta^{\beta/2}}\sqrt{\bar\mcE(h,h)}, 
\end{align*} 
where $\beta$ is the parameter of (LS) and $\theta=1/(4A)$. Then, 
by  the doubling property of $m_0$ and $\sigma$, we have for some
 $C_2\in (0,\infty)$ depending on $\sigma$ that 
\begin{equation}\label{e:6.12}
	\bar \sE (h,h)\geq  \Big(\frac{1-\theta^{\beta/2}}{C_1} \Big)^2
	\int_{y\in  B(x,r)}\frac{\big(h(y)-h_{r}(y)\big)^2 V(y,r) }{\Psi(r)V_\sigma(y,r)}\sigma(dy)\geq C_2\frac{V(x,r)}{\Psi(r)}, 
\end{equation}
where  we use the fact that $h(z)<1-C_3$ for some $C_3>0$ if $d_D(z)>r/(4A)$, which follows an argument similar to that of Lemma \ref{lemma46} since we can find a path $\gamma$ connecting $z$ and $\tD\setminus B(x,3r)$ such that $d(\gamma,\partial D\cap B(x,r))>\frac{r}{4A(1+A)}$.
(iii) follows immediately from \eqref{e:6.11} and \eqref{e:6.12}. \medskip 
	
(iii)$\Rightarrow$(ii) and (iv)$\Rightarrow$(ii).  We fix $x\in \partial D$ and $0<r<R<\diam(\partial D)/(8AC_{2,b})$, and we choose $x_0\in D\setminus \overline{B(x,C_{2,b}R)}$. Then, by Proposition \ref{P:6.8},
\begin{equation}\label{e:6.13}
C_4\frac{\bar{g}_D(x_0,y_1)}{\bar{g}_D(x_0,y_2)}\leq\frac{\Theta_{\Psi, \omega} (x,r)}{\Theta_{\Psi, \omega} (x,R)}=\frac{V_\omega (x,r) \frac{\Psi(r)}{V (x,r)}}{V_\omega (x,R) \frac{\Psi(R)}{V(x,R)}}\leq C_5\frac{\bar{g}_D(x_0,y_1)}{\bar{g}_D(x_0,y_2)}, 
\end{equation}
where $y_1\in B(x,r)\cap D_{r/(4A)}$ and $y_2\in B(x,R)\cap D_{R/(4A)}$.

By Corollary \ref{coro65}, we know that for some $\gamma,C_6>0$, 
\begin{equation}\label{e:6.14}
\frac{\bar{g}_D(x_0,y_1)}{\bar{g}_D(x_0,y_2)}\leq C_6 \left(\frac{r}{R} \right)^{\gamma}. 
\end{equation}
Hence, $\Theta_{\Psi,\omega}$ satisfies {\rm (LS)} by \eqref{e:6.13} and \eqref{e:6.14}. 

By Lemma \ref{lemma44} (a), we also know that  
\begin{equation}\label{e:6.15}
\frac{\bar{g}_D(x_0,y_1)}{\bar{g}_D(x_0,y_2)}\geq C_7
\end{equation}
for some $C_7>0$ if $r=R/2$. Then, $\omega$ satisfies {\rm (VD)} and have full support on $\partial D$ by \eqref{e:6.13} and \eqref{e:6.15}. 
\medskip

(iii)$\Rightarrow$(iv) is   immediate  by using Lemma \ref{lemma46}, the fact $\widetilde{\operatorname{Cap}}\big(B(x,r)\setminus D,B(x,C_2r)\big)\geq \overline  {\rm Cap}\big(B(x,r)\cap \partial D,B(x,C_2r)\big)$, where $C_2$ is the constant of (iv), as well as the corkscrew property   in Lemma \ref{lemma21}  of  $D$  being an $A$-uniform domain in $(\sX, \tilde d)$   and the  (VD) property of $\tilde m$.
  \end{proof}

\section{Trace Dirichlet form}\label{S:7}

In this section, we give the characterization of  the trace Dirichlet form of $(\check\sE,\check\sF_e)$ with respect to the measure $\omega$ introduced in Theorem \ref{thm5+7}, under condition that $\omega$ satisfies (VD) and (LS) holds for $\Theta_{\Psi, \omega}$. Note that    
equivalent conditions for these two properties are given in Theorem \ref{T:6.1}.

For two  measures $\mu, \nu$ on a set $E$, we say that $\mu\asymp \nu$ if there is $C\in (1,\infty)$ such that 
$C^{-1}\nu(A)\leq\mu(A)\leq C\nu (A)$ for each $A\subset E$.

\begin{theorem}\label{thm71}
Suppose that $(\tD,d,m_0,\bar\sE,\bar\sF)$ satisfies {\rm HK($\Psi$)}. Let $\omega$ be the renormalized harmonic measure of Theorem \ref{thm5+7}.
Suppose that $\omega$ has full support on $\partial D$ and is {\rm (VD)}, and $\Theta_{\Psi, \omega}$ satisfies {\rm(LS)}.
 Let $(\check \sE, \check  \sF )$ be the trace Dirichlet form of $\big(\bar \sE ,\bar \sF \big)$ on $\partial D$ using   $\omega$ as the Revuz measure for the time-change. Then  there are positive constants $c_1,c_2$ depending on the parameter $A$ in the uniform domain condition, the parameters in {\rm(VD)}, {\rm HK($\Psi$)} for $(\tD,d,m_0,\bar\sE,\bar\sF)$, and the parameters in   {\rm(VD)} for $\omega$ and  {\rm(LS)}  for $\Theta_{\Psi, \omega}$
 such that the following  holds for the Beurling-Deny expression \eqref{e:2.7} 
of $(\check \sE, \check  \sF )$
	\begin{enumerate} [\rm (a)] 
 	
\item $ \displaystyle\frac{c_1}{V_\omega (x,d(x,y))\Theta_{\Psi, \omega} (x,d\big(x,y)\big)}\leq \frac{\check J(dx,dy)}{ \omega (dx)\omega(dy)}
\leq \frac{c_2}{V_\omega (x,d(x,y)) \Theta_{\Psi, \omega} (x,d\big(x,y)\big)}.$
	
	\item If $(\check \sE, \check  \sF )$ is recurrent or $\partial D$ is unbounded, then $\check\kappa=0$ (no killings). 
	
	\item If   $(\check \sE, \check  \sF )$ is transient and $\partial D$ is bounded, then 
	\begin{equation} \label{e:7.1a}
	c_1\frac{\overline  {\rm Cap}_0(\partial D)}{\omega(\partial D)}\leq\frac{\check\kappa(dx)}{\omega(dx)}\leq c_2\frac{\overline  {\rm Cap}_0(\partial D)}{\omega(\partial D)}.
	\end{equation} 
	Here  $\overline  {\rm Cap}_0(\partial D)=\check \sE ( {\mathbbm 1}_{\partial D},{\mathbbm 1}_{\partial D})=\bar \sE (\mathcal{H}{\mathbbm 1}_{\partial D},\mathcal{H}{\mathbbm 1}_{\partial D})$.
 	\end{enumerate} 
\end{theorem}

 \begin{lemma}\label{lemma72}
Suppose  that $(\tD,d,m_0,\bar\sE,\bar\sF)$ satisfies {\rm HK$(\Psi)$}. Recall that $\sH^K h $ is defined by \eqref{eqn52}
and $A>1$ is the characteristic parameter for the uniform domain condition  of $(D, d)$. 
There is $C\in (0,1)$ such that
\[
h(y)-\sH^K h (y)\geq C\,h(y)  \quad \hbox{ for each }y\in D\setminus B(x,2r)
\]
for any $x\in \partial D$, $0<r<\operatorname{diam}(\partial D)/3$, $K=\overline{B(x,r)}\cap D_{r/(4A)}$ and $h\in C(\tD)\cap\bar \sF_e$ that is non-negative, regular harmonic in $D$ and satisfies the Dirichlet boundary condition along $\partial D\setminus \overline{B(x,r)}$.
\end{lemma}

\begin{proof}
By Lemma \ref{lemma21}(a),  there is $z_0\in K$ so that $B(z_0,r/(12A))\subset K$. Let $C_{1,b},C_{2,b}> 1$ be constants in (BHP), 
and let $C_1=4AC_{2,b}+1$.

For each $y\in D_{r/C_1}\cap(B(x,3r)\setminus B(x,r))$, as $(D, d)$ is $A$-uniform, by Lemma \ref{lemma21}(b), 
there is a path $\gamma$ connecting $y,z_0$ in $D$ such that $d(\gamma,\partial D)>\frac{r}{C_1(A+1)}$ and $\diam(\gamma)<4Ar$. So  by Lemma \ref{lemma45}, 
$ \bar\bP_y(\sigma_K<\tau_D)\geq C_2.$
Moreover, by Lemma \ref{lemma44}(a), 
$ h(z)\geq  C_3 h(y)$ for each $z\in K.$
By the above two estimates, for $y\in D_{r/C_1}\cap(B(x,3r)\setminus B(x,r))$,
\begin{align}\label{e:7.1}
\begin{split}
h(y)-\sH^K h (y)&=\bar\bE_y[h(\bar X_{\tau_D});\sigma_K<\tau_D]=\bar\bE_y\big[\bar\bE_{X_{\sigma_K}}[h(\bar X_{\tau_D})];\sigma_K<\tau_D\big]\\
&=\bar\bE_y\big[h(\bar X_{\sigma_K});\sigma_K<\tau_D\big]\geq\bar\bE_y(\sigma_K<\tau_D)\inf_{z\in K}h(z)\geq C_2C_3 h(y). 
\end{split}
\end{align}

Next, for each $y\in D_{0,r/C_1}\cap\partial B(x,2r)$, there is  $\xi\in \partial D$ so  that $d(\xi,y)<r/C_1$. 
Note 
 $B(\xi,4AC_{2,b}r/C_1)\subset B(x,3r)\setminus B(x,r)$ and $B(\xi,4Ar/C_1)\cap D_{r/C_1}\neq\emptyset$ by Lemma \ref{lemma21}(a). Hence, by (BHP) and \eqref{e:7.1}, 
\begin{equation}\label{e:7.2}
h(y)-\sH^K h (y)\geq \frac{C_2C_3}{C_{1,b}}h(y). 
\end{equation}

Combining \eqref{e:7.1} and \eqref{e:7.2}, we see that $h(y)-\sH^K h (y)\geq \frac{C_2C_3}{C_{1,b}}h(y)$ for each $y\in \partial B(x,2r)$. 
The above inequality holds for every $y\in D\setminus B(x,2r)$ as both $h-\sH^K h$ and $h$ are regular harmonic in $ D\setminus K$ and satisfies Dirichlet boundary condition along $\partial D\setminus \overline{B(x,r)}$. 
\end{proof}

\begin{proof}[Proof of Theorem \ref{thm71}] (a).
Let $C_{1,b}$ and $C_{2,b}$ be constants of (BHP). Let $x,y\in \partial D$ and let $f,g\in C_c(\partial D)\cap \check \sF_e$ such that $f,g$ are non-negative, $f(z)=0$ for each $z\notin B(x,r)$, $g(z)=0$ for each $z\notin B(y,r)$, where $r=d(x,y)/3$.   
Then
\begin{equation}\label{e:7.3}
\begin{aligned}
2\int_{\partial D}f(z)g(w)\check J(dz,dw) & =-\check \sE (f,g) =-\bar \sE (\sH f,\sH g)=-\bar \sE (\sH f,\sH^{D\setminus B(y,r)}g),
\end{aligned}
\end{equation} 
where the third quality is due to the fact that $\sH f$ is harmonic in $D$. Let $K=\overline{B(x,r)}\cap D_{r/(4A)}$. Note that by Lemma \ref{lemma72}, $C_1\sH f(z)\leq\sH f(z)-\sH^K f(z)\leq \sH f(z)$ for every $z\in \partial B(y,r)$. By Lemma \ref{L:5.4}, 
\begin{equation}\label{e:7.4}
 \bar \sE (\sH f-\sH^Kf,\sH^{D\setminus B(y,r)}g)\asymp  \bar \sE (\sH f,\sH^{D\setminus B(y,r)}g).
\end{equation}
Moreover, by Lemmas \ref{lemma51}(a) and \ref{lemma44}(a), 
 $\sH f(z)-\sH^K f(z)\asymp c(f,K)$ for $z\in K$. Hence $\sH f-\sH^K{ f}\asymp  c(f,K) e_K$. By Lemmas \ref{lemma51}(a) and \ref{lemma52}, 
\begin{equation}\label{e:7.5}
 \bar \sE (\sH f-\sH^Kf,\sH^{D\setminus B(y,r)}g)
\asymp  c(f,K)\bar\sE(e_K,\sH^{D\setminus B(y,r)}g) 
= - c(f,K)c(g,K)\bar\sE(e_K,e_K). 
\end{equation}
Combining \eqref{e:7.3}--\eqref{e:7.5} and Lemma \ref{lemma66}, we see  
\begin{equation}\label{e:7.6}
\int_{\partial D}f(z)g(w)\check J(dz,dw)\asymp c(f,K)c(g,K)\frac{V(x, r)}{\Psi(r)}. 
\end{equation}
  Let $v\in K$.  Then $\omega_v(B(x,r))\geq C_2$ and $\omega_v(B(y,r))\geq C_2$ for some $C_2>0$ by Proposition \ref{prop63}. Then, by  Lemma \ref{lemma44}, Lemma \ref{lemma51}(a) and Theorem \ref{thm5+7},  
\begin{align}
\label{e:7.7}&c(f, K)\asymp  \sH f(v)=\int_{\partial D}f(z)\omega_v(dz)\asymp \int_{\partial D}f(z)\frac{\omega_v(dz)}{\omega_v(B(x,r))}
\asymp \int_{\partial D}f(z)\frac{\omega(dz)}{V_\omega(x,r)},\\
\label{e:7.8}&c(g, K)\asymp  \sH g(v)=\int_{\partial D}g(w)\omega_v(dw)\asymp \int_{\partial D}g(w)\frac{\omega_v(dw)}{\omega_v(B(x,r))}
\asymp \int_{\partial D}g(w)\frac{\omega(dw)}{V_\omega(y,r)}.
\end{align}
By Lemma \ref{lemma66} and (VD). By combining \eqref{e:7.6}--\eqref{e:7.8}, we see
\[
\begin{aligned}
&\quad\ \int_{\partial D\times \partial D\setminus \operatorname{diag}}f(z)g(w)\check J(dz,dw)
\\&\asymp \int_{\partial D\times \partial D}\frac{V (x,r)} {\Psi(r) V_\omega  (x,r) V_\omega (y,r)}f(z)g(w)\omega(dz)\omega(dw)\\
&\asymp \int_{\partial D\times \partial D}\frac{V (z,r) }{\Psi(r) V_\omega (z,r)^2} f(z)g(w)\omega(dz)\omega(dw)\\
&\asymp  \int_{\partial D\times \partial D}\frac{1}{\Theta_{\Psi, \omega} (z,d(z,w))V_\omega (z,d(z,w))} f(z)g(w)\omega(dz)\omega(dw).
\end{aligned}
\]
This finishes the proof since the estimate works any $x\neq y$ and $f,g$ support on small neighborhoods of $x,y$ respectively. 

\smallskip

(b). Suppose that $(\check \sE,  \check \sF)$ is recurrent. Then clearly it has no killing measure. Suppose that $(\check \sE,  \check \sF)$ is transient and $\partial D$ is unbounded. As by \cite[Theorem 5.2.5]{CF}, the transience and recurrence property is invariant under time changes,   $(\bar \sE, \bar \sF)$ is transient. Fix $\xi\in \partial D$ and let $v_n\in\bar\sF_e$ be defined as
\[
v_n(x)=\bar\bP_x\big(\sigma_{B(\xi,n)\cap\partial D}<\infty\big) \quad \hbox{ for each }x\in \tD.
\]
Then   $0\leq v_n\leq 1$ and $ \lim_{n\to \infty} v_n(x)= 1$   for each $x\in \tD$ by Lemma \ref{lemma64}. 
 Next, we fix non-negative $f\in C_c(\partial D)\cap \check\sF$, fix $r>0$ such that $\operatorname{supp}[f]\subset B(\xi,r/C_{2,b})$, fix $x_0\in B(\xi,r/C_{2,b})\cap D$, fix $\psi\in\bar\sF\cap C_c(\tD)$ such that $\psi|_{B(\xi,r)}=0$, and fix non-$\sE$-polar set $K\subset D\setminus B(\xi,r)$. Then,
\begin{align*}
&\quad\ \int_{\partial D}f(x)\check\kappa(dx)=\lim_{n\to\infty}\int_{\partial D}f(x)v_n(x)\check\kappa(dx)=\lim_{n\to \infty}\check\sE(f,v_n|_{\partial D})
=\lim_{n\to \infty}\bar\sE(\sH f,v_n)\\
&=\lim_{n\to\infty}-\bar\sE(\sH^{D\setminus B(\xi,r/C_{2,b})} f,\psi-v_n)\asymp \lim_{n\to\infty}\frac{v_n(x_0)-1}{e_K(x_0)}\bar\sE(\sH^{D\setminus B(\xi,r/C_{2,b})} f,e_K)=0,
\end{align*}
where in the second equality, we use the fact 
\begin{align*}
&\quad\ \lim\limits_{n\to\infty}\int_{\partial D\times\partial D}\big(f(x)-f(y)\big)\big(v_n(x)-v_n(y)\big)\check J(dx,dy)\\
&=\lim\limits_{n\to\infty}2\int_{x\in \partial D\cap B(\xi,r/C_{2,b})}\int_{y\in \partial D\setminus B(\xi,r)}f(x)\big(1-v_n(y)\big)\check J(dx,dy)=0
\end{align*}
by dominated convergence theorem, and we use Lemma \ref{lemma5+3} in the last inequality.

	\smallskip
	
(c) In the case that $\big(\bar \sE ,\bar \sF \big)$ is transient and $\partial D$ is bounded, 
$\bar \bP_x(\sigma_{\partial D}<\infty)<1$ for each $x\in D$.  So $\check\kappa(\partial D)=\check \sE({\mathbbm 1}_{\partial D},{\mathbbm 1}_{\partial D})=\bar \sE (\mathcal{H}{\mathbbm 1}_{\partial D},\mathcal{H}{\mathbbm 1}_{\partial D})={\rm Cap}_0(\partial D)>0$. 

\medskip

Next, we fix non-negative $f\in C_c(\partial D)\cap\check\sF_e$. 
Let $\psi\in C_c(\tD)\cap \bar \sF$ such that $\psi|_{D_{0,C_{2,b}r}}=1$, where $C_{1,b},C_{2,b}$ are the constants of (BHP) and $r=\diam(\partial D)$. Then 
\[
\int_{\partial D}f(x)\check\kappa(dx)=\bar \sE (\mathcal{H}f,\mathcal{H}{\mathbbm 1}_{\partial D})=\bar\sE(\mathcal{H}^{D_r}f,\mathcal{H}{\mathbbm 1}_{\partial D})=-\bar \sE(\mathcal{H}^{D_r}f,\psi-\mathcal{H}{\mathbbm 1}_{\partial D}),
\]
where the last inequality holds by the strongly local property of $(\bar\sE,\bar\sF)$. 

Let $K\subset D_{C_{2,b}r}$ be a compact subset. Then, by Lemma \ref{lemma5+3}, for $x_0\in D_{0,r}$
\[
-\bar\sE (\mathcal{H}^{D_r}f,\psi-\mathcal{H}{\mathbbm 1}_{\partial D})\asymp -\frac{1-\sH\1_{\partial D}(x_0)}{e_K(x_0)}\bar\sE (\mathcal{H}^{D_r}f,e_K).
\]
Noticing that $-\bar\sE(\mathcal{H}^{D_r}f,e_K)=\bar\sE (e_K,e_K)\int_{\partial D}f(x)\omega_K(dx)$ by Lemma \ref{lemma51}(b), we see
\[
\int_{\partial D}f(x)\check\kappa(dx)
\asymp \bar\sE(e_K,e_K)\frac{1-\sH\1_{\partial D}(x_0)}{e_K(x_0)}\int_{\partial D}f(x)\omega_K(dx).
\]  
Since $\omega\asymp \omega_K$,   it follows that 
\[
C_1^{-1}\bar\sE(e_K,e_K)\frac{1-\sH\1_{\partial D}(x_0)}{e_K(x_0)}\leq\frac{\int_{\partial D}f(x)\check \kappa(dx)}{\int_{\partial D}f(x)\omega(dx)}\leq C_1\bar\sE(e_K,e_K)\frac{1-\sH\1_{\partial D}(x_0)}{e_K(x_0)}
\]
for some constant $C_1\in(1,\infty)$ depending only on the bounds in (BHP). Noticing that $\check\kappa(\partial D)={\rm Cap}_0(\partial D)$, 
we conclude that $C_1^{-2}\frac{\check\kappa(\partial D)}{\omega(\partial D)}\leq\frac{\check\kappa(dx)}{\omega(dx)}\leq C_1^2\frac{\check\kappa(\partial D)}{\omega(\partial D)}$.
\end{proof}

\section{Mixed stable-like heat kernel estimates}\label{S:8}

In this short section, we point out that the stable-like heat kernel estimate holds for $(\partial D,d,\omega,\check\sE,\check\sF)$ if we in addition assume that $(\partial D,d)$ is uniformly perfect, which means that there is $C_0 \in (1,\infty)$ such that 
\[ 
\partial D\cap \big(B(x,r)\setminus B(x,r/C_0 )\big)\neq\emptyset\qquad\hbox{ for all }x\in\partial D\hbox{ and }0<r<\diam(\partial D)/C_0.
\]
  It is known that $\omega$ satisfies (RVD) if $(\partial D,d)$ is uniformly perfect and $\omega$ satisfies (VD), see \cite[Exercise 13.1]{Hei}. 
  By increasing the value of $C_0>1$ if needed, we may and do assume that there is some $c_0>0$ so that 
\begin{equation}\label{e:8.1a}
\omega ( B(x, C_0r)\setminus \overline{ B(x, r)}) \geq c_0 \omega (B(x, r))  
\quad \hbox{for every } x\in \partial D \hbox{ and } 0<r < \diam (\partial D)/C_0.
\end{equation}

\begin{lemma}\label{lemma81}
Assume that $(\tD,d,m_0,\bar\sE,\bar\sF)$ satisfies {\rm HK($\Psi$)}, and that $(\partial D,d)$ is uniformly perfect. Let $\omega$ be the renormalized harmonic measure of Theorem \ref{thm5+7}. Suppose that $\omega$ has full support on $\partial D$ and is {\rm (VD)}, and $\Theta_{\Psi, \omega}$ satisfies {\rm(LS)}.
 Recall that $\check{X}_t=\bar{X}_{\tau_t}$ is the time-changed process, where $\tau_t:=\inf\{s\geq 0:A_s^\omega>t\}$ and $A_s^\omega$ is the positive continuous additive functional of $\bar{X}$ with Revuz measure $\omega$.  Then, there are positive constants $C_1$ and $C_2$ so that
\[
C_1\Theta_{\Psi, \omega} (x,r)\leq\check\bE_x[\tau_{B(x,r)}]\leq C_2\Theta_{\Psi, \omega} (x,r)
\]
for each $x\in\partial D$ and $0<r<\diam(\partial D)/C_0$,  where $C_0 >1$ is the constant in \eqref{e:8.1a}. 
\end{lemma}

\begin{proof} 
 Let $\lambda>2$ be the constant of \eqref{e:4.3}, let  $C_\Psi\geq 1$ and $0<\beta_1<\beta_2$ be parameters in \eqref{eqnpsi}, and let $c_1:=C_\Psi^{-1}\lambda^{-\beta_2}$ and $c_2:=(c_1/(2C_\Psi))^{1/\beta_1}$. Then, there exists $c_3>0$ so that for $x_0\in\tD$, $r<\diam(D)$, $c_1\Psi(r)/2<t<c_1\Psi(r)$ and $x,y\in B(x_0,c_2r)$ 
\begin{align*}
\bar p_{B(x_0,r)}(t,x,y)\geq \bar p_{B(x_0,\lambda \Psi^{-1}(t))}(t,x,y)\geq \frac{c_3}{V(x,\Psi^{-1}(t))} \geq \frac{c_3}{V(x,r)}, 
\end{align*}
where the first inequality is due to $c_1\Psi(r)\leq \Psi(r/\lambda)$, and the second inequality is due to $c_2r\leq \Psi^{-1}(c_1\Psi(r)/2)$ and \eqref{e:4.3}. Hence we have for $\  x,y\in B(x_0,c_2r)$,
\[ 
\bar{g}_{B(x_0, r)} (x,y)=\int_0^\infty \bar{p}_{B(x_0,  r)} (t,x,y)dt\geq \int_{c_1\Psi(r)/2}^{c_1\Psi(r)}\bar{p}_{B(x_0,  r)} (t,x,y)dt\geq \frac{c_1c_3\Psi(r)}{2V(x, r)}. 
\]
Denote by $A^\omega$  the positive continuous additive functional of $\bar X$ having Revuz measure $\omega$.
Then  the boundary trace process $\check X_t = \bar X_{\tau_t}$, where $\tau_t:=\inf\{r>0: A^\omega_r >t\}$.
 Note that since $\omega$ is supported on $\partial D$,  $A^\omega_t = \int_0^t \1_{\partial D} (\bar X_s) dA^\omega_s$
 for every $t\geq 0$. Hence  for $x\in\tD$ and $0< r<\diam(D)$, 
\begin{eqnarray*}
\check{\bE}_x[\tau_{ B(x,r)}] &=&  \check{\bE}_x[\sigma_{\partial D\setminus B(x,r)}] = \check{\bE}_x \int_0^{\sigma_{\partial D\setminus B(x,r)} } 
\1_{\partial D} ( \bar X_{\tau_s}) ds   \\
& =&   {\bE}_x \int_0^{\sigma_{\partial D\setminus B(x,r)} }
\1_{\partial D} ( \bar X_r)  dA^\omega_r 
\, = \, \bE_x \big[ A^\omega_{ \sigma_{\partial D\setminus B(x,r)}} \big]\\
&=&  \int_{B(x, r)} \bar{g}_{D\cup B(x, r)} (x,y) \omega (dy)\geq  \int_{B(x, r)} \bar{g}_{B(x, r)} (x,y) \omega (dy) \\
 &\geq & \int_{B(x,c_2r)} \frac{c_1c_3\Psi(r)}{2V(x,r)}\omega (dy)  \geq c_4 \Theta_{\Psi, \omega} (x,r).
  \end{eqnarray*} 
where the third equality is by a change of variable formula \cite[(A.3.16)]{CF}, the fifth equality is by \cite[Propositions 4.1.10 and 4.1.12]{CF}, while  the last inequality is due to (VD) of $\omega$. 

  By the jump kernel estimate of Theorem \ref{thm71}, and L\'evy system equality \cite[(A.3.31)]{CF}, we have
for every $x\in \partial D$ and $0<r<\diam(\partial D)/C_0$, 
\begin{eqnarray*}
1&\geq & \check \bP_x \left(  \check X_{\tau_{B(x, r)}} \in \partial D \setminus \overline{B(x,r)}\right)    \\
  &\geq &  \check \bE_x \int_0^{\tau_{B(x, r)}}  \int_{\partial D \setminus \overline{B(x, r)}} \frac{ c_5} { V_\omega (x,d(x,y))  \Theta_{\Psi, \omega} (x,d\big(x,y)\big)} \omega(dy)  ds \\
 &\geq & \check \bE_x [ \tau_{B(x, r)}] \,   \frac{c_5 \, \omega (B(x, C_0r) \setminus \overline{B(x, r)})  } {V_\omega (x, C_0 r) \Theta_{\Psi,\omega} (x, C_0r)}  
 \geq      \frac{ c_6 \,\check{\bE}_x [ \tau_{B(x, r)}] } { \Theta_{\Psi,\omega} (x,  r)} ,
  \end{eqnarray*}
where the last inequality is due to the  (VD) and (RVD) of $\omega$ and the (LS) of $\Theta_{\Psi,\omega}(x,r)$. 
This proves  that $\check \bE_x  [ \tau_{B(x, r)} ]  \leq   \Theta_{ \Psi,\omega}(x,r)/c_6$. 
\end{proof}

Observe that under (LS) of $\omega$ and (VD) of $m_0$, there is a function $\tilde\Theta_{\Psi,\omega}(x,r)$ that is continuous and increasing in $r$ so that  $\tilde\Theta_{\Psi,\omega}(x,r)\asymp \Theta_{\Psi,\omega}(x,r)$.  Define for $t>0$, 
 $$
 \Theta^{-1}_{\Psi,\omega}(x,t) = \inf\{r>0: \tilde \Theta_{\Psi,\omega}(x,r)> t\}.
 $$
 Clearly, $\tilde\Theta_{\Psi,\omega} (x,\Theta^{-1}_{\Psi,\omega}(x,t))=t$  and so 
   $ \Theta_{\Psi,\omega} (x,\Theta^{-1}_{\Psi,\omega}(x,t)) \asymp t$ for every $t>0$ and $x\in \tD$. 
 
\begin{theorem}\label{T:8.2}
Suppose that $(\tD,d,m_0,\bar\sE,\bar\sF)$ satisfies {\rm HK($\Psi$)}, and that $(\partial D,d)$ is uniformly perfect. Let $\omega$ be the renormalized harmonic measure of Theorem \ref{thm5+7}.  Suppose that $\omega$ has full support on $\partial D$ and is {\rm (VD)}, and $\Theta_{\Psi, \omega}$ satisfies {\rm(LS)}.
 Then the trace Dirichlet form $(\check\sE,\check\sF)$  has jointly continuous heat kernel $\check{p}(t,x,y)$ on $(0, \infty)\times \partial D\times \partial D$ with respect to 
the measure $\omega$, and the following estimates hold.

\begin{enumerate}[\rm (a)]
\item If $(\check \sE, \check  \sF )$ is recurrent or $\partial D$ is unbounded, then  
\[
\check{p}(t,x,y)\asymp \frac{1}{V_\omega (x,\Theta_{\Psi, \omega} ^{-1}(x,t))}\wedge \frac{t}{ V_\omega (x,d(x,y) )\Theta_{\Psi, \omega} (x,d(x,y))}
\]
for all $t>0$ and $x, y \in \partial D$. The constants in $\asymp$ depends only depending only on the parameter $A$ in the uniform domain condition, the parameters in {\rm(VD)}, {\rm HK($\Psi$)} for $(\tD,d,m_0,\bar\sE,\bar\sF)$, and the parameters in  {\rm(VD)} for $\omega$ and  {\rm(LS)}  for 
$\Theta_{\Psi, \omega}$.

\item If $(\check\sE,\check\sF)$ is transient and $\partial D$ is bounded, then there are positive constants $0<c_1<c_2$ and $\lambda_1\geq\lambda_2>0$ depending on the parameter $A$ in the uniform domain condition, the parameters in {\rm(VD)}, {\rm HK($\Psi$)} for $(\tD,d,m_0,\bar\sE,\bar\sF)$, and the parameters in  {\rm(VD)} for $\omega$ and  {\rm(LS)}  for $\Theta_{\Psi, \omega}$ so that 
\begin{eqnarray*}
&& c_1 e^{- \alpha \lambda_1   t} 
 \left( \frac{1}{ V_\omega (x,\Theta_{\Psi, \omega} ^{-1}(x,t)) } \wedge \frac{t}{ V_\omega (x,d(x,y))\Theta_{\Psi, \omega} (x,d(x,y))} \right) \\
 &\leq & 
\check{p}(t,x,y)\, \leq \, c_2 e^{- \alpha\lambda_2    t} 
\left( \frac{1}{ V_\omega (x,\Theta_{\Psi, \omega} ^{-1}(x,t)) }
\wedge \frac{t}{  V_\omega (x,d(x,y)) \Theta_{\Psi, \omega} (x,d(x,y))} \right)
\end{eqnarray*}
for all $t>0$ and $x,y \in \partial D$, where $\alpha:=  \frac{\overline  {\rm Cap}_0(\partial D)}{\omega(\partial D)}$. 
 \end{enumerate}
\end{theorem}

\begin{proof}  
The  proof is along the same line as that of \cite[Theorem 2.40]{KM}. By a quasi-symmetric change of metric as given in \cite[Proposition 5.2 and the proof of Lemma 5.7]{BM},   it suffices to consider the case that $\Theta_{\Psi,\omega} (x,r)\asymp r^\beta$.

 \smallskip
 
 (a). In this case, the desired conclusion follows directly from Theorem \ref{thm71}(a)(b),  Lemma \ref{lemma81}, \cite[Theorem 1.13]{CKW} and Remark \ref{R:8.3}  below.

(b). Suppose that $(\check\sE,\check\sF)$ is transient and $\partial D$ is bounded.
 By Theorem \ref{thm71}(c), 
 there are positive constants $0<c_1<c_2$ and $\lambda_1\geq \lambda_2>0$ 
depending on the parameter $A >1$ in the uniform domain condition, the parameters in {\rm(VD)}, {\rm HK($\Psi$)} for $(\tD,d,m_0,\bar\sE,\bar\sF)$, and the parameters in {\rm(VD)} for $\omega$ and  {\rm(LS)}  for $\Theta_{\Psi, \omega}$  so that 
\begin{equation} \label{e:8.2}
 \alpha  \lambda_2 \leq \check\kappa(x):=\frac{\check\kappa(dx)}{  \omega (dx)}  \leq \alpha \lambda_1 \quad \hbox{on } \partial D.
\end{equation} 
Denote by $(\check\sE^{(j)},\check\sF)$ the Dirichlet form  defined by 
\[
\check\sE^{(j)}(f,g)=\int_{\partial D\times\partial D\setminus{\rm diagonal}}(f(x)-f(y))(g(x)-g(y))\check J(dx,dy)
\quad \hbox{ for }f,g\in\check\sF,
\]
where $\check J(dx,dy)$ is the jump kernel of $(\check\sE,\check\sF)$. 
 By \cite[Theorem 5.1.5]{CF} or \cite[Theorem 6.1.1]{FOT},  $(\check\sE, \check \sF)$ can be obtained from $(\check\sE^{(j)},\check\sF)$ through killing at rate $\check\kappa(x)$ via Feynman-Kac transform. That is, 
\begin{equation} \label{e:8.3}
\check{P}_tf(x)=\check{\bE}^{(j)}_x  \left[e^{-\int_0^t \check \kappa (\check{X}^{(j)}_s)ds}f(\check{X}^{(j)}_{ t}) \right]
\quad\hbox{ for }x\in\partial D,
\end{equation}
where $\check{P}_tf(x):=\bE_x[f(\check{X}_t)]$ and $\check{X}^{(j)}$ is the Hunt process associated with $(\check\sE^{(j)},\check\sF)$.
So for each $ x\in\partial D$ and $0<r<\diam(\partial D)/ C_0$, we have   by Lemma \ref{lemma81} that  
\begin{equation} \label{e:8.4}
\check \bE_x^{(j)}  [ \tau_{B(x, r)} ]\geq \check \bE_x  [ \tau_{B(x, r)} ] \geq C_1\Theta_{\Psi, \omega} (x,r),
 \end{equation}
 while by the same L\'evy system   argument  as in the proof of  Lemma \ref{lemma81}, 
 \begin{equation}\label{e:8.5} 
\check \bE_x^{(j)}  [ \tau_{B(x, r)} ] \leq C_2\Theta_{\Psi, \omega} (x,r). 
 \end{equation}
It follows from  Theorem \ref{thm71}(a),  \eqref{e:8.4}-\eqref{e:8.5}, 
  \cite[Theorem 1.13]{CKW} and Remark \ref{R:8.3} below that 
the Hunt process $\check{X}^{(j)}$ has a jointly continuous   heat kernel $\check{p}^{(j)}(t,x,y)$   satisfying the two sided heat kernel estimates
on $(0, \infty) \times \partial D\times \partial D$: 
\begin{equation}\label{e:8.6}
\check{p}^{(j)}(t,x,y)\asymp \frac{1}{V_\omega (x,\Theta_{\Psi, \omega} ^{-1}(x,t)) }
\wedge \frac{t}{V_\omega (x,d(x,y)) \Theta_{\Psi, \omega} (x,d(x,y))} . 
\end{equation}
  The desired conclusion now follows from \eqref{e:8.2}-\eqref{e:8.3}  and \eqref{e:8.6}. 
\end{proof}

\begin{remark} \label {R:8.3} \rm
We remark here that although it is assumed in \cite{CKW} that the state space  is unbounded,  the results there hold for bounded state spaces as well with some minor modifications  and  also some simplifications. For instance, in the setting of \cite{CKW}, suppose that  $\sX$ is bounded. We do not need to take truncations on the jump size of $X$. Instead,  
 by considering the 1-subprocess of $X$, \cite[Proposition 7.4]{CKW} holds for $(\sE_1, \sF)$ in place of $(\sE, \sF)$. 
  The proof is the same except noting in the proof of (2) $\Rightarrow  {\rm Nash} (\phi)_B$ that 
RVD holds on $\sX$ for  $ r<\diam ( \sX)/C$ for some $C>1$, (2) readily gives ${\rm Nash}(\phi)_B$. 
 This proposition together with 
\cite[Lemma 4.1]{CKW} 
  shows that the jump kernel lower bound condition $ {\bf (J)}_{\phi, \geq}$
 implies the Faber-Krahn inequality ${\bf FK}(\phi)$ for $(\sX, d, m, \sE_1, \sF)$.
 By the same arguments but without taking trucations on the jump size, Lemma 4.18, Theorem 4.23 and Theorem 4.25 in \cite{CKW} 
 hold for  the 1-subprocess of $X$ with $\rho=\diam (\sX)$ there and  condition
  ${\bf FK}(\phi)$ for $(\sX, d, m, \sE_1, \sF)$ in place  of  $(\sX, d, m, \sE, \sF)$.
  In such a way, we get two-sided heat kernel estimates ${\bf HK}(\phi)$ for $X$ for $t\leq 1$. 
  On the other hand, since $\sX$ is bounded, we have $p(t, x, y) \asymp 1$ on $[1, \infty)\times \sX \times \sX$ by exponential ergodicity. Hence   \cite[Theorem 1.13]{CKW} holds for bounded $\sX$, where the exit time condition ${\bf E}_\phi$ should be modified  to hold ``for all  
  $r \in (0, \diam (\sX) )$''
  instead of ``for all $r>0$" in \cite[Definition 1.9]{CKW}.   \qed
\end{remark} 
 
\begin{remark} \label{R:8.4} \rm
While we  were working on this project,  we learned that Kajino and  Murugan \cite{KM} were 
studying heat kernel estimates for the trace of symmetric reflected diffusions on uniform domains. 
The setting of \cite{KM} is slightly more restrictive than ours. 
In \cite{KM},  it is assumed that there is an ambient complete volume doubling strongly local 
 MMD space  $(\sX, d, m, \tilde \sE,  \tilde \sF)$ 
that  enjoys the (VD) property and heat kernel estimates {\bf HK}$(\Psi)$. 
Let $D\subset \sX$ be a uniform domain with respect to the original metric $d$ on $\sX$,
and $(\sE^0, \sF^0)$ be the part Dirichlet form of $(\tilde \sE, \tilde \sF)$ on $D$.
That is, $(\sE^0, \sF^0)$ is the Dirichlet form of the subprocess $X^0$ of the diffusion $\tilde X$ on $\sX$ associated with $(\tilde \sE, \tilde \sF)$ 
killed upon leaving $D$. The reflected Dirichlet form $(\bar \sE, \bar \sF)$ studied in \cite{KM} is the one that is generated by $(\sE^0, \sF^0)$. By   \cite{GSC, Mathav},  (VD) and {\bf HK}$(\Psi)$ hold
 for $(\bar \sE, \bar \sF)$ on $(\overline U, d, m|_{\overline U})$. 
Under these settings and a  stronger condition   \eqref{e:6.4a} than   \eqref{e:6.2},
the  results of Theorem \ref{thm71}  have also been obtained  in \cite{KM}, independently,  by a different method through showing the existence of a Na\"im kernel and deriving  the Doob-Na\"im formula for the trace Dirichlet form $(\check \sE, \check \sF)$;  see Theorem 5.8, Corollary 5.10 and Theorem 5.13 there. As mentioned in Proposition \ref{P:6.3} and Remark \ref{R:6.4}, condition   \eqref{e:6.4a} excludes  the case where $\partial D$ is bounded but the reflected diffusion $\bar X$ on $\overline D$ is transient. It  follows from  Proposition \ref{P:6.3},  Remark \ref{R:6.4} and Theorem \ref{thm71} of this paper, under condition   \eqref{e:6.4a}  the  trace Dirichlet form $(\check \sE, \check \sF)$ admits no killings.
As mentioned in  Remark \ref{R:6.9}, in a recent updated version of \cite{KM}, the authors have given an outline how their arguments can be modified 
to obtain their results under the condition \eqref{e:6.2}. 
 
 \smallskip

 There is also a subtle difference between the viewpoints of this paper and that of \cite{KM}.
 We do not assume a priori  that there is an ambient complete volume doubling strongly local 
 MMD space  $(\sX, d, m, \tilde \sE,  \tilde \sF)$ 
that  enjoys the heat kernel estimates {\bf HK}$(\Psi)$ so that $D$ is a uniform subdomain in $(\sX, d)$.
We start with a minimal diffusion $X^0$ on a metric space $(D, d)$, 
or equivalently, a  minimal strongly local Dirichlet form $(\sE^0, \sF^0)$, 
and then consider its reflected diffusion and trace process. 
The reflected diffusion $\bar X$  is uniquely determined by the minimal diffusion $X^0$.
The information about any ambient diffusion $\tilde X$ beyond $X^0$ is irrelevant to $\bar X$. 
On the other hand,  relevant to the minimal diffusion $X^0$ in $D$ is the topology on $D$, not the actual metric on $D$.
So an advantage of the viewpoints of this paper is that it allows
us to take suitable metric $d$ on $D$ so that under which the reflected diffusion can have the two-sided heat kernel estimates
 {\bf HK}$(\Psi)$.  This is the setting  of the second part of this paper. 
 (In the first part of this paper on restriction and extension theorems, no heat kernel estimate condition {\bf HK}$(\Psi)$ is assumed.)
Such a  point of view is illustrated by several examples in Section \ref{S:9}, including 
Sierpinksi gasket example  in \S \ref{S:9.1},  Sierpinksi carpet example in \S \ref{S:9.2},
 and  inner uniform domains in $\R^d$   which has the slit disc example  in \S \ref{S:9.5} as a particular case. 
 In each of these examples, there is a natural ambient complete volume doubling strongly local 
 MMD space  $(\sX, d, m, \tilde \sE,  \tilde \sF)$ 
that  enjoys the heat kernel estimates {\bf HK}$(\Psi)$ but under which the domain $D$ is a not uniform domain.
Hence  the results from \cite{KM} are not applicable if using this natural ambient complete volume doubling strongly local  MMD space  $(\sX, d, m, \tilde \sE,  \tilde \sF)$. 
However, we can change the original  metric $d$  on $D$ to a new metric $\rho_D$  which still preserves the original topology on $D$.
Under this new metric $\rho_D$, which is the geodesic metric in $D$ in these examples, one can verify that $(D, \rho_D)$ is uniform
and the strongly local  active reflected MMD  space   $(\tD, \rho_D, m_0, \bar \sE, \bar \sF)$ has (VD) and  {\bf HK}$(\Psi)$ property. 
So all the results in Sections \ref{S:3}-\ref{S:8} are applicable to these examples. Of course, one can then view $(\tD, \rho_D, m_0, \bar \sE, \bar \sF)$ as the ambient complete volume doubling strongly local   MMD space  for the minimal diffusion process $X^0$ associated with $(\sE^0, \sF^0)$. In this way, the results from \cite{KM} also become applicable. 
In summary, starting with a minimal diffusion $X^0$ on a metric space $(D, d)$ and then considering its reflected diffusion is more intrinsic, while it is extrinsic to assume a priori that there is an ambient complete volume doubling strongly local MMD space  $(\sX, d, m, \tilde \sE,  \tilde \sF)$ that satisfies the heat kernel estimates {\bf HK}$(\Psi)$ and a uniform domain $D$ of $(\sX, d)$ so that $X^0$ is the part process of the diffusion associated with $(\tilde \sE, \tilde \sF)$ killed upon leaving $D$.
\qed
\end{remark}

\section{Examples}\label{S:9}

\subsection{A subdomain of the Sierpinski gasket}\label{S:9.1}

Let $p_0=(\frac12,\frac{\sqrt{3}}2)$, $p_1=(0,0)$ and $p_2=(1,0)$ be the three vertices of an equilateral triangle, and let $F_i(x)=\frac12x+\frac12p_i$ for $i=0, 1,2,$
 be similarity maps with contraction ratio $1/2$ and fixed points $p_i$. Then the Sierpinski gasket ($SG$ for short) is the unique compact subset of $\R^2$ such that $SG=\bigcup_{i=1}^3F_i(SG)$. Denote by $\overline{p_1,p_2}$ the line segment connecting $p_1,p_2$, and  let $D=SG\setminus \overline{p_1,p_2}$. See figure \ref{fig1} for an illustration. 

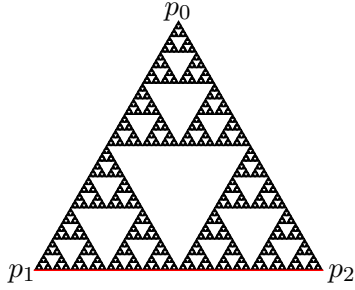
\begin{figure}[htp]
\includegraphics[width=4cm]{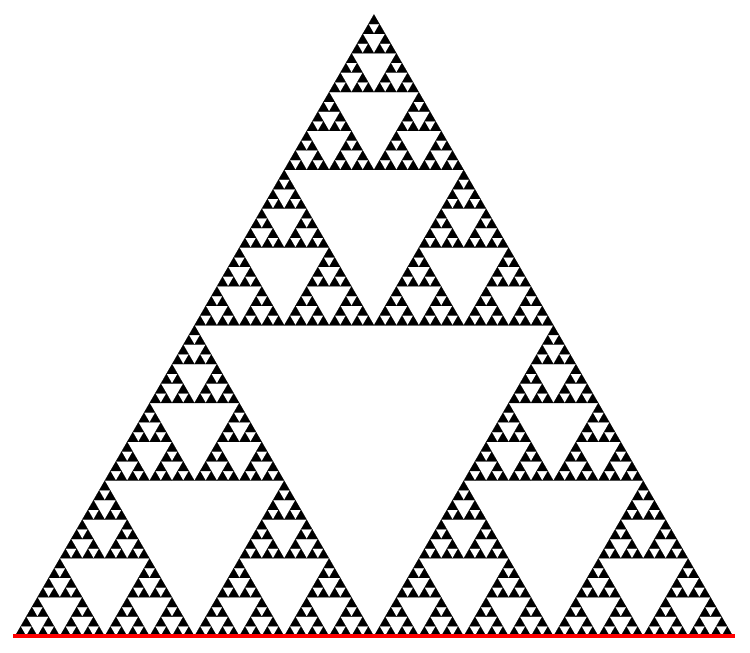}
\begin{picture}(0,0)
\put(-4,0){$p_2$}
\put(-125,0){$p_1$}
\put(-66,99){$p_0$}
\end{picture}
\caption{The Sierpinski gasket $SG$ and the line $\overline{p_1,p_2}$ (colored red)}
\label{fig1}
\end{figure}

\noindent\textbf{(Metrics).}  Denote by  $d(x,y)$  the Eucliden distance on $\R^2$. 
Define $\rho_D$  the geodesic distance in $D$, that is,  
\[\rho_D(x,y)=\inf\{\operatorname{length}(\gamma):\,\gamma\hbox{ is a rectifiable path in $D$ connecting }x,y\}\hbox{ for }x,y\in D,\]
where $\operatorname{length}(\gamma)$ is the length of a continuous rectifiable curve $\gamma$ in $\R^2$ metered with Euclidean distance. Denote by $(\tD,\rho_D)$ the completion of $(D, \rho_D)$, and let $\partial D=\tD\setminus D$.

\medskip

\noindent\textbf{(Description of $(\partial D,\rho_D)$).} We can identify $\partial D$ with the Cantor set $\{1,2\}^\mathbb{N}$: for an infinite word ${\bm\lambda}=\lambda_1\lambda_2\cdots\in \{1,2\}^\mathbb{N}$, define 
\[
\pi({\bm\lambda})=\lim_{n\to\infty}F_{\lambda_1}\circ F_{\lambda_2}\circ\cdots\circ F_{\lambda_n}(q_0),
\]
where the limit is taken in $(\tD,\rho_D)$. Moreover, for ${\bm\lambda}=\lambda_1\lambda_2\cdots,\,{\bm\lambda}'=\lambda'_1\lambda'_2\cdots\in\{1,2\}^\mathbb{N}$, we have
\[
\rho_D(\pi({\bm\lambda}),\pi({\bm\lambda}'))=\frac322^{-k} \qquad \hbox{where } \ 
 k=\min\{i\geq 1:\lambda_i\neq\lambda'_i\}-1.
\] 
It is clear that $(\partial D,\rho_D)$ is uniformly perfect.

\begin{proposition}\label{propSG1}
$(D,d)$ is not a uniform domain.  $(D,\rho_D)$ is a uniform domain, in other words, $D$ is an inner unniform domain in $(SG,d)$. 
\end{proposition}

\begin{proof}
First, we show that $(D,d)$ is not a uniform domain. In fact, for $x_n=F_1\circ F_2^n(p_0)$ and $y_n=F_2\circ F_1^n(p_0)$ with $n\geq 0$, $d(x_n,y_n)=2^{-n-1}$. However, any continuous curve $\gamma$ connecting $x_n$ and $y_n$ in $D$ pass through the  points $F_1(p_0)$ and
$ F_2(p_0)$,  and  so  $\diam(\gamma)\geq 1/2$. Hence $(D, d)$ can not be a uniform domain.

The second statement is a special case of \cite[Proposition 4.1]{L2}.
\end{proof}

\noindent\textbf{(Measures).}  Let $m_H$ be the Hausdorff measure of dimension $\frac{\log3}{\log2}$ on $SG$,  normalized so that $m_H(SG)=1$. It is well-known that $(SG,d,m_H)$ is $\frac{\log3}{\log2}$-Ahlfors regular. 

 Let $m_0$ be the Radon measure on $\tD$ such that $m_0|_D=m_H|_D$ and $m_0(\partial D)=0$.\medskip

\begin{lemma}\label{lemmaSG2}
The metric measure space $(\tD,\rho_D,m_0)$ is $\frac{\log3}{\log2}$-Ahlfors regular.
\end{lemma}

\begin{proof}
Denote by $B(x, r)$ the ball centered at $x$ with radius $r$ under metric $\rho_D$. First, we show that $V(x,r)\gtrsim r^{\log3/\log2}$ for each $x\in \tD$ and $r<1$. We  can find a word $\lambda_1\lambda_2\cdots\lambda_n\in\{0,1,2\}^n$ such that $r/4\leq2^{-n}<r/2$ and $x\in \overline{D\cap F_{\lambda_1}\circ \cdots\circ F_{\lambda_n}(SG)}$. Notice that diameter of $\overline{D\cap F_{\lambda_1}\circ \cdots\circ F_{\lambda_n}(SG)}$ 
under the $\rho_D$  metric is smaller than $2\cdot 2^{-n}< r$.
Thus $\overline{D\cap F_{\lambda_1}\circ \cdots\circ F_{\lambda_n}(SG)}\subset B(x,r)$,  and so
\[
V(x,r)\geq m_H (F_{\lambda_1}\circ \cdots\circ F_{\lambda_n}(SG))=3^{-n}\geq  r^{\log3/\log2}/9.
\]

Next, we show that $V(x,r)\lesssim r^{\log3/\log2}$ for each $x\in \tD$ and $r<1$. Indeed, let $\Pi:(\tD,\rho_D)\to (SG,d)$ be the continuous extension of the identity map $(D,\rho_D)\to (D,d)$. For
every $x\in \tD$ and $r<1$, $\Pi(B(x,r))$ is contained in a ball centered at $\Pi  (x)$ with radius $r$ in $(SG,d)$. Since $m_0(\partial D)=0$, we have $V(x,r)\lesssim r^{\log3/\log2}$.
This proves that $ (\tD,\rho_D,m_0)$ is $\log3/\log2$-regular.
\end{proof}

\noindent\textbf{(Dirichlet forms).} It is well-known that, up to a constant multiple,  there is a unique strongly local regular Dirichlet form $(\sE,\sF)$ on $L^2(SG; m_H )$ such that
$\sF\subset C(SG)$ and the self-similar property holds: 
\begin{align*}
&u\in \sF\ \hbox{ if and only if }\ u\circ F_i\in \sF\ \hbox{ for } i=0,1,2, \\
&\sE(u,u)=\frac53\sum_{i=1}^3\sE(u\circ F_i,u\circ F_i)\quad\hbox{ for each } u\in \sF.
\end{align*}
Moreover, $(\sE,\sF)$ is a resistance form on $SG$ in the sense of \cite[Definition 2.3.1]{Ki}. 

Let $(\bar\sE,\bar\sF)$ be the Dirichlet form of the reflected process on $\tD$, i.e. $\bar\sF=L^2(\tD;m_0)\cap  \sF^\rf$, where $\sF^\rf$ is defined in \eqref{e:rf}.  

\begin{proposition}\label{propSG2}
$\bar\sF= \sF^\rf
\subset C(\tD)$. 
Moreover, $(\bar\sE,\bar\sF)$ is a resistance form with 
$$
\bar{R}(x,y)\asymp  \rho_D(x,y)^{\log (5/3)/\log 2} \quad 
\hbox{for every } x,y\in \tD,
$$
 where $\bar{R}$ is the  corresponding resistance metric, i.e. 
$$
\bar{R}(x,y)=\sup \left\{ \bar{\sE}(f,f)^{-1}:\,f\in \bar\sF,\,f(x)=0,\,f(y)=1 \right\}.
$$ 
\end{proposition}

\begin{proof}
The proposition was proved in \cite[Theorems 4.3 and 4.5]{KTa} using the compatible sequence method. 
For the convenience of the reader, we  present an alternative proof 
starting from  the known fact that $(\bar\mcE,\bar\mcF)$ is the reflected Dirichlet form. 
	
For every $\lambda\in\bigcup_{n=1}^\infty\{1,2\}^n$, define 
\[
K_\lambda :=F_{\lambda_1}\circ F_{\lambda_2}\circ\cdots\circ F_{\lambda_n}\circ F_0(SG)
\quad \hbox{and} \quad 
p_\lambda :=F_{\lambda_1}\circ F_{\lambda_2}\circ\cdots\circ F_{\lambda_n}(p_0),
\]
with the convention that $\{1,2\}^0 :=\{\emptyset\}$, $K_\emptyset :=F_0(SG)$ and $p_\emptyset := p_0$. Then, $\big\{K_\lambda:\lambda\in\bigcup_{n=0}^\infty\{1,2\}^n\big\}$ forms a partition of $D$. 

Let $R(x,y) :=\sup\{ \sE(f,f)^{-1}:\,f\in\sF,\,f(x)=0,\,f(y)=1\}$ for $x,y\in SG$ be the resistance metric associated with $(\sE,\sF)$.  
By self-similarity and the fact $R(x,y)\asymp d(x,y)^{\log (5/3)/\log 2}$,  there is $C_1\in(0,\infty)$ so that
\begin{equation}\label{eqn81}
\big(f(x)-f(y)\big)^2\leq C_1\,\mu_{\<f\>}(K_\lambda)d(x,y)^{\log (5/3)/\log 2}\leq C_1\mu_{\<f\>}(D) 
\rho_D(x,y)^{\log (5/3)/\log 2}
\end{equation}
if $x,y\in K_\lambda$ for some $\lambda\in\bigcup_{n=0}^\infty\{1,2\}^n$. If there is no $\lambda\in\bigcup_{n=0}^\infty\{1,2\}^n$ 
such that $\{x,y\}\subset K_\lambda$, we let $\gamma$ be the continuous path connecting $x,y$ in $D$, and let $q_1,q_2,\cdots,q_n$ be the vertices in $\gamma\cap\{p_\lambda:\,\lambda\in\bigcup_{n=0}^\infty\{1,2\}^n\}$ in the order that $\gamma$ passes them.   Then 
\begin{align}\label{eqn82}
\begin{split}
&\big|f(x)-f(y)\big|\\
\leq& |f(x)-f(q_1)|+\sum_{i=1}^{n-1}|f(q_i)-f(q_{i+1})|+|f(q_n)-f(y)|\\
\leq& \sqrt{C_1\mu_{\<f\>}(D)d(x,q_1)^{\log (5/3)/\log 2}}+\sum_{i=1}^{n-1}\sqrt{C_1\mu_{\<f\>}(D)d(q_i,q_{i+1})^{\log (5/3)/\log 2}} \\
& +\sqrt{C_1\mu_{\<f\>}(D)d(q_n,y)^{ \log (5/3)/\log 2}}\\
\leq& 2  \sqrt{C_1\mu_{\<f\>}(D)\rho_D(x,y)^{\log (5/3)/\log 2}} \, 
\left(1+ (3/5)^{1/2} +(3/5)^{2/2} +(3/5)^{3/2} +\cdots \right),
\end{split}
\end{align}
where the second inequality is due to \eqref{eqn81},  and the last inequality is due to the observation that among the $d(x,q_1)$, $d(y,q_n)$ and $d(q_i,q_{i+1}),1\leq i\leq n-1$, there is at most one with its value in $[\frac12 \rho_D(x,y), \rho_D(x,y)]$, 
 at most two with values in $[\frac14 \rho_D(x,y),\frac12 \rho_D(x,y))$ by the  geometry of $D$, and so on. Combining \eqref{eqn81} and \eqref{eqn82}, we
get   
\begin{equation}\label{eqn75}
\big(f(x)-f(y)\big)^2\leq C_2\mu_{\<f\>}(D) \rho_D(x,y)^{\log (5/3)/\log 2}\hbox{ for every }x,y\in D\hbox{ and }
f\in \sF^\rf.
\end{equation}
Hence, $\sF^\rf \subset C(\tD)$, which implies that $\bar{\sF}=\sF^\rf$. 
Moreover, \eqref{eqn75} holds for any $x,y\in\tD$, and hence 
\[
\bar{R}(x,y)\leq C_2 \rho_D(x,y)^{\log (5/3)/\log 2}. 
\]
(RF4) of \cite[Definition 2.3.1]{Ki} follows immediately,
while   (RF1), (RF2), (RF3) and (RF5) are easy to see to hold.
 So  $(\bar\sE,\bar\sF)$ is a resistance form.

It remains show that $\bar{R}(x,y)\geq C_3 \rho_D(x,y)^{\log (5/3)/\log 2}$ for every $x,y\in\tD$, that is, we need to find $f$ so that $\{f(x),f(y)\}=\{0,1\}$ and $\bar\sE(f,f)\leq C_3^{-1} \rho_D(x,y)^{-\log (5/3)/\log 2}$. We consider two cases as before.
 When $x,y\in K_\lambda$ for some $\lambda$, there is some 
 $f$ on $K_\lambda$ so that $\{f(x),f(y)\}=\{0,1\}$ and $\mu_{\<f\>}(K_\lambda)\asymp d(x,y)^{-\log (5/3)/\log 2}\asymp  \rho_D(x,y)^{-\log (5/3)/\log 2}$.
  We extend it to $\tD$ by taking constant values on each connected component of $\tD\setminus K_\lambda$.
  When  $x,y$ does not belong to a common cell of the form $K_\lambda$, 
  there is  $\lambda\in \bigcup_{n=0}^\infty\{1,2\}^\infty$ such that 
  $\diam(K_\lambda)=\operatorname{length}(\gamma \cap K_\lambda)\geq\rho_D(x,y) /3 $, 
  where $\gamma$ is the geodesic connecting $x,y$. Let $f\in \bar\sF$ be the function that takes values $0$ or $1$ on the two ends of $\gamma\cap K_\lambda$ and has constant values on each component of $\tD\setminus K_\lambda$. 
  Then $\{f(x),f(y)\}=\{0,1\}$ as $x,y$ belong to different components of $\tD\setminus K_\lambda$ that contains an end of $\gamma$. Consequently, 
  $$
  \bar\sE(f,f)=\mu_{\<f\>}(K_\lambda)\asymp\diam(K_\lambda)^{-\log (5/3)/\log 2}
\lesssim  \rho_D(x,y)^{-\log (5/3)/\log 2}.
$$ 
This establishes $\bar{R}(x,y)\geq C_3 \rho_D(x,y)^{\log (5/3)/\log 2}$ for every $x,y\in\tD$.
\end{proof}

\begin{proposition}\label{propSG3}
$(\tD, \rho_D\,m_0,\bar\sE,\bar\sF)$ satisfies {\rm HK$(\Psi)$} with $\Psi(r)=r^{\log5/\log2}$. Moreover, the trace Dirichlet form $(\partial D, \rho_D,\omega_{p_0},\check\sE,\check\sF)$ has a heat kernel that satisfies  the following two sided estimates
\[
\check{p}(t,x,y)\asymp (t\wedge 1)^{-\log 2/\log (10/3)} 
\wedge \frac{t}{ \rho_D(x,y)^{ \log(20/3) /\log2}}
\quad \hbox{for all } x, y\in \partial D \hbox{ and } t>0.
\]
\end{proposition}

\begin{proof}
In view of    \cite[Theorems 15.10 and 15.11]{Ki2}, the  heat kernel estimate of $(\tD, \rho_D,m_0,\bar\sE,\bar\sF)$ is a consequence of Lemma \ref{lemmaSG2} and Proposition \ref{propSG3}.
Then, the heat kernel estimates of $(\partial D, \rho_D,\omega_{p_0},\check\sE,\check\sF)$ follows by Theorem \ref{T:8.2}, 
as  $\omega_{p_0}$ is the $1$-Hausdorff measure on $(\partial D, \rho_D)$
 and so $\Theta_{\omega_{p_0}}(x, r) = r^{\log (10/3)/\log 2}$. 
\end{proof}

\begin{remark}\rm
More generally, we have a class of fractals named $SG_n,n\geq 2$  constructed in the following way. We
 begin with an equilateral triangle with vertices $p_1=(0,0)$, $p_2=(1,0)$ and $p_0=(\frac12,\frac{\sqrt{3}}2)$, and divide it into $n^2$ equilateral triangles of side length $1/n$, where $\frac{n(n+1)}{2}$ of them are upward and $\frac{n(n-1)}2$ of them are downward. Let $F_i,i=1,2,\cdots,\frac{n(n+1)}2$ be orientation preserved similarities that map the oringinal triangle onto an upward small triangle.
Then $SG_n$ is the unique compact subset of $\R^n$ such that $SG_n=\bigcup_{i=1}^{n(n+1)/2}F_i(SG_n)$. See Figure \ref{fig2} for a picture of $SG_3$. 

\begin{figure}[htp]
\includegraphics[width=4cm]{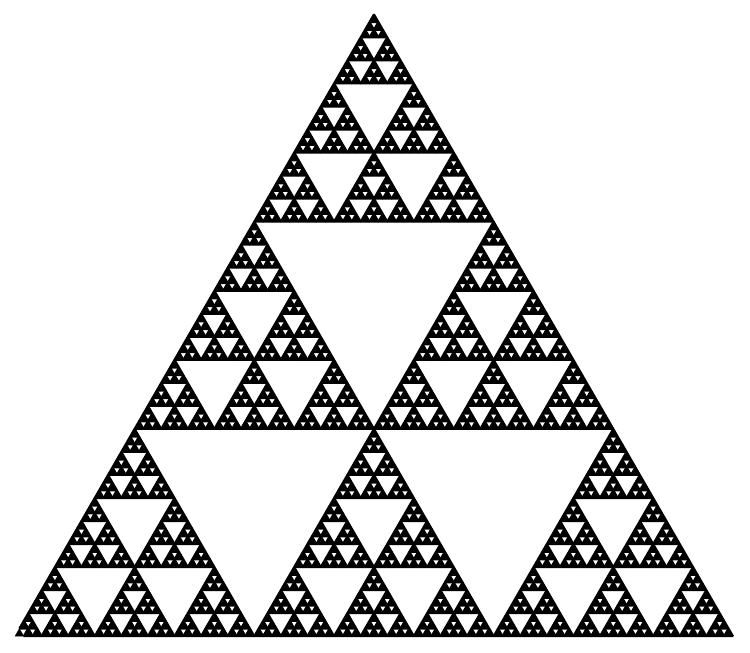}
\caption{A picture of $SG_3$.}\label{fig2}
\end{figure}

On $SG_n$, up to a constant multiple, there is a unique symmetric self-similar resistance form $(\sE,\sF)$. Let $D=SG_n\setminus\overline{p_1,p_2}$, and $ \rho_D$ be geodesic distance in $D$. Then, by the same proof of \cite[Proposition 4.1]{L2},  $(D, \rho_D)$ is a uniform domain, with $\partial D$ being a Cantor set. 

On the $(\tD, \rho_D)$, the sub-Gaussian heat kernel estimates holds by an argument similar to SG case, and Theorem \ref{T:8.2} applies so we have heat kernel estimates for the trace process on $\partial D$.  Moreover, the harmonic measure $\omega_{p_0}$ is a self-similar measure on $\partial D$, and can be computed with the method in \cite[Section 5.2]{CQ}.   \qed
\end{remark}

\subsection{Domains in the Sierpinski carpet} \label{S:9.2}

Let $F_0:=[0,1]^2$ be the unit square and let $\mathcal{Q}_1(F)$ be a collection of eight squares of side length $1/3$ such that their union is $[0,1]^2\setminus (\frac13,\frac23)^2$. For each $Q\in \mathcal{Q}_1(F)$, let $\Psi_Q$ be the orientation preserve affine map such that $\Psi_Q(F_0)=Q$, then the Sierpinski carpet $SC$ is the unique subset of $\R^2$  so that $SC=\bigcup_{Q\in \mathcal{Q}_1(F)}\Psi_Q(SC)$. See Figure \ref{fig3} below for a picture of $SC$.

\begin{figure}
	\centering
	\begin{minipage}{.5\textwidth}
		\centering
		\includegraphics[width=4cm]{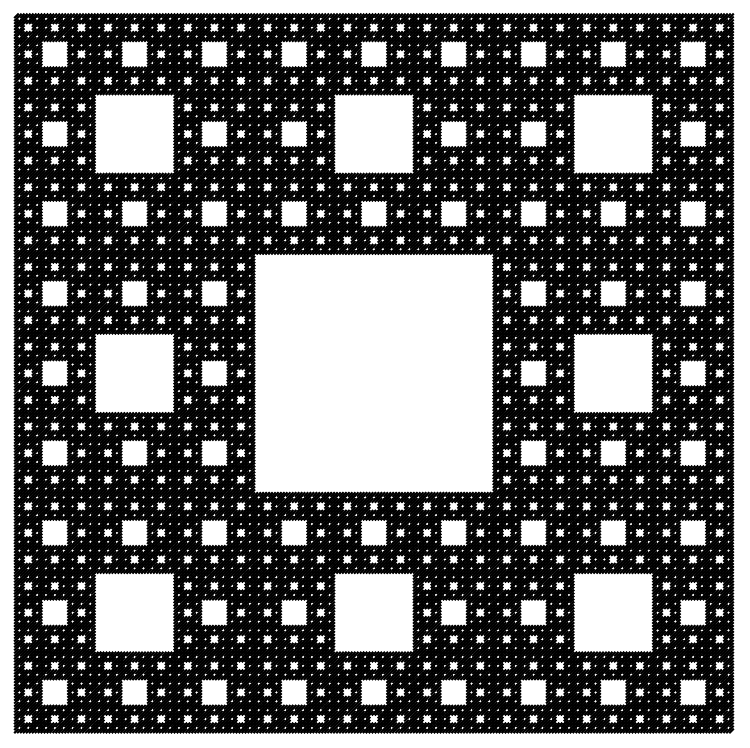}
		\caption{The standard SC}\label{fig3}
	\end{minipage}%
	\begin{minipage}{.5\textwidth}
		\centering
		\includegraphics[width=4cm]{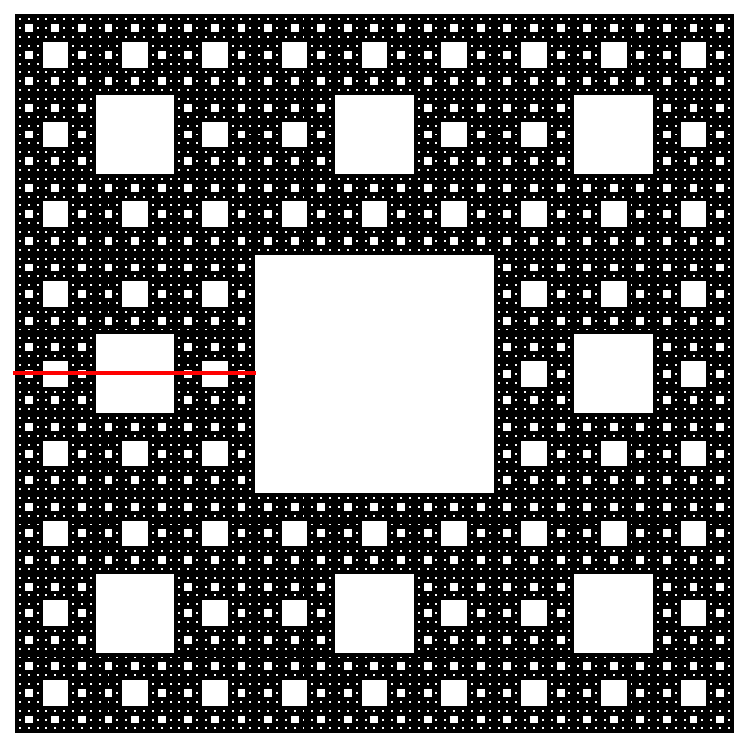}
		\caption{Domain 2}\label{fig4}
	\end{minipage}
\end{figure}

Let $m$ be the Hausdorff measure on $SC$. By the construction of Barlow and Bass \cite{BB1,BB2,BB3,BB4,BB5}, there is a strongly local regular Dirichlet form on $(\sE,\sF)$ on $L^2(SC;m)$ that satisfies HK($r^{d_W}$) for some $d_W\geq 2$. \medskip 
	
\noindent\textbf{Domain 1}.  By \cite[Proposition 4.4]{L2} or \cite[Proposition 2.4]{CQ2}, 
 $D=SC\cap(0,1)^2$ is a uniform domain in $(SC,d)$, where $d$ is Eculidean metric. Thus, $\tD=SC$ and $\partial D=SC\setminus D$. Moreover, HK($r^{d_W}$) holds for $(\tD,d,m_0,\bar\sE,\bar\sF)$ by Theorem \cite[Theorem 2.8]{Mathav}. By Lemma \cite[Lemma 3.7]{CC}, the Hausdorff measure $\sigma$ on $SC$ satisfies {\rm (LS)}. Clearly, $\sigma$ satisfies {\rm (VD)}.  Hence, all the results of Sections 3-7 applies.  \medskip 

\noindent\textbf{Domain 2}. We can check that $D=SC\setminus ([0,\frac13]\times\{\frac12\})$ is not a uniform domain in $(SC,d)$. On the other hand, $D$ is an inner uniform domain in $(SC,d)$: in other words, 
$D$ is a uniform domain under the geodesic distance metric $\rho_D $ in $D$. Let $\tD$ be the  completion of $D$ under metric $ \rho_D$, 
under which each point on  $([0,\frac13]\times\{\frac12\}) \cap SC$ is split into two: the upper and lower points. See Figure \ref{fig4} for an illustration.

 The boundary  $\partial D=\tD\setminus D$ consists of two copies of $SC\cap ([0,\frac13]\times\{\frac12\})$, and is homeomorphic to a Cantor set. By an argument similar to the previous section, we can show that $(\tD, \rho_D,m_0,\bar\sE,\bar\sF)$ satisfies HK$(r^{d_W})$, and all the theorems of Sections 3-7 applies. 

\begin{remark}\rm
In \cite{BB1,BB2,BB3,BB4,BB5}, a more general class of fractals named generalized Sierpinski carpets are introduced.
 We can consider domains similar to Domain 1 and Domain 2 in that setting. 

For example, we can consider $D=SP\setminus [0,\frac13]^2\times\{\frac12\}$, where $SP$ is the Sierpinski sponge generated by deleting small cubes from $[0,1]^3$. See Figure \ref{fig5} below for a picture of the Sierpinski sponge. As before, $D$ is an inner uniform domain, and we can check that HK$(r^{d_W})$ holds for $(\tD, \rho_D,m_0,\bar\sE,\bar\sF)$ for some $d_W>0$. For this, we need to check Poincar\'e inequalities and cutoff Sobolev inequalities.  We only need to take care of balls that intersect $\partial D$, which intersects only the upper half or the lower upper half SP (which are uniform domains of $(SP,d)$).  Hence the desired inequalities hold on these ball by  \cite[Theorem 2.8]{Mathav}. We also check that the Hausdorff measure $\sigma$ on $\partial D$ satisfies (LS) and (VD) by an argument similar to that of \cite[Lemma 3.7]{CC}. Hence, all theorems of Sections 3--7 holds.
\qed
\end{remark}

\begin{figure}
	\centering
	\begin{minipage}{.5\textwidth}
		\centering
		\includegraphics[width=4cm]{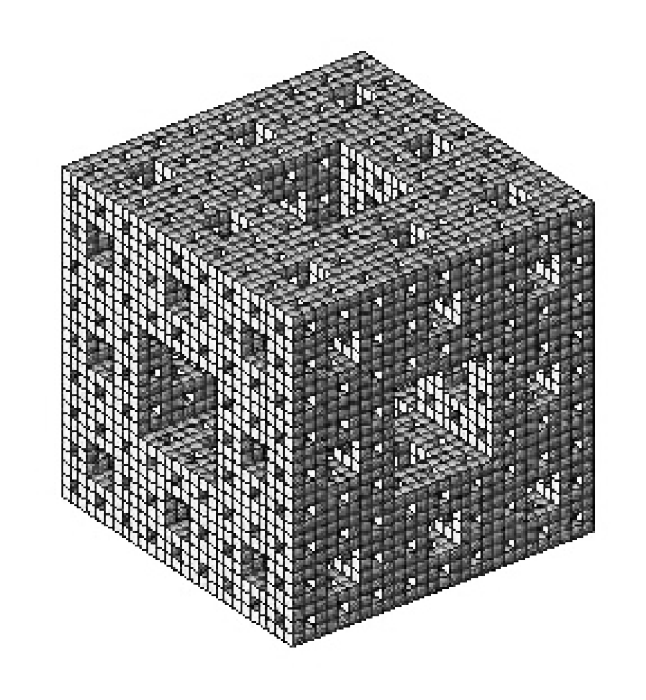}
		\caption{The Sierpinski sponge}\label{fig5}
	\end{minipage}%
	\begin{minipage}{.5\textwidth}
		\centering
		\includegraphics[width=4cm]{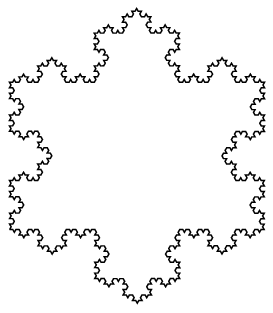}
		\caption{The snowflake domain}\label{fig6}
	\end{minipage}
\end{figure}

\subsection{Snowflake domain  in $\R^2$.} \label{S:9.3} 
Let $D$ be the Koch snowflake domain, which is a bounded uniform domain of $\R^2$.
See Figure \ref{fig6} below for an illustration. 
Consider the Dirichlet form  $(\bar\sE,\bar\sF)=(\sE,W^{1,2}(D))$, with
\[
\sE(f,g):=\frac12\int_D\nabla f(x)  \cdot \nabla g(x)m_0(dx)\quad\hbox{ for }f,g\in W^{1,2}(D),
\]
on $L^2(\tD,m_0)$, where $\tD$ is the closure of $D$ and $m_0$ is the restriction of the Lebesgue measure on $\tD$. 
Then, $(\sE,W^{1,2}(D))$ corresponds to the reflected Brownian motion on $\tD$, and satisfies the Gaussian-heat kernel estimates HK$(r^2)$. The boundary $\partial D$ of $E$ has Hausdorff dimension $d:=\frac{\log4}{\log3}$.
Let $\sigma$ be the $d$-dimensional Hausdorff measure on $\partial D$. 
  It is well known that $\sigma$ is  Ahlfors $d$-regular with $d:=\frac{\log4}{\log3}$ in the sense that 
$\sigma (B(x, r)\cap \partial D) \asymp r^d$ for any $x\in \partial D$ and $0<r<\diam (\partial D)$. Hence  we have
 $\Theta_{\Psi, \sigma} (x,r)\asymp r^{d}$,
 which satisfies condition {\rm (LS)}. Clearly, $\sigma$ is (VD) and $\partial D$ is connected. All the theorems in Sections \ref{S:3}--\ref{S:8} hold. 
 In particular, we have by Proposition \ref{P:3.15} that $\sigma$ is an $\bar \sE$-smooth measure having full $\bar \sE$-quasi-support on $\partial D$.
 The corresponding trace Dirichlet space $(\check \sE, \check \sF)$ on $L^2(\partial D; \sigma)$ have the following characterization by 
 Theorem \ref{T:3.16} as $ \Lambda_{\Psi, \sigma} \cap C (\partial D)$ is a core of $\check \sF$ and $D$ is bounded:
 \begin{align}
 \check \sF &= \Lambda_{\Psi, \sigma} = \left\{f \in L^2(\partial D; \sigma):  \int_{\partial D\times \partial D} \frac{ (f(x)-f(y))^2}{|x-y|^{2d}}
  \sigma (dx) \sigma (dy) <\infty\right\}  , \label{e:9.4}   \\
  \check \sE (f, f) &\asymp\int_{\partial D\times \partial D} \frac{(f(x)-f(y))^2}{|x-y|^{2d}} \sigma (dx) \sigma (dy)
 \quad \hbox{ for } f\in \check \sF_e .    \label{e:9.5} 
   \end{align}
 Recall that  $\check \sF_e:=\sF_e |_{\partial D}$ and $\check \sF = \check \sF_e \cap L^2(\partial D; \sigma)$. 
     By \cite[Theorem 5.2.15]{CF},   $\check \sF_e$  coincides with   the extended Dirichlet space  $(\check \sF)_e$   
     of $(\check \sE, \check \sF)$ on $L^2(\partial D; \sigma)$.      Clearly,
 $$
(\check \sF)_e   \subset   \dot \Lambda_{\Psi, \sigma}  = 
 \left\{ f  \in \sB (\partial D):  \int_{\partial D\times \partial D} \frac{(f(x)-f(y))^2}{|x-y|^{2d}}
  \sigma (dx) \sigma (dy) <\infty \right\} . 
 $$
 Conversely, for each $f\in \dot \Lambda_{\Psi, \sigma} $, $f_n:= ((-n)\vee f )\wedge n \in \check\sF$ converges to $f$ $\sigma$-a.e. on $\partial D$
 and $\{f_n; n\geq 1 \}$ is $\check \sE$-Cauchy by \eqref{e:9.4}-\eqref{e:9.5}. Hence $f\in (\check \sF)_e$. Thus we have
 \begin{equation}\label{e:9.6}
 \check \sF_e= (\check \sF)_e   = 
 \left\{ f  \in \sB (\partial D) :  \int_{\partial D\times \partial D} \frac{(f(x)-f(y))^2}{|x-y|^{2d}}
  \sigma (dx) \sigma (dy) <\infty \right\} . 
\end{equation} 
 Note that each $f\in \check \sF_e$ has an $\check \sE$-quasi-continuous version on $\partial D$.
 We always represent members in  $ \check \sF_e$ by their $\check \sE$-quasi-continuous versions.

 Since $D$ is bounded,  as we see from the first part of the  proof for Theorem \ref{thm5+7}, 
  renormalized harmonic measure $\omega$ can be taken to be a harmonic measure 
 $\omega_x (dz):=  \bar \bP_x (\bar X_{\sigma_{\partial D}} \in dz)$   of $D$ with pole at  $x\in D$. 
  It is known that  the support of  $\omega_x$   has Hausdorff dimension one so it   
 is singular with respect to the Hausdorff measure $\sigma$ on $\partial D$.  While the harmonic measure $\omega_x$ is doubling and 
 so satisfying the (LS) condition by Theorem \ref{T:6.1}, it is less explicit than  $\sigma$ on $\partial D$.
 It seems to be difficult to get an explicit description of the trace Dirichlet form $(\check \sE^\omega, \check \sF^\omega)$  
 directly if we use $\omega$  as the time-change measure. 
 However with the help of   \eqref{e:9.6} deduced using the Hausdorff measure $\sigma$ on $\partial D$, 
  we have by \eqref{e:trace1}-\eqref{e:trace2} that 
 \begin{align*}
 \check \sF^\omega&=\check \sF_e \cap L^2(\partial D; \omega)
  = \left\{f \in L^2(\partial D; \omega):  \int_{\partial D\times \partial D} \frac{ (f(x)-f(y))^2}{|x-y|^{2d}}
  \sigma (dx) \sigma (dy) <\infty\right\}  ,     \\
  \check \sE^\omega  (f, f)&=\check \sE (f, f) \asymp  \int_{\partial D\times \partial D} \frac{(f(x)-f(y))^2}{|x-y|^{2d}} \sigma (dx) \sigma (dy)
 \quad \hbox{ for } f\in \check \sF^\omega .    
   \end{align*}

\subsection{Inner uniform domains in metric measure spaces}\label{S:9.4}

Let $(\sX, d)$ be a locally compact separable  length  (or  geodesic) metric space,
and $m$ a $\sigma$-finite Radon measure with full support
on $\sX$. Let $(\sE, \sF)$ be a strongly local regular Dirichlet form
on $L^2(\sX; m)$  so that  it admits Gaussian-heat kernel estimates ${\bf HK}(r^2)$.
It is well known that this is equivalent to that (VD)  and ${\rm PI}(r^2)$ hold for the strongly local metric measure Dirichlet 
space $(\sX, d, m, \sE, \sF)$. 
Such a space  is called a {\it Harnack-type Dirichlet space} in \cite[Chapter 2]{GSC}.
It is known that there is a conservative Feller process $X$ having strong Feller property associated with
$(\sX, d, m, \sE, \sF)$.

For a domain (i.e.\ connected open subset) $D$   of  the length metric space $(\sX, d)$, define for $x, y\in D$,
\begin{equation}\label{e:rhoD}
\rho_D (x, y)=\inf\{\hbox{length}(\gamma): \hbox{a continuous curve }
\gamma \hbox{ in } D \hbox{ with } \gamma (0)=x \hbox{ and }
\gamma (1)=y\}.
\end{equation}
Note that $(D, d)$ and $(D, \rho_D)$ generate the same topology on $D$. 
Denote by $\tD $ the completion of $D$ under the metric $\rho_D$.
Note that
  $(\tD,\rho_D)$ is a length metric space but may not be
  locally compact in general.
  For example, see \cite[Remark 2.16]{GSC}.
   We extend the definition of $m|_D$ to a measure $m_0$ ${\bar D}$ by setting
$m_0  (A)= m(A\cap D)$ for $A\subset \tD$.

Following \cite[Definition 3.6]{GSC},
we say that an open subset $D\subset \sX$ is  {\it inner uniform}
if there are constants
 $C_1, C_2 \in (0,\infty)$ such
that, for any $x, y \in D$, there exists a continuous curve
$\gamma_{x, y}: [0, 1] \to D$ with
$\gamma_{x, y}(0) = x$, $\gamma _{x, y}(1)=y$
 and satisfying the following two properties:
\begin{enumerate} [\rm (i)]  
\item  The length of
$\gamma_{x,y}$ is at most $ C_1 \rho_D (x, y)$;

\item   For any $z \in \gamma_{x,y}([0, 1])$,
$$
\rho_D  (z, \partial D  ):=\inf_{w \in   \partial D} \rho_D (z, w)   \geq C_2\min\{ \rho_D (x, z), \rho_D (z, y)\}.
$$
\end{enumerate}
In the above definition, (i) can be replaced by ``$\diam (\gamma_{x, y}) \leq   C_1 \rho_D (x, y)$";
see \cite[Proposition 3.3]{GSC} and \cite[Lemma 2.7]{MS}.
Let $D$ be  an inner uniform domain $D$ in $(\sX, d)$ so that $\sX\setminus D$ has positive $\sE$-capacity.
Note that it  is a uniform domain in the sense of 
Definition \ref{D:2.1}  under metric $\rho_D$.  Moreover, $(\tD, \rho_D)$ is a locally compact separable metric space; see \cite[Lemma 3.9]{GSC}.
 Denote by $( \sE^0, \sF^0)$ the Dirichlet form on $L^2(D; m)$ of the subprocess $X^0$ of $X$ killed upon leaving 
 $D$.  Since $D$ is pathwise connected, $(\sE^0, \sF^0)$  is irreducible, transient and strongly local. 
Let $(\bar \sE, \bar \sF)$ be the active reflected Dirichlet form on $L^2 (\tD, m_0)$ generated by  $(\sE^0, \sF^0)$.
It is shown in   \cite[Theorems 3.10 and   3.13 and  Corollary 3.31]{GSC} 
that (VD), ${\rm PI} (r^2)$ and ${\rm HK}(r^2)$ hold for the strongly local
 MMD space   $(\tD, \rho_D, m_0, \bar \sE, \bar \sF)$. 
So all the theorems in Sections \ref{S:3}--\ref{S:8} are applicable for the boundary trace process of 
the reflected diffusion $\bar X$ associated with $(\tD, \rho_D, m_0, \bar \sE, \bar \sF)$.

\subsection{Inner uniform domains in Euclidean spaces and slit disc.} \label{S:9.5} 

Euclidean spaces $(\R^d, d)$ are  length spaces.  Denote by $m$ the Lebesgue measure on $\R^d$.
Let $(\sE, \sF)$ be the Dirichlet form on $L^2(\R^d; m)$ associated with a divergence form elliptic operator 
$\sL={\rm div} (A(x) \nabla)$ with measurable coefficients that is uniformly elliptic and bounded. 
By a celebrated result of Aronson,  $(\R^d, d, m, \sE, \sF)$ is a   Harnack-type Dirichlet space.
So by \S \ref{S:9.4}, all the theorems in Sections \ref{S:3}--\ref{S:8}
are applicable for the boundary trace process of 
the reflected diffusion $\bar X$ on any inner uniform subdomain $D\subset \R^d$ whose  complement is non-polar. 
 In particular, they are applicable to the  planar slit disc   $D=B\setminus \Gamma$,  
where $B$ is the unit ball centered at $(0,0)$ in $\R^2$ and $\Gamma$ is the line segment connecting $(0,0)$ and $(1,0)$.  
See Figure \ref{fig7} below for an illustration. 
\begin{figure}[htp]
\includegraphics[width=4cm]{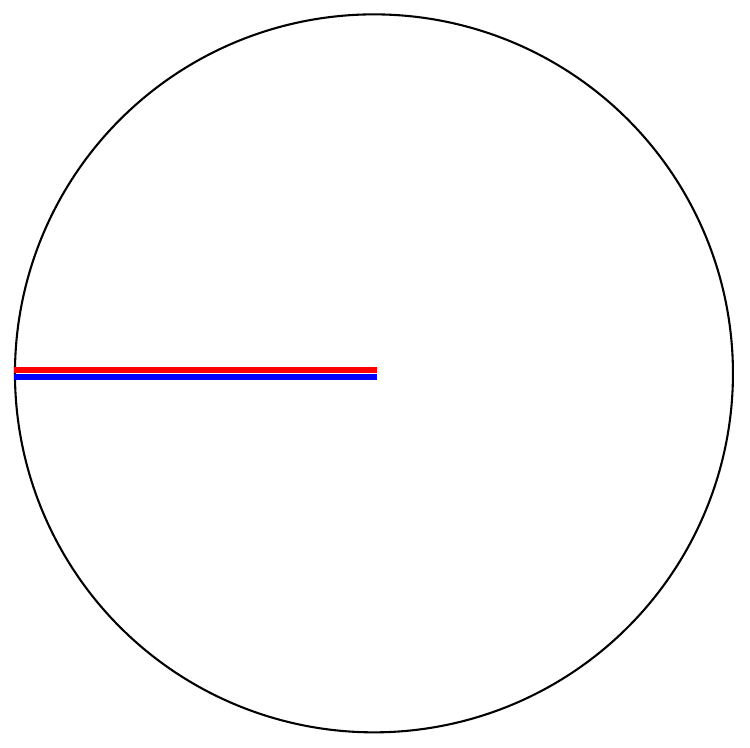}
\caption{Slit disc}\label{fig7}
\end{figure}

The slit domain $D$ is an inner uniform domain in $\R^2$ but not a uniform domain under the standard Euclidean metric $d$. 
Let $\rho_D$ be the geodesic metric in $D$ defined by \eqref{e:rhoD}. 
Under $\rho_D$, $D$ is a uniform domain  and $(D, \rho_D )$ has the same topology as the Euclidean topology in $D$. Denote by  $\tD$  the completion of $D$ with respect to the metric $\rho_D$, for which  the slit $\Gamma$ is splitted into upper and lower line segments. Let $m_0$  be the Radon measure on $\tD$ so that $m_0|_D$ is the Lebesgue measure on $D$ and $m_0(\tD \setminus D )=0$. 

Note that in this case, the active reflected Dirichlet form $(\bar \sE, \bar \sF)=(\bar \sE, W^{1,2}(D))$, where  
\begin{eqnarray*}
W^{1,2}(D)&=& \{ f\in L^2(D; dx): \nabla f\in L^2(D; dx)\}, \\
\bar \sE(f,g)&=& \int_D\nabla f(x) \cdot A(x)\nabla g(x)m_0(dx)\quad\hbox{ for }f,g\in W^{1,2}(D). 
\end{eqnarray*}
It  is a strongly local regular Dirichlet form on $L^2(\tD; m_0)$.
Its associated Hunt process is the symmetric reflecting diffusion $\bar X$ on $\tD$, cf. \cite{C2}. 
As already noted in \S \ref{S:9.4},  $\bar X$ enjoys the two-sided Gaussian heat kernel estimates HK$(r^2)$  with respect to the path metric $\rho_D$.
 Note that $\partial D=\tD \setminus D $  has Hausdorff  dimension $1$ and the Hausdorff measure $\sigma$  on $\partial D:=\tD \setminus D$  satisfies (VD). Consequently $\Theta_{\Psi, \sigma } (x,r)  \asymp  r$ for $r\in (0, \diam (\partial D))$ and $x\in\partial D$,   which satisfies (LS). 
 So 
 all the theorems of Sections  \ref{S:3}--\ref{S:8} apply in this case. In particular, the boundary trace process $\check X$ of
 $\bar X$ on $\partial D$ has heat kernel estimate
$$
\check{p}(t,x,y)\asymp 
   (t\wedge 1)^{-1} 
\wedge \frac{t}{\rho_D  (x,y)^2} \quad \hbox{for all } t>0 \hbox{ and } x, y\in \partial D  .
$$

More generally, let $D$ be an  inner uniform domain $D$ in $\R^n$ with $\partial D:=\tD \setminus D$ being an Ahlfors $d$-regular set for
 $d\in [n-1, n)$.  Here $\tD$ is the completion of $D$ under the metric $\rho_D$ defined by \eqref{e:rhoD}.
 As we already noted in the above, the reflected Dirichlet space $(\tD, \rho_D, m_0, \bar \sE, \bar \sF)$ has (VD) and ${\bf HK}(r^2)$ property.  
  Denote by $\sigma$ the $d$-dimensional Hausdorff measure on $\partial D$ which has the property that 
$V_\sigma (x, r) \asymp r^{d}$  on  $\partial D \times  (0, \diam (\partial D))$. 
Thus  we have  $\Psi (r)= r^2$ and $\Theta_{\Psi, \sigma} (x,r)  \asymp  r^{2+d-n}$ for $r\in (0, \diam (\partial D))$  and $x\in\partial D$,  which satisfies (LS). 
 So  all the theorems of Sections  \ref{S:3}--\ref{S:8} apply in this case. In particular, we have as in \S \ref{S:9.3} that 
 \begin{equation}\label{e:9.8}
 \check \sF_e= (\check \sF)_e   = \dot \Lambda_{\Psi,  \sigma} = 
 \left\{ f  \in \sB (\partial D) :  \int_{\partial D\times \partial D} \frac{(f(x)-f(y))^2}{|x-y|^{2d+2-n}}
  \sigma (dx) \sigma (dy) <\infty \right\} ,
\end{equation} 
and the trace theorems hold. In particular, we have 
and $\check \sE (u, u)\asymp \lb u \rb_{\Lambda_\sigma}^2+\|u\|_{L^2(\partial D;\sigma)}^2$
when $\partial D$ is bounded and $(\bar \sE ,\bar \sF )$ is transient, and 
$\check \sE (u, u)\asymp \lb u \rb_{\Lambda_\sigma}^2$ otherwise. 
This extends the corresponding result  of Jonsson-Wallin \cite{JW}
from uniform domains to inner uniform domains. 

\medskip

\noindent  {\bf Acknowledgements.}
While the authors were working on this project,  we learned Kajino and Murugan \cite{KM} were 
studying heat kernel estimates for the trace of symmetric reflected diffusions on uniform domains. 
Their work is closely related to the topics in the second part of this paper, especially  those in Sections \ref{S:7} and \ref{S:8}.  
Both groups obtained the results in these parts  regarding jump kernel estimates, independently. 
But there are some differences in the setting,  viewpoints and scopes.  
The scopes of this paper are slightly more general. The approaches are also different.  In \cite{KM}, the expression for the jump kernel $\check J(dx, dy)$ of the boundary trace process is derived through showing the existence of Na\"im kernel and establishing the 
Doob-Naim formula for the trace Dirichlet form $(\check \sE, \check \sF)$,
under the (VD) and {\rm HK$(\Psi)$} condition for the ambient complete strongly local MMD spaces of $(D, d, m, \sE^0, \sF^0)$.
We start directly from the definition of the trace Dirichlet form $(\check \sE, \check \sF)$ in \eqref{e:trace1}-\eqref{e:trace2}
and its Beurling-Deny decomposition \eqref{e:2.7} to derive two-sided bounds on $\check J(dx, dy)$
as outlined in the Introduction.
See Remark \ref{R:8.4} for more information. 
The main results of this paper were reported in the joint Probability-Analysis seminar   at the University of Cincinnati on March 20, 2024,
and at the 11th Bielefeld-SNU Joint Workshop in Mathematics held at the Seoul National University from April 8-10, 2024.  
We thank N. Kajino,  M. Murugan, and the anonymous referee for comments on the manuscript. 

The research of Shiping Cao is partially supported by a grant from the Simons Foundation Targeted Grant (917524) to the Pacific Institute for the Mathematical Sciences. 
The research of Zhen-Qing Chen is partially supported by  a Simons Foundation fund 962037.

 \hskip 0.2truein
 
\end{document}